\documentclass[11pt,leqno]{amsart}
\RequirePackage[colorlinks,citecolor=blue,urlcolor=blue]{hyperref}
\usepackage{pkgfile}

\newcommand{\pr}[2]{\langle {#1} , {#2} \rangle}
\newcommand{\norm}[1]{\left \| #1 \right \|}
\newtheorem{assumption}[thm]{Assumption}

\def \etc {,\ldots,}

\def \dist {{\rm dist}}

\newcommand{\Graph}{{\sf G}}
\newcommand{\Adj}{\mathrm{Adj}}
\newcommand{\dGraph}{\overrightarrow{\Graph}}
\newcommand{\bGraph}{{\sf BG}}

\def\corAB{\cyan}

\def \d {\delta}

\begin{document}

\title[{Invertibility of adjacency matrices}]{{Sharp transition of the invertibility of the adjacency matrices of sparse  random graphs}}

\author[A.\ Basak]{Anirban Basak$^*$}
 \author[M.\ Rudelson]{Mark Rudelson$^\dagger$}


\address{$^{*}$International Center for Theoretical Sciences
\newline\indent Tata Institute of Fundamental Research
\newline\indent Bangalore 560089, India
\newline\indent and
\newline\indent Department of Mathematics, Weizmann Institute of Science
\newline\indent POB 26, Rehovot 76100, Israel}

\address{$^\dagger$ Department of Mathematics, University of Michigan
\newline\indent East Hall, 530 Church Street
\newline\indent Ann Arbor, Michigan 48109, USA}

\date{\today}

\subjclass[2010]{46B09, 60B20.}

\keywords{Random matrices, sparse matrices, Erd\H{o}s-R\'{e}nyi graph, invertibility, smallest singular value, condition number}

\maketitle

\begin{abstract}
We consider three models of sparse random graphs:~undirected and directed Erd\H{o}s-R\'{e}nyi graphs and random bipartite graph with two equal parts. For such graphs, we show that if the edge connectivity probability $p$ satisfies $np\ge\log n+k(n)$ with $k(n)\to\infty$ as $n\to\infty$, then the adjacency matrix is invertible with probability approaching one ($n$ is the number of vertices in the two former cases and the same for each part in the latter case). For $np\le\log n-k(n)$ these matrices are invertible with probability approaching zero, as $n\to\infty$. In the intermediate region, when $np=\log n+k(n)$, for a bounded sequence $k(n)\in\R$,  the event $\Omega_0$ that the adjacency matrix has a zero row or a column and its complement both have a non-vanishing probability. For such choices of $p$ our results show that conditioned on the event $\Omega_0^c$ the matrices are again invertible with probability tending to one. This shows that the primary reason for the non-invertibility of such matrices is the existence of a zero row or a column. We further derive a bound on the (modified) condition number of these matrices on $\Omega_0^c$, with a large probability, establishing von Neumann's prediction about the condition number up to a factor of $n^{o(1)}$.
\end{abstract}


{
  \hypersetup{linkcolor=black}
  \tableofcontents
}

\section{Introduction}\label{sec:intro}
For an $n \times n$ real matrix $A_n$ its \emph{singular values} $s_k(A_n), \, k  \in [n]:=\{1,2,\ldots,n\}$ are the eigenvalues of $|A_n| :=\sqrt{A_n^* A_n}$ arranged in a non-increasing order. The maximum and the minimum singular values, often of particular interest, can also be defined as
\[
s_{\max}(A_n):= s_1(A_n) := \sup_{x \in S^{n-1}} \|A_n x \|_2, \qquad s_{\min}(A_n):=s_n(A_n):= \inf_{x \in S^{n-1}} \|A_n x \|_2,
\]
where $S^{n-1}:= \{ x \in \R^n: \|x\|_2 = 1\}$ and $\| \cdot \|_2$ denotes the Euclidean norm of a vector. Further, let \beq\label{eq:Omega-0}
\Omega_0:=\{A_n \text{ has either a zero row or a zero column}\}.
\eeq
Obviously, $s_{\min}(A_n)=0$ for any $A_n \in \Omega_0$.
For matrices with i.i.d.~(independent and identically distributed) Bernoulli entries we establish the following sharp transition of invertibility.
\begin{thm}\label{thm:bernoulli}
Let $A_n$ be an $n \times n$ matrix with i.i.d.~$\dBer(p)$ entries. That is, for $i,j \in [n]$
\[
\P(a_{i,j}=1)=p, \qquad \P(a_{i,j}=0) = 1-p,
\]
with $p \in (0,1/2]$, where $a_{i,j}$ is the $(i,j)$-th entry of $A_n$.
\begin{enumeratei}
\item There exist absolute constants $0 < c_{\ref{thm:bernoulli}}, \ol c_{\ref{thm:bernoulli}}, C_{\ref{thm:bernoulli}} < \infty$ such that for any $\vep >0$, and $p$ such that $np \ge \log(1/p)$, we have
\beq\label{eq:s-min-bd-ber}
\P\left(\left\{s_{\min}(A_n) \le c_{\ref{thm:bernoulli}} \vep \exp\left( -C_{\ref{thm:bernoulli}}\f{\log(1/p)}{\log (np)}\right)\sqrt{\f{p}{n}}\right\}\cap \Omega_0^c\right) \le {\vep + n^{-\ol{c}_{\ref{thm:bernoulli}}}}. 
\eeq
\item There exists an absolute constant $\bar{C}_{\ref{thm:bernoulli}}$ such that for $p$ satisfying $np \le \log(1/p)$ we have
\[
\P(\Omega_0) \ge 1- \f{\bar{C}_{\ref{thm:bernoulli}}}{\log n}.
\]
\end{enumeratei}
\end{thm}

\begin{rmk}
It can be easily verified that the condition $np \ge \log(1/p)$ is almost identical to the condition that $np \ge \log n - \log \log n$. We work with the former as it naturally arises in our proofs. Let us also add that if $p \in (1/2,1)$ then for $n \bar p \ge \log(1/\bar p)$ the conclusion of Theorem \ref{thm:bernoulli} continues to hold, where $\bar p:=1-p$. The proof of this extension easily follows from that of Theorem \ref{thm:bernoulli} by studying the smallest singular value of $\wh A_n - {\bm J}_n$, where $\wh A_n:= {\bm J}_n -  A_n$ and ${\bm J}_n$ is the matrix of all ones. Since our goal in this paper is treat small values of $p$ we avoid this extra step in our arguments, and work only with $p \in (0,1/2]$.
\end{rmk}

\vskip10pt
 Theorem \ref{thm:bernoulli} is a consequence of Theorem \ref{thm:s-min-graphs} which is proved under a more general set-up including, in particular, symmetric Bernoulli matrices. 
{The latter theorem shows that the same phase transition occurs for a broader class of random matrices. To simplify the exposition, we will start with the easier case of  matrices with i.i.d. entries. The results pertaining to other types of random matrices are discussed later in this section.}
 Throughout the paper $p=p_n$ may depend on $n$. For ease of writing we suppress this dependence.

 To understand the implication of Theorem \ref{thm:bernoulli} we see that studying the invertibility property of any given matrix amounts to understanding  the following three different aspects of it. Probably the easiest one is to find the probability that a random matrix is singular. For any given random matrix $A_n$, it means to find a bound on $\gp_n$, where
 \[
\gp_n:=\gp(A_n):=\P(A_n \text{ is non-invertible})= \P(\det(A_n)=0) = \P( s_{\min}(A_n)=0).
 \]
 If the entries of $A_n$ have densities with respect to the Lebesgue measure then $\gp_n=0$. However, for matrices with discrete entries, the problem of evaluating the singularity probability $\gp_n$ is non-trivial.

 The second question regarding the invertibility is more of a quantitative nature. There one is interested in finding the distance between $A_n$ and the set of all singular matrices. As
\[
s_{\min}(A_n) = \inf\{ \|A_n - B\|: \det(B)=0\},
\]
where $\|A_n - B\|$ denotes the operator norm of the matrix $A_n - B$,  a lower bound on $s_{\min}(A_n)$ yields a quantitative measure of invertibility.

The third direction, probably the most difficult, is to find the main reason for non-invertibility of a given random matrix. To elaborate on this let us consider the following well known conjecture:

 \begin{conj}\label{conj:rad}
 Let $A_n$ be a $n \times n$ matrix with i.i.d.~Rademacher random variables ($\pm 1$ with equal probability). Then
 \[
\gp_n= \left(1 +o(1)\right)n^2/2^{n-1},
 \]
 where we recall that the notation $a_n=o(b_n)$ means $\lim_{n \to \infty} a_n/b_n=0$.
 \end{conj}

 \vskip5pt
 It can be noted that the expression above is the probability that there exist either two columns or two rows of $A_n$ which are identical up to a change in sign.
Conjecture \ref{conj:rad} appears in \cite{kks, komlos3, odlyzko}.
 This conjecture, if true, may indicate that the main reason for the singularity for a matrix with i.i.d.~Rademacher entries  is conjectured to be the existence of two identical columns or rows, up to a reversal of sign.

 A few months after  the first posting of this paper on arXiv, a weaker version of Conjecture \ref{conj:rad} asserting that
 \[
  \gp_n= \left(\frac{1}{2} +o(1)\right)^n
 \]
  was proved by Tikhomirov \cite{T}.

Theorem \ref{thm:bernoulli} addresses all three different aspects of invertibility for sparse random matrices. As it yields a lower bound on the smallest singular value it readily gives a quantitative estimate on the invertibility of matrices with i.i.d.~Bernoulli entries. Setting $\vep=0$ in Theorem \ref{thm:bernoulli}(i) we obtain a bound on the singularity probability. 

Probably, the most important feature of Theorem \ref{thm:bernoulli} is that it identifies the existence of a zero row or a column as the primary reason for non-invertibility. To see this, let us denote
\beq\label{eq:Omega-0-col}
\Omega_{\col}:=\{ A_n \text{ has a zero column}\} \qquad \text{ and } \qquad \Omega_{\row}:=\{A_n \text{ has a zero row}\}.
\eeq
As the entries of $A_n$ are i.i.d.~$\dBer(p)$ it is immediate that
\beq\label{eq:omega-col-bd}
\P(\Omega_{\col})= \P(\Omega_{\row}) = 1 - (1 -(1-p)^n)^n.
\eeq
This shows that if $np \ge \log n + k(n)$ then
\[
\P(\Omega_{\col}) \to 0, \quad \text{ as } n \to \infty,
\]
whereas for $np \le \log n - k(n)$ one has
\[
\P(\Omega_{\col}) \to 1, \quad \text{ as } n \to \infty,
\]
for any sequence $k(n) \to \infty$ as $n \to \infty$.
As $np = \log(1/p)$ implies that $np = \log n - \delta_n \log\log n$, for some $\delta_n \sim 1$, from Theorem \ref{thm:bernoulli} we therefore deduce the following corollary.
 \begin{cor}\label{cor:bernoulli}
 Let $A_n$ be a matrix with i.i.d.~$\dBer(p)$ entries. Then we have the following:
 \begin{enumeratea}
 \item If $np =\log n +k(n)$, where $\{k(n)\}_{n \in \N}$ is such that $k(n) \to \infty$ as $n \to \infty$, then
 \[
\P(\Omega_0^c) \to 1 \qquad \text{ and } \qquad \P(A_n \text{ is invertible} ) \to 1 \qquad \text{ as } n \to \infty.
 \]

 \item If $np= \log n - k(n)$ then
 \[
\P(\Omega_0^c) \to 0 \qquad \text{ and } \qquad \P(A_n \text{ is invertible} ) \to 0 \qquad \text{ as } n \to \infty.
 \]

 \item  Moreover, if $np \ge \log(1/p)$ then
 \[
 \P(A_n \text{ is invertible} \mid \Omega_0^c) \to 1 \quad \text{ whenever }
 \P(\Omega_0^c) \cdot { n^{{\ol{c}_{\ref{thm:bernoulli}}}}} \to \infty 
 \quad\text{ as } n \to \infty.
 \]
 %
 \end{enumeratea}
 \end{cor}
 

 \vskip5pt

 Corollary \ref{cor:bernoulli}(a)-(b) shows that the invertibility of a matrix with i.i.d.~Bernoulli entries undergoes a sharp transition essentially at $p =\f{ \log n}{n}$. On the event $\Omega_0$ the matrix is trivially singular. The importance of Corollary \ref{cor:bernoulli}(c) lies in the fact that it shows that even when $\Omega_0$ has a non-trivial probability, on the event $\Omega_0^c$ there is an exceptional set of negligible probability outside which the matrix is again invertible with large probability. This indicates that the main reason for the non-invertibility of a matrix with i.i.d.~Bernoulli entries is the existence of a zero row or a column.
 Moreover, the same phenomenon occurs for two other classes of sparse random matrices with some dependence between the entries symmetric with respect to the diagonal (see Theorem \ref{thm:s-min-graphs} for a precise assumption).
 To the best of our knowledge this is the first instance where the primary reason of the non-invertibility for these three classes of sparse random matrices, in particular the one with i.i.d.~entries, has been rigorously established.

 \begin{rmk}
The reader may also be tempted to interpret that the main reason for the non-invertibility of 
a matrix with i.i.d.~Bernoulli entries 
can only be identified as the presence of a zero row or a column if one shows that
 \beq\label{eq:main-reason-ref}
 \P(\Omega_0 |A_n \text{ is singular}) \to 1 \quad \text{ as } n \to \infty.
 \eeq
{When $\P(\Omega_0)$ is small, \eqref{eq:main-reason-ref} is certainly  stronger than what has been derived in Corollary \ref{cor:bernoulli}(c). However, when $\P(\Omega_0) \to 1$ it can be seen that \eqref{eq:main-reason-ref} is trivial to obtain and it does not provide any information, while Corollary \ref{cor:bernoulli}(c) yields an insight regarding the  reason of invertibility of $A_n$.  As we are concentrating on the phase transition from singularity to invertibility, the formulation of Corollary \ref{cor:bernoulli}(c) seems to be more natural. }

Let us also note that \eqref{eq:main-reason-ref} is false for $p=\frac12$. This is due to the fact that $\P(\wt \Omega_0)/\P(\Omega_0) \to \infty$, as $n \to \infty$, where $$\wt \Omega_0:= \{\text{either two rows or two columns of } A_n \text{ are identical}\}.$$ 
It is further believed that the leading order of of the probabilities of the events $\{A_n \text{ is singular}\}$ and $\wt \Omega_0$ match with each other for $p=\f12$ (see e.g.~\cite[Conjecture 1.2]{hH}). There have been some progress in this direction, see \cite{JSS2}. 
 
{There was a significant very recent progress in determining the main reason for singularity of Bernoulli random matrices with i.i.d. entries, and    \eqref{eq:main-reason-ref}  has been proved  for all $p \in [\log n/n, 1/2)$ in \cite{hH, JSS1, LT}. Establishing  \eqref{eq:main-reason-ref}  for other classes of random matrices including symmetric Bernoulli ones remains an open problem.}

 \begin{rmk}
 It can be seen that for for $p$ such that $np \le \frac{\log n}{2n}$, with probability approaching one, $A_n$ contains two identical columns and on this event $A_n$ is singular. Thus, for such choices of $p$ the bound \eqref{eq:s-min-bd-ber} should not hold. It is possible in that regime the primary reason for invertibility is the existence of two identical rows or columns. 
 \end{rmk}
 
 \end{rmk}

Understanding the singularity probability and the analysis of extremal singular values of random matrices have applications in compressed sensing, geometric functional analysis, theoretical computer science, and many other fields of science. Moreover, to find the limiting spectral distribution of any non-Hermitian random matrix ensemble one essentially needs to go via Girko's Hermitization technique which requires a quantitative lower bound on the smallest singular value. This has spurred a renewed interest in studying the smallest singular value. There have been numerous works in this direction over the last fifteen years. We refer the reader to \cite{BCZ, BD, BR-circular, bcc, C1, LLTTY, RT, RV3, tao_vu, wood}, the survey articles \cite{bordenave_chafai, tao-survey}, and the references therein.

The study of the smallest singular value of a random matrix dates back to 1940's when von Neumann and his collaborators used random matrices to test their algorithm for the inversion of large matrices. They speculated that
\beq\label{eq:von-neumann}
s_{\min}(A_n) \sim n^{-1/2}, \qquad s_{\max}(A_n) \sim n^{1/2}, \qquad \text{ with high probability}
\eeq
(see \cite[pp.~14, 477, 555]{von1} and \cite[Section 7.8]{von2}), where the notation $a_n \sim b_n$ implies that $0 < \liminf_{n \to \infty} a_n/b_n \le \limsup_{n\to\infty} a_n/b_n < \infty$. Therefore, the condition number, which often serves as a measure of stability in matrix algorithms in numerical linear algebra,
\beq\label{eq:cond-von-neumann}
\sigma(A_n):= \frac{s_{\max}(A_n)}{s_{\min}(A_n)} \sim n , \qquad \text{ with high probability}.
\eeq
A more precise formulation of this conjecture can be found in \cite{smale}.

For matrices with i.i.d.~standard normal entries Edelman \cite{edelman} showed that
\beq\label{eq:s-min-gaussian}
\P(s_{\min}(A_n)\le \vep n^{-1/2}) \sim \vep, \quad \text{ for } \vep \in (0,1).
\eeq
On the other hand Slepian's inequality and standard Gaussian concentration inequality for Lipschitz functions (see, e.g. \cite[Corollary 5.35]{V0}) yield that
\beq\label{eq:s-max-gaussian}
\P(s_{\max}(A_n) \ge 2 n^{1/2} + t ) \le \exp(-t^2/2).
\eeq
Therefore combining \eqref{eq:s-min-gaussian}-\eqref{eq:s-max-gaussian} one deduces \eqref{eq:von-neumann}-\eqref{eq:cond-von-neumann} for Gaussian matrices. In \cite{SST} it is shown that \eqref{eq:von-neumann}-\eqref{eq:cond-von-neumann} continues to hold for perturbations of Gaussian matrices that have operator norms bounded by $\sqrt{n}$. The prediction for general matrices remained open for a long time.


A lower bound of order $n^{-3/2}$ on the the smallest singular value of matrices with i.i.d.~centered sub-Gaussian entries was derived in \cite{R}. The optimal order $n^{-1/2}$ was achieved in \cite{RV1} where the assumption on the entries was also relaxed  to the condition that the entries have a finite fourth moment. Under this assumption it was shown that for any $\delta >0$ there exists a $\vep >0$ such that
\beq\label{eq:von-neumann-smin-fourth}
\P(s_{\min}(A_n) \le \vep n^{-1/2}) \le \f{\delta}{2}.
\eeq
Furthermore, from \cite{latala} it follows that for any $\delta >0$ there exists $K$ large enough so that
\[
\P(s_{\max}(A_n) \ge K n^{1/2}) \le \f{\delta}{2}.
\]
Hence, one finds that for such matrices
\beq\label{eq:von-neumann-cond-fourth}
\P(\sigma(A_n) \ge K \vep^{-1} n) \le \delta.
\eeq
It was recently shown in \cite{RT} that the bound for the smallest singular value is valid under even weaker moment assumptions.
Yet, the estimate of the largest one required for the condition number bound may not hold if the fourth moment assumption is relaxed.

Inequality \eqref{eq:von-neumann-cond-fourth} establishes von Neumann's prediction for the condition number for general matrices with i.i.d.~centered entries having finite fourth moments. If the entries are sub-Gaussian the results of \cite{RV1} further show that the probability bounds in \eqref{eq:von-neumann-smin-fourth} and \eqref{eq:von-neumann-cond-fourth} can be improved to $C \vep + c^n$ for some large constant $C$ and $c \in (0,1)$ that depend polynomially on the sub-Gaussian norm of the entries. We emphasize that one cannot obtain a probability estimate similar to \eqref{eq:s-min-gaussian}, as Rademacher random variables are sub-Gaussian and as noted earlier matrices with i.i.d.~Rademacher entries are singular with probability at least $(\f12+o(1))^n$.

As sparse matrices are more abundant in many fields such as statistics, neural network, financial modeling, electrical engineering, wireless communications  (we refer the reader to \cite[Chapter 7]{BS} for further examples, and their relevant references) it is natural to ask if there is an analogue of \eqref{eq:von-neumann}-\eqref{eq:cond-von-neumann} for such matrices. One natural model for sparse random matrices are matrices that are Hadamard products of matrices with i.i.d.~entries having a zero mean and unit variance, and matrices with i.i.d.~$\dBer(p)$ entries, where $p=o(1)$. In \cite{tao-vu-sparse} it was shown that (a similar result appeared in \cite{gotze_tikhomirov}) if $p = \Omega(n^{-\alpha})$, for some $\alpha \in (0,1)$ (the notation $a_n = \Omega(b_n)$ implies that $b_n= O(a_n)$), then for such matrices one has that $s_{\min}(A_n) \ge n^{-C_1}$ with large probability, for some large constant $C_1>0$. In \cite{gotze_tikhomirov} it was further shown that $s_{\max}(A_n) \le n\sqrt{p}$, with probability approaching one, under a minimal assumption on the moments of the entries. This shows that $\sigma(A_n) = O(n^C)$, for some large constant $C$, which is much weaker than the prediction \eqref{eq:cond-von-neumann}.

In \cite{BR-invertibility}, under an optimal moment assumption on the entries, this was improved to show that $\sigma(A_n)$ is indeed $O(n)$ with large probability, whenever $p= \Omega(n^{-\alpha})$ for some $\alpha \in (0,1)$. Results of \cite{BR-invertibility} further show that when the entries of the matrix are products of i.i.d.~sub-Gaussian random variables and $\dBer(p)$ variables then $\sigma(A_n)= O(n^{1+o(1)})$ with large probability, as long as $p \ge C \f{\log n}{n}$, for some large $C$. This also matches with von Neumann's prediction regarding the condition number of a random matrix except for the factor $n^{o(1)}$. As noted earlier in Corollary \ref{cor:bernoulli} when $p$ is near $\f{\log n}{n}$  one starts to see the existence of zero rows and columns, which means that matrix is singular with positive probability, and therefore von Neumann's prediction can no longer hold beyond $\f{\log n}{n}$ barrier.

In this paper our goal is to show that $\f{\log n}{n}$ is the sharp threshold where a general class of random matrices with $0/1$-valued entries undergoes a transition in their invertibility properties. Moreover, for such matrices we show that the existence of a zero row or a zero column is the main reason for the non-invertibility.

A related research direction was pursued by Costello and Vu in \cite{CV1} where they analyzed the rank of $\Adj_n$, the adjacency matrix of an Erd\H{o}s-R\'{e}nyi graph. Later in \cite{CV2} they considered the adjacency matrix of an Erd\H{o}s-R\'{e}nyi graph with general edge weights. The case of the matrix with i.i.d.~$\dBer(p)$ entries was treated in \cite{AE}. In \cite{CV1} it was shown that if $n p > c \log n$, where $p \in (0,1)$ is the edge connectivity probability and $c > \f12$ is some absolute constant, then the co-rank of $\Adj_n$ equals the number of isolated vertices in the graph with probability at least $1 - O((\log \log n)^{-1/4})$ (analogus result for the matrix with i.i.d.~$\dBer(p)$ entries were obtained in \cite[Theorem 2.2]{AE}). This, in particular establishes an analogue of Corollary \ref{cor:bernoulli}(a)-(b) for such matrices. Since \cite{CV1} studies only the rank of such matrices, unlike  Theorem \ref{thm:s-min-graphs} and Corollary \ref{cor:condition-no}, it does not provide any quantitative estimate on the lower bound on $s_{\min}$ and the upper bound on the modified condition number. Let us also add that, from \cite[Theorem 1.2]{CV1} and \cite[Theorem 2.2]{AE} it follows that the same conclusion as in Corollary \ref{cor:bernoulli}(c) holds for $A_n= \Adj_n$ and $A_n$ as in Corollary \ref{cor:bernoulli}, whenever $\P(\Omega_0^c) \cdot (\log \log n)^{1/4} \to \infty$ which is  weaker than the lower bound on $\P(\Omega_0^c)$ required in Corollary \ref{cor:bernoulli}(c).

Before describing the models for the sparse random graphs that we work with in this paper, let us mention the following class of random matrices with $0/1$-valued entries that are closely related. Recently, there have been interests to study properties of the adjacency matrices of directed and undirected $d$-regular random graphs. In the context of the invertibility, it had been conjectured that the adjacency matrices of random $d$-regular ($d \ge 3$) directed and undirected graphs on $n$ vertices are non-singular with probability approaching one, as $n \to \infty$, see \cite{frieze, vu1, vu2}. After a series of partial results \cite{BCZ, C2, C1, LLTTY1, LLTTY2, LLTTY3, LSY} the conjecture has been  recently proved in \cite{H1, H2, M, NW} for both the configuration model and the permutation model.

The adjacency matrix of a random $d$-regular graph and that of an Erd\H{o}s-R\'{e}nyi graph are similar in nature in many aspects. However, in the context of the invertibility property, the latter ceases to be non-singular when the average degree drops below $\log n$, and whereas the former remains invertible even when the degree is bounded. As highlighted in Corollary \ref{cor:bernoulli}(c) (see also Remark \ref{rmk:sharp-transition-gen}) the non-invertibility of the latter is purely due to the existence of a zero row or a column. Since, the former always have $d$ non-zero entries per row and column one does not see the transition in its invertibility property. 


Let us now describe the models of the random graphs.  We begin with the well known notion of undirected Erd\H{o}s-R\'{e}nyi graph. 

\begin{dfn}[Undirected Erd\H{o}s-R\'{e}nyi graphs] \label{def: undir ER}
The undirected Erd\H{o}s-R\'{e}nyi graph $\Graph(n,p)$ is a graph with vertex set $[n]$ such that for every pair of vertices $i$ and $j$ the edge between them is present with probability $p$, independently of everything else. Thus denoting $\Adj(\Graph)$ to be the adjacency matrix of a graph $\Graph$ we see that
\[
\Adj(\Graph(n,p))(i,j)= \left\{\begin{array}{ll}
\delta_{i,j} & \mbox{for } i <j\\
\delta_{j,i} & \mbox{for } i >j\\
0 & \mbox{otherwise}
\end{array}
\right.
,
\]
where $\{\delta_{i,j}\}_{i < j}$ are i.i.d.~$\dBer(p)$ random variables taking one with probability $p$ and zero with probability $(1-p)$.
\end{dfn}

\vskip5pt

Next we describe the model for the directed Erd\H{o}s-R\'{e}nyi graph.

\begin{dfn}[Directed Erd\H{o}s-R\'{e}nyi graphs] \label{def: dir ER}
We define the directed Erd\H{o}s-R\'{e}nyi graph with vertex set $[n]$ as follows: for each pair of vertices $i$ and $j$ the edge between them is drawn with probability $2p$, independently of everything else, and then the direction of the edge is chosen uniformly at random. Such graphs will be denoted by $\overrightarrow{\Graph}(n,p)$. We therefore note that
\[
\Adj(\dGraph(n,p))(i,j)= \left\{\begin{array}{ll}
\tilde \delta_{i,j} \cdot \theta_{i,j} & \mbox{for } i <j\\
\tilde \delta_{j,i} \cdot (1-\theta_{j,i})& \mbox{for } i >j\\
0 & \mbox{otherwise}
\end{array}
\right.
,
\]
where $\{\tilde \delta_{i,j}\}_{i <j}$ are i.i.d.~$\dBer(2p)$ and $\{\theta_{i,j}\}_{i <j}$ are i.i.d.~$\dBer(1/2)$ random variables, and these two collections of random variables are independent of each other.

It is easy to note that $\Adj(\dGraph(n,p))$ has the following representation which will be useful later:
\beq\label{eq:dir-erdos}
\Adj(\dGraph(n,p))(i,j)= \left\{\begin{array}{ll}
\tilde \delta_{i,j} \cdot \theta_{i,j} & \mbox{for } i <j\\
\tilde \delta_{i,j} \cdot (1-\theta_{j,i})& \mbox{for } i >j\\
0 & \mbox{otherwise}
\end{array}
\right.
\eeq
where $\{\theta_{i,j}\}_{i <j}$ are as above and $\{\tilde \delta_{i,j}\}_{i,j=1}^n$ are i.i.d.~$\dBer(2p)$ random variables, and as above these two sets of random variables are independent of each other. This representation yields additional independence which is exploited in our proofs.
\end{dfn}

 Below we define a random bipartite graph.

\begin{dfn}[Random bipartite graphs] \label{def: bip ER}
Fix $m, n \in \N$ and let $\bGraph(m,n,p)$ be a bipartite graph on $[m+n]$ vertices such that for every $i \in [m]$ and $j \in [m+n]\setminus [m]$ the edge between them is present with probability $p$, independently of everything else. Therefore,
\[
\Adj(\bGraph(m,n,p))(i,j)= \left\{\begin{array}{ll}
\delta_{i,j} & \mbox{for } i \in [m], j \in [n+m]\setminus[m]\\
\delta_{j,i} & \mbox{for } i  \in [n+m]\setminus[m],  j \in [m]\\
0 & \mbox{otherwise}
\end{array}
\right.
,
\]
where $\{\delta_{i,j}\}$ are i.i.d.~$\dBer(p)$. When $m=n$, for brevity we write $\bGraph(n,p)$.
\end{dfn}

Now we are ready to describe the main result of this paper. Let us recall the definition of $\Omega_0$ from \eqref{eq:Omega-0}.
{The definitions above give rise to three classes of random matrices. Namely,  the matrix appearing in Definition \ref{def: bip ER} consists of two off-diagonal blocks of fully i.i.d. Bernoulli matrices, and its singular values are the same as for each  of the blocks. Note that such matrices appear in the literature as a linearization of sample covariance-type matrices, see e.g.~\cite{Knowles-Yin}. The adjacency matrix in Definition \ref{def: undir ER} is a symmetric Bernoulli matrix with a zero diagonal, and the matrix in \eqref{eq:dir-erdos} does not match any classical ensemble. }

The next theorem states that on the event that the graph has no isolated vertices,  the same lower bound for the smallest singular value holds  for all three classes.

\begin{thm}\label{thm:s-min-graphs}
Let $A_n= \Adj(\Graph(n,p)), \, \Adj(\dGraph(n,p))$, or $\Adj(\bGraph(n,p))$, and $p \in (0,1/2]$.
\begin{enumeratei}
\item If $np \ge \log(1/p)$ then there exist absolute constants $0 < c_{\ref{thm:s-min-graphs}}, \wt C_{\ref{thm:s-min-graphs}}, C_{\ref{thm:s-min-graphs}} < \infty$ such that for any $\vep >0$, we have
\beq\label{eq:s-min-bd}
\P\left(\left\{s_{\min}(A_n) \le c_{\ref{thm:s-min-graphs}} \vep \exp\left( -\wt C_{\ref{thm:s-min-graphs}}\f{\log(1/p)}{\log (np)}\right)\sqrt{\f{p}{n}}\right\}\cap \Omega_0^c\right) \le \vep^{1/5} + \f{ C_{\ref{thm:s-min-graphs}}}{\sqrt[4]{np}}
\eeq
\item If $np \le \log(1/p)$ then there exists an absolute constant $\bar{C}_{\ref{thm:s-min-graphs}}$ such that
\[
\P(\Omega_0) \ge 1- \f{\bar{C}_{\ref{thm:s-min-graphs}}}{\log n}.
\]
\end{enumeratei}
\end{thm}

\begin{rmk}\label{eq:iid-bipartite}
Note that
\[
\Adj(\bGraph(n,p)):= \begin{bmatrix} {\bm 0}_n & \Adj_{12}(\bGraph(n,p))\\ \Adj_{12}^*(\bGraph(n,p)) & {\bm 0}_n \end{bmatrix},
\]
where ${\bm 0}_n$ is the $n \times n$ matrix of all zeros and $\Adj_{12}(\bGraph(n,p))$ is a matrix with i.i.d.~$\dBer(p)$ entries. Therefore the set of singular values of $\Adj(\bGraph(n,p))$ are same with that of $\Adj_{12}(\bGraph(n,p))$ and each of the singular values of the former has multiplicity two. To simplify the presentation, we will use the $n \times n$ matrix $\Adj_{12}(\bGraph(n,p))$ as the adjacency matrix of a bipartite graph instead of the $(2n) \times (2n)$ matrix $\Adj(\bGraph(n,p))$.
This is the random matrix with  i.i.d.~$\dBer(p)$ entries considered above.
\end{rmk}

\begin{rmk}\label{rmk:sharp-transition-gen}
Theorem \ref{thm:s-min-graphs} implies that Corollary \ref{cor:bernoulli} holds for all three classes of adjacency matrices of random graphs. It means that the phase transition from invertibility to singularity occurs at the same value of $p$ in all three cases, and the main reason for singularity is also the same.


\end{rmk}

\begin{rmk}\label{thm:s-min-general-subgaussian}
To make the presentation simpler, we formulated Theorem \ref{thm:s-min-graphs}  for adjacency matrices of random graphs.
However, a similar result  holds for general sparse random matrices. For example,
a straightforward modification of the proof of Theorem \ref{thm:s-min-graphs} shows that it extends to a symmetric matrix with i.i.d.~Bernoulli entries on and above the diagonal, which is one of the classical ensembles in random matrix theory.
Moreover, it is immediate from the proofs that Theorem \ref{thm:s-min-graphs} extends to the case when the entries are product of a $\dBer(p)$ variable, and a sub-Gaussian random variable (independent with the $\dBer(p)$ variable) with support disjoint from zero. With some additional effort similar statements can be proved for matrices with i.i.d. random entries and for symmetric random matrices whose entries are products of Bernoulli  variables and i.i.d. sub-Gaussian variables, i.e., in the setup similar to \cite{BR-invertibility, wei}. We will not discuss these extensions here to keep the paper to a reasonable length.
\end{rmk}

\begin{rmk}
For $m \sim n$, one can extend the proof of Theorem \ref{thm:s-min-graphs} to derive a quantitative lower bound on $s_{\min}(\Adj(\bGraph(m,n,p)))$,where for rectangular matrices of dimension $m \times n$ we let $s_{\min}(\cdot)=s_{\min\{m, n\}}(\cdot)$. We do not pursue this extension here.
\end{rmk}

Building on Theorem \ref{thm:s-min-graphs} we now proceed to find an upper bound on the condition number.  We point out to the reader that as the entries of $A_n$ have non-zero mean $s_{\max}$ is of larger order of magnitude than the rest of the singular values. For example, it is well known that when $p \sim 1$ the bulk of the singular values is of order $\sqrt{n}$, with one outlier, the top singular value which is of order $n$. Thus, for such matrices to study the stability of inversion algorithms the natural choice would be to analyze the maximum of the ratios of singular values in the bulk.  Therefore, we define the following notion of modified condition number.





 \begin{dfn}\label{def:cond-no}
For any matrix $A_n$ we define its modified condition number as follows:
\[
\wt \sigma(A_n):= \f{s_2(A_n)}{s_{\min}(A_n)}.
\]
\end{dfn}

\vskip5pt
To obtain an upper bound on $\wt \sigma(A_n)$ we need the same for $s_2(A_n)$ which follows from the theorem below.
\begin{thm}\label{thm:s-max-general}
Let $A_n$ be as in Theorem \ref{thm:s-min-graphs}. Fix $c_0 >0, C_0 \ge 1$ and let $p \ge c_0 \f{\log n}{n}$. Then there exists a constant $C_{\ref{thm:s-max-general}}$, depending only on $c_0$ and $C_0$ such that
\[
\P\left(\|A_n -\E A_n\| \ge C_{\ref{thm:s-max-general}} \sqrt{np}\right) \le \exp(-C_0 \log n),
\]
for all large $n$.
\end{thm}

\begin{rmk}\label{rmk:s-2-bd}
If $A_n=\Adj(\Graph(n,p_n))$ or $\Adj(\dGraph(n,p_n))$ we note that $p {\bm J}_n - \E A_n = p I_n$, where ${\bm J}_n$ is the matrix of all ones and $I_n$ is the identity matrix. Therefore, Theorem \ref{thm:s-max-general} immediately implies that $\|A_n - p {\bm J}_n\| = O(\sqrt{np})$ with large probability for such matrices. Since
\[
s_2(A_n) = \inf_{v \in \R^n} \sup_{\substack{x \in S^{n-1}, \\ x \perp v}} \|A_n x \|_2 \, \le \sup_{\substack{x \in S^{n-1}, \\ x \perp {\bm 1} }}
\|A_n x \|_2 \le \sup_{x \in S^{n-1}} \|(A_n - p{\bm J}_n) x \|_2,
\]
it further yields that the same bound continues to hold for the second largest singular value of the adjacency matrices of directed and undirected Erd\H{o}s-R\'{e}nyi graphs. As $\Adj_{12}(\bGraph(n,p))$ is a matrix with i.i.d.~Bernoulli entries we have that $\E \Adj_{12}(\bGraph(n,p)) = p {\bm J}_n$. Therefore, recalling Remark \ref{eq:iid-bipartite} we deduce from Theorem \ref{thm:s-max-general} that $s_2(A_n) = O(\sqrt{np})$, with large probability, when $A_n$ is the adjacency matrix of a random bipartite graph.
\end{rmk}

Remark \ref{rmk:s-2-bd} combined with Theorem \ref{thm:s-min-graphs} yields the following corollary.

\begin{cor}\label{cor:condition-no}
Let $A_n= \Adj(\Graph(n,p)), \, \Adj(\dGraph(n,p))$, or $\Adj(\bGraph(n,p))$ and $p \in (0,1/2]$. If $np \ge \log(1/p)$ then there exist absolute constants $0 < C_{\ref{cor:condition-no}}, \wt C_{\ref{cor:condition-no}}, \bar C_{\ref{cor:condition-no}} < \infty$ such that for any $\vep >0$, we have
\beq\label{eq:s-min-bd}
\P\left(\left\{\wt \sigma(A_n) \ge C_{\ref{cor:condition-no}} \vep^{-1} n^{1+\f{\wt C_{\ref{cor:condition-no}}}{\log \log n}}\right\}\cap \Omega_0^c\right) \le \vep^{1/5} + \f{ \bar C_{\ref{cor:condition-no}}}{\sqrt[4]{np}}.
\eeq
\end{cor}

\vskip10pt


Thus, Corollary \ref{cor:condition-no} shows that
\emph{up to a set of a small probability, we have a dichotomy: either the matrix $A_n$ contains a zero row or zero column, and so $\wt \sigma(A_n) =\infty$, or $\wt \sigma(A_n) $ is roughly of the same order as for the dense random matrix}.

This establishes an analogue of von Neumann's conjecture for the condition number for the entire range of $p$. Let us add that the conclusion of Corollary \ref{cor:condition-no} continues to hold for $\sigma(A_n \odot R_n)$, where $A_n \odot R_n$   is the entry-wise product of (symmetric) matrices with i.i.d.~$\dBer(p)$ entries and Rademacher entries, independent of each other. The proof is a simple adaptation of that of Corollary \ref{cor:condition-no}.


The rest of the paper is organized as follows: In Section \ref{sec:prelim} we provide an outline of the proofs of Theorems \ref{thm:s-min-graphs} and \ref{thm:s-max-general}. In Section \ref{sec:inv-dom-comp} we show that $A_n$ is well invertible over the set of vectors that are close to sparse vectors. We split the set of such vectors into three subsets: vectors that are close very sparse vectors, close to moderately sparse vectors, and those that have a large spread component.
Section \ref{sec:incomp} shows that the matrix in context is well invertible over the set of vectors that are not close to sparse vectors. In Section \ref{sec:proof-main-thm} we first prove Theorem \ref{thm:s-min-graphs}(ii) which essentially follows from Markov's inequality. Then combining the results of Sections \ref{sec:inv-dom-comp}-\ref{sec:incomp} and using Theorem \ref{thm:s-max-general} we prove Theorem \ref{thm:s-min-graphs}(i). The proof of Theorem \ref{thm:s-max-general} can be found in Section \ref{sec:spectral-norm}. Appendix \ref{sec:appendix} contains the proofs of some structural properties of the adjacency matrices of the sparse random graphs that are used to treat very sparse vectors. In Appendix \ref{sec:invert-large-spread} we prove invertibility over vectors that are close to sparse vectors having a large spread component.

\vskip5pt

\noindent
{\bf Acknowledgements.} We thank the anonymous referees for their suggestions that led to an improvement of the presentation of this paper. AB acknowledges support of the Department of Atomic Energy, Government of India (GOI), under project no.~RTI4001. 
Research of AB was partially supported by grant 147/15 from the Israel Science Foundation, a funding from the European Research Council under the European Unions Horizon 2020 research and innovation program (grant agreement number 692452), an Infosys--ICTS Excellence Grant, and a Start-up Research Grant (SRG/2019/001376) and a MATRICS grant (MTR/2019/001105) from Science and Engineering Research Board of GOI. Research of AB is carried out in part  as a member of the Infosys-Chandrasekharan virtual center for Random Geometry, supported by a grant from the Infosys Foundation. Part of this research was performed while MR visited Weizmann Institute of Science in Rehovot, Israel, where he held Rosy and Max Varon Professorship. He is grateful to Weizmann Institute for its hospitality and for creating an excellent work environment.
The research  of MR was supported in part by the NSF grant DMS 1464514 and by a fellowship from the Simons Foundation.

\section{Proof outline}\label{sec:prelim}
In this section we provide outlines of the proofs of Theorems \ref{thm:s-min-graphs} and \ref{thm:s-max-general}.
Broadly, the proof of Theorem \ref{thm:s-max-general} consists of two parts. One of them is to show that $\|A_n - \E A_n\|$ concentrates near its mean. This is a consequence of Talagrand's concentration inequality for convex Lipschitz functions. The second step is to find a bound on $\E \| A_n - \E A_n\|$. This can be derived using \cite{BH}. The proof of Theorem \ref{thm:s-min-graphs}(ii) follows from standard concentration bounds.

The majority of this paper is devoted to the proof Theorem \ref{thm:s-min-graphs}(i), i.e.~to finding a lower bound on the smallest singular value. As we are interested in finding a lower bound on $s_{\min}$ for sparse matrices,  we will assume that $p \le c$ for some absolute constant $c \in (0,1)$ whenever needed during the course of the proof.

We begin by noting that
\[
s_{\min}(A_n)=\inf_{x \in S^{n-1}} \|A_n x\|_2.
\]
To obtain a lower bound on the infimum over the whole sphere we split the sphere into the set of vectors that are close to sparse vectors and its complement. Showing invertibility over these two subsets of the sphere requires two different approaches.

First let us consider the set of vectors that are close to sparse vectors. This set of vectors has a low metric entropy. So, the general scheme would be to show that for any unit vector $x$ that is close to some sparse vector, $\|A_n x \|_2$ cannot be too small with large probability. Then the argument will be completed by taking a union over an appropriate net of the set of such vectors that has a small cardinality.

To obtain an effective probability bound on the event that $\|A_n x\|_2$ is small when $x$ is close to a sparse vector we further need to split the set of such vectors into three subsets: vectors that are close to very sparse vectors, vectors that are close to moderately sparse vectors, and vectors that are close to sparse vectors having a sufficiently large spread component, or equivalently a large non-dominated tail (see Sections \ref{sec:sparse-1}-\ref{sec:sparse-3} for precise formulations).

Unlike the dense set-up, the treatment of very sparse vectors turns out be significantly different for sparse random matrices. It stems from the fact that for such vectors, the small ball probability estimate is too weak to be combined with the union bound over a net. A different method introduced in \cite{BR-invertibility} and subsequently used in \cite{wei}
relies on showing that for any very sparse vector $x$, one can find a large sub-matrix of $A_n$ such that it has exactly one non-zero entry per row. It effectively means that there is no cancellation in $(A_n x)_i$ for a large collection of rows $i \in [n]$. This together with the fact that the set of coordinates of $x$ indexed by the columns of the sub-matrix chosen supports a significant proportion of the norm completes the argument. However, as seen in \cite{BR-invertibility}, this argument works only when $np \ge C \log n$, for some large constant $C$. When, $np \le C \log n$ light columns (i.e.~the columns for which the number of non-zero entries is  much smaller than $np$, see also Definition \ref{def:light-col}) start to appear, with large probability. Hence, the above sub-matrix may not exist.

To overcome this obstacle one requires new ideas. Under the current set-up, we show that given any unit vector $x$, on the event that there is no zero row or column in $A_n$,  the vector $A_n x$ and the coordinates of $x$ that are not included in the set of light columns cannot have a small norm at the same time (see Lemma \ref{lem: normal coordinates}).
This essentially allows us to look for sub-matrices of $A_n$ having one non-zero entry per row, whose columns do not intersect with the set of light columns. In the absence of the light columns one can use Chernoff bound to obtain such a sub-matrix. This route was taken in \cite{BR-invertibility, wei}. However, as explained above, to carry out the same procedure here we need to condition on events involving light columns of $A_n$. So the joint independence of the entries is lost and hence Chernoff bound becomes unusable.

To tackle this issue we derive various structural properties of $A_n$ regarding light and normal (i.e.~not light) columns. Using this we then show that there indeed exists a large sub-matrix of $A_n$ with desired properties, with large probability. We refer the reader to Lemmas \ref{lem: typical structure} and \ref{lem: pattern} for a precise formulation of this step.

Next, we provide an outline of the proof to establish the invertibility over the second and the third sets of sparse vectors. To treat the infimum over such vectors, we first need to obtain small ball probability estimates. This is done by obtaining bounds on the L\'{e}vy concentration function which is defined below.
 \begin{dfn}[L\'{e}vy concentration function]
Let $Z$ be a random variable in $\R^n$. For every $\vep >0$, the L\'{e}vy concentration function of $Z$ is defined as
\[
\cL(Z,\vep):= \sup_{u \in \R^n}\P(\| Z - u \|_2 \le \vep).
\]
\label{dfn:levy}
\end{dfn}

The desired bound on the L\'{e}vy concentration function for the second set of vectors is a consequence of Paley-Zygmund inequality and a standard tensorization argument. Since the third set of vectors has a higher metric entropy than the second, the small ball probability bound derived for the second set of vectors becomes too weak to take a union bound. So using the fact that any vector belonging to the third set has a large spread component, we obtain a better bound on the L\'{e}vy concentration function which is essentially a consequence of the well known Berry-Ess\'{e}en theorem (see Lemma \ref{lem:bound-levy}). Using this improved bound we then carry out an $\vep$-net argument to show that $A_n$ is also well invertible over the third set of sparse vectors. Let us add that the outline of the arguments presented above to deduce invertibility over sparse vectors work for a more general class of matrices than those considered in Theorem \ref{thm:s-min-graphs}, including, in particular, skew-symmetric random matrices. See also Remark \ref{rmk:skew symm}.

Now it remains to provide an outline of the proof of the invertibility over non-sparse vectors. It is well known that such vectors have a large metric entropy, so one cannot use the same argument as above. Instead, using \cite{RV1} we obtain that it is enough to control $\dist(A_{n,1},H_{n,1})$, the distance of $A_{n,1}$, the first column of $A_n$, to $H_{n,1}$, the subspace spanned by the rest of the columns. To control the distance, we derive an expression for it that is more tractable (see Proposition \ref{prop:distance-quadratic}). From Proposition \ref{prop:distance-quadratic}, after some preprocessing, we find that it suffices to show that $\langle C_n^{-1} {\bm x}, {\bm y}\rangle$ is not too small with large probability, where $C_n^{\sf T}$ is the $(n-1) \times (n-1)$ sub-matrix of $A_n$ obtained by deleting its first row and column, and ${\bm x}^{\sf T}$ and ${\bm y}$ are the first row and column of $A_n$ with the first common entry removed, respectively (if $C_n$ is non-invertible, then there is an alternate and simpler lower bound on the relevant distance).

Since Theorem \ref{thm:s-min-graphs} allows ${\bm x}$ and ${\bm y}$ to be dependent a bound on $\cL(\langle C_n^{-1} {\bm x}, {\bm y}\rangle, \vep)$ is not readily available. We use a decoupling argument to show that it is enough to find a bound on the L\'{e}vy concentration function of the random variable $\langle C_n^{-1} \widehat{\bm x}, \widehat{\bm y}\rangle$ for some properly chosen $\widehat{\bm x}$ and $\widehat{\bm y}$, where $\widehat{\bm x}$ and $\widehat{\bm y}$ are now independent. This follows the road-map introduced in \cite{V} for symmetric matrices, although the implementation of it in our case is harder due to the fact that ${\bm x}$ and ${\bm y}$ may be different.
 Having shown this, the desired small ball probability follows once we establish that the random vector $v_\star:=C_n^{-1} \widehat{\bm x}$ has a large spread component. Note that $v_\star$ solves the equation $C_n v =\wh {\bm x}$.
We have already established invertibility of $C_n$  over sparse vectors that has a large spread component.  Now, we extend that argument to show that any solution of the equation $C_n v =\wh{\bm x}$ must also have a large spread component. This allows us to deduce the desired properties of $v_\star$. It completes the outline of the proof of Theorem \ref{thm:s-min-graphs}(i).

 \section{Invertibility over compressible and dominated vectors}\label{sec:inv-dom-comp}

 To prove a uniform lower bound on $\|A_n x \|_2$ for $x$ close to sparse vectors when $A_n$ is the adjacency matrix of one of the three models of the random graphs described in Section \ref{sec:intro}, we will unite them under the following general set-up. It is easy to see that the adjacency matrices of all three models of random graphs satisfy this general assumption.

  \begin{assumption}\label{ass:matrix-entry}
Let $A_n$ be a $n \times n$ matrix with entries $\{a_{i,j}\}$ such that
\begin{enumerate}

\item[(a)] The diagonals $\{a_{i,i}\}_{i=1}^n$ and the off-diagonals $\{a_{i,j}\}_{i \ne j}$ are independent of each other.

\item[(b)] The random variables $\{a_{i,i}\}_{i=1}^n$ are jointly independent and $a_{i,i}\sim \dBer(p_i)$ with $p_i \le p$ for all $i \in [n]$.

\item[(c)] For every $i \ne j \in [n]$, $a_{i,j} \sim \dBer(p)$ and independent of the rest of the entries except possibly $a_{j,i}$.
\end{enumerate}
\end{assumption}

\begin{rmk}\label{rmk:skew symm}
The proofs of the main results of this section extend for matrices with symmetrized Bernoulli entries satisfying the dependency structure of Assumption \ref{ass:matrix-entry}. That is, one can consider the matrix $A_n$ with
\[
\P(a_{i,j} = \pm 1) = \f{p}{2}, \qquad \P(a_{i,j}=0) =1 -p,
\]
and $a_{j,i} = - a_{i,j}$. Note that, this extension in particular includes skew-symmetric matrices. Although skew-symmetric matrices of odd dimension are singular, it shows that they are invertible over sparse vectors.
\end{rmk}

Before proceeding further let us now formally define the notions of vectors that are close to sparse vectors. These definitions are borrowed from \cite{BR-invertibility}.
\begin{dfn}\label{dfn:comp}
Fix $m<n$. The set of $m$-sparse vectors is given by
 \[
  \text{Sparse}(m):=\{ x \in \R^{n} \mid |\text{supp}(x)| \le m \},
 \]
 where $|S|$ denotes the cardinality of a set $S$.
 Furthermore, for any $\delta>0$, the unit vectors which are $\delta$-close to $m$-sparse vectors in the Euclidean norm, are called $(m, \delta)$-compressible vectors. The set of all such vectors hereafter will be denoted by $\text{Comp}(m, \delta)$. Thus,
 \[
   \text{Comp}(m, \delta):= \{x \in S^{n-1} \mid \exists y \in \text{Sparse}(m) \text{ such that } \norm{x-y}_2 \le \delta \},
  \]
  where $\| \cdot \|_2$ denotes the Euclidean norm. The vectors in $S^{n-1}$ which are not compressible, are defined to be incompressible, and the set of all incompressible vectors is denoted as $\text{Incomp}(m,\delta)$.
  \end{dfn}
As already seen in \cite{BR-invertibility, wei} for sparse random matrices one can obtain an effective bound over the subset of the incompressible vectors that have a non-dominated tail. This necessitates the following definition of  dominated vectors. These are also close to sparse vectors, but in a different sense.
 \begin{dfn}\label{dfn:dom}
 For any $x \in S^{n-1}$, let $\pi_x: [n] \to [n]$ be a permutation which arranges the absolute values of the coordinates of $x$ in a non-increasing order. For $1 \le m \le m' \le n$, denote by $x_{[m:m']} \in \R^n$ the vector with coordinates
  \[
    x_{[m:m']}(j)=x_j \cdot \mathbf{1}_{[m:m']}(\pi_x(j)).
  \]
  In other words, we include in $x_{[m:m']}$ the coordinates of $x$ which take places from $m$ to $m'$ in the non-increasing rearrangement.

  \noindent
  For $\alpha<1$ and $m \le n$ define the set of vectors with dominated tail as follows:
  \[
    \text{Dom}(m, \alpha):= \{ x \in S^{n-1} \mid\norm{x_{[m+1:n]}}_2 \le \alpha \sqrt{m} \norm{x_{[m+1:n]}}_{\infty} \}.
  \]
The set of vectors $S^{n-1} \setminus \text{Dom}(m,\alpha)$ will be called non-dominated vectors.
  
Note that by definition, $\text{Sparse}(m) \cap S^{n-1} \subset \text{Dom}(m, \alpha)$, since for $m$-sparse vectors, $x_{[m+1:n]}=0$.
  \end{dfn}

 \subsection{Invertibility over vectors close to very sparse}\label{sec:sparse-1}
 As mentioned in Section \ref{sec:prelim}, the key to control the $\ell_2$ norm of $A_n x$ when $x$ is close to very sparse vectors is to show that $A_n$ has large sub-matrices containing a single non-zero entry per row. This will be then followed by an $\vep$-net argument and the union bound. As we will see a direct application of this idea requires that $\|A_n\|=O(\sqrt{np})$ which does not hold with high probability, because the entries of $A_n$ have a non-zero mean. To overcome this obstacle we use the {\em folding trick} introduced in \cite{BR-invertibility}.

\begin{dfn}[Folded matrices and vectors]
Denote $\gn:=\lfloor n/2\rfloor$. For any $y \in \R^n$ we define
\[
\fold(y):= y_1  - y_2,
\]
where $y_i, \, i=1,2,$ are the vectors in $\R^\gn$ whose entries are the first and the next $\gn$ coordinates of $y$, i.e.~$y_1:= (y^{(1)}, y^{(2)}, \ldots, y^{(\gn)})^{\sf T}$ and $y_2:=(y^{(\gn+1)}, y^{(\gn+2)}, \ldots, y^{(2\gn)})^{\sf T}$, and $\{y^{(i)}\}_{i \in [n]}$ are the entries of $y$. Similarly for a $n \times n$ matrix $B_n$ we define
\[
\fold(B_n):=B_{n,1} - B_{n,2},
\]
where $B_{n,i}, \, i=1,2$ are $\gn \times n$ matrices consisting of the first and the next $\gn$ rows of $B_n$.
\label{dfn:folded}
\end{dfn}

It is easy to see that except a few of entries of $\fold({A}_n)$, the rest have zero mean which allows us to deduce that $\|\fold({A}_n)\|=O(\sqrt{np})$ with large probability. For example, one can use Talagrand's concentration inequality for quasi-convex Lipschitz functions and \cite{BH}. Moreover, using the triangle inequality we see that $\|\fold({A}_n) x\|_2 \le 2 \|A_n x \|_2$. So, we can work with $\fold({A}_n)$ instead of $A_n$. 

To obtain the small ball probability estimate on $\|\fold(A_n) x\|_2$, where $x$ is very close to a sparse vector we need to derive some structural properties of $A_n$.


To this end, we introduce the following notion of light and normal columns and rows.
\begin{dfn}[Light and normal columns and rows]
For a matrix $B_n$ and $i \in [n]$, let us write $\row_i(B_n)$ and $\col_i(B_n)$ to denote the $i$-th row and column of $B_n$ respectively. Let $\delta_0 \in (0,1/10)$ be a fixed constant.
We call $\col_j(B_n), \, j \in [n]$, light if $|\supp(\col_j((B_n))| \le \delta_0 np$. A column which is not light will be called normal. Similar definitions are adopted for the rows. \label{def:light-col}
\end{dfn}
Next denote
\[
\cL(B_n):=\{j \in [n]: \col_j(B_n) \text{ is light}\}.
\]
We are now ready to state the following result on the typical structural properties of $A_n$.

\begin{lem}[Structural properties of $A_n$]\label{lem: typical structure}
 Let $A_n$ satisfy Assumption \ref{ass:matrix-entry} and
 \beq\label{eq:p-ass-mod}
   np \ge \log (1/\bar C p),
 \eeq
 for some $\bar C \ge 1$. Let $\Omega_{\ref{lem: typical structure}}$ be the event such that the following assertions hold:
 \begin{enumerate}
   \item \label{item: no heavy} (No heavy rows and columns)
   For any $j \in [n]$,
   \[
   |\supp(\row_j(A_n))|, |\supp(\col_j(A_n))| \le C_{\ref{lem: typical structure}}np,
   \]
   where $C_{\ref{lem: typical structure}}$ is a large absolute constant.

   \item \label{item: disjoint support} (Light columns have disjoint supports)
   For any $(i,j) \in \binom{[n]}{2}$ such that $\col_i(A_n),\col_j(A_n)$ are light, $\supp(\col_i(A_n)) \cap  \supp(\col_j(A_n)) = \varnothing$.
   \item \label{item: bounded multiplicity} (The number of light columns connected to any given column is bounded) There is an absolute constant $r_0$ such that for any $j \in [n]$, the number of light columns $\col_i(A_n), \, i \in [n]$, with $\supp(\col_i(A_n)) \cap \, \supp(\col_j(A_n)) \neq \varnothing$ does not exceed $r_0$.
   \item \label{item: normal column} (The support of a normal column has a small intersection with the light ones)
       For any $j \in [n]$ such that $\col_j(A_n)$ is normal,
       \[
        \left|\supp(\col_j(\fold(A_n))) \cap  \left( \bigcup_{i \in \cL(A_n)} \supp(\col_i(\fold(A_n))) \right) \right|
        \le \frac{\delta_0}{16} np.
       \]
   \item \label{item: extension} (Extension property of the underlying graph)
 For any $I \subset [n]$ with $2 \le |I| \le c_{\ref{lem: typical structure}} p^{-1}$ 
   \[
    \left| \bigcup_{j \in I} \left(\supp(\col_j(\fold(A_n)))\right) \right|
    \ge \sum_{j \in I} |\supp(\col_j(\fold(A_n)))| - \frac{\d_0}{16} np |I|,
   \]
   where 
   $c_{\ref{lem: typical structure}}$ is a constant depending only on $\d_0$.

   \item \label{item:diff} (supports of columns of the matrix and its folded version are close in size) For every $j \in [n]$,
   \[
   \big| |\supp(\col_j(A_n))| - |\supp(\col_j(\fold(A_n)))|  \big| \le \f{\d_0}{8} np.
   \]
 \end{enumerate}
 Then there exists $n_0$, depending only $\bar C$ and $\delta_0$, such that for any $n \ge n_0$ the event $\Omega_{\ref{lem: typical structure}}$ occurs with probability at least $1-n^{-\bar c_{\ref{lem: typical structure}}}$ for some $\bar c_{\ref{lem: typical structure}}>0$ depending only on $\delta_0$.
\end{lem}

The proof of Lemma \ref{lem: typical structure} relies on standard tools such as Chernoff bound, and Markov inequality. Its proof is deferred to Appendix \ref{sec:appendix}.

\begin{rmk}
From the proof of Lemma \ref{lem: typical structure} it follows that one can take $r_0 = 19$. The last property of the event $\Omega_{\ref{lem: typical structure}}$ holds when $p$ is assumed to be sufficiently small (possibly depending on $\delta_0$). 
\end{rmk}

\begin{rmk}\label{rmk:typical structure-rectangle}
As we will see in Section \ref{sec:incomp} (also mentioned in Section \ref{sec:prelim}), to establish the invertibility over incompressible and non-dominated vectors for the adjacency matrices of undirected and directed Erd\H{o}s-R\'{e}nyi graphs, one needs to find a uniform lower bound on $\|\fold(A_n) x - y_0\|_2$ over compressible and dominated vectors $x$ and some fixed $y_0 \in \R^n$ with $|\supp(y_0)| \le C_\star np$ for some $C_\star >0$.
While showing invertibility over vectors that are close to very sparse vectors, we tackle this additional difficulty by deleting the rows from $A_n$ that are in $\supp(y_0)$. This requires proving an analog of Lemma \ref{lem: typical structure} for rectangular sub-matrix $\bar A_n$ of dimension $\bar n \times n$, where $n - C_\star np \le \bar n \le n$. This means that to apply Lemma \ref{lem: typical structure} for the original matrix $A_n$ we need to prove it under the assumption \eqref{eq:p-ass-mod} rather than the assumption $np \ge \log(1/p)$. To keep the presentation of this paper simpler we prove Lemma \ref{lem: typical structure} only for square matrices. Upon investigating the proof it becomes clear that the extension to rectangular, almost square, matrices requires only minor changes.
\end{rmk}

\vskip10pt
Next we define the following notion of a good event needed to establish the small ball probability estimates on $\fold(A_n) x - \fold(y_0)$ for $x$ close to very sparse vectors and some fixed vector $y_0 \in \R^n$.

 \begin{dfn}[Good event]\label{dfn:good-event}
 Let $A_n$ satisfy Assumption \ref{ass:matrix-entry}. Fix $\kappa \in \N$, $J, J' \subset [n]$ disjoint sets. Denote
\[
\bar J :=\bar J(J):= \{ j \in [\gn]: j \in J \text{ or } j + \gn \in J\},
\]
and similarly $\bar J'$. 
For any $c>0$, define
\(
\cA_{c, \kappa}^{J,J'}
\)
to be the event that there exists $I \subset[\gn]\setminus (\bar J \cup \bar J')$ with $|I| \ge c \kappa np$ such for every $i \in I$ there further exists $j_i \in J$ so that
\[
|\ga_{i,j_i}|=  1, \quad \ga_{i,j} =0 \text{ for all } j \in  (J \setminus \{j_i\}) \cup J',
\]
where $\{\ga_{i,j}\}$ are the entries of $\fold(A_n)$, and
\beq\label{eq:row-does-not-sect}
\supp(\row_i(\fold(A_n))) \cap \cL(A_n)  = \emptyset,  \quad \text{ for all } i \in I.
\eeq
\end{dfn}

\begin{rmk}
In Definition \ref{dfn:good-event} above we needed to define $\bar J$ and $\bar J'$ because we work with $\fold(A_n)$. Since the entry $a_{i,j}$ may depend on $a_{j,i}$ we further require the set $I\subset [\gn]$ to be disjoint from $\bar J \cup \bar J'$. To treat matrices with i.i.d.~entries these modifications are not needed.
\end{rmk}

Now we are ready to state the structural lemma that shows that the good event $\cA_c^{J,J'}$ holds with high probability for appropriate sizes of $J$ and $J'$.

  \begin{lem}  \label{lem: pattern}
 Let $A_n$ satisfy Assumption \ref{ass:matrix-entry} and $ np \ge  \log (1/\bar C p)$ for some $\bar C \ge 1$.
Then, there exist an absolute constant $\bar c_{\ref{lem: pattern}}$, and constants ${c}_{\ref{lem: pattern}}, c^\star_{\ref{lem: pattern}}$, depending only on $\delta_0$, such that
\beq\label{eq: pattern}
\P\left( \bigcup_{\kappa \le c^\star_{\ref{lem: pattern}} (p \sqrt{pn})^{-1} \vee 1 } \ \bigcup_{\substack{J \in \binom{[n]}{\kappa} \\ J \cap \cL(A_n) = \varnothing}} \ \bigcup_{J' \in \binom{[n]}{\gm}, \, J \cap J'=\emptyset} \left({\mathcal A}^{J, J'}_{{c}_{\ref{lem: pattern}}, \kappa} \right)^c \cap  \Omega_{\ref{lem: typical structure}} \right)
\le n^{-\bar c_{\ref{lem: pattern}}},
\eeq
for all large $n$, where for $\kappa, m \in \N$ we write $\binom{[n]}{m}:= \{\wt J \subset [n]: |\wt J|=m\}$ and 
 \[
   \gm=\gm(\kappa):= \left(\kappa \sqrt{pn}\right) \wedge \left(\f{c^\star_{\ref{lem: pattern}}}{ p}\right)
  \]
 \end{lem}

\begin{rmk}\label{rmk:diff-with-br}
We point out to the reader that \cite[Lemma 3.2]{BR-invertibility} derives a result similar to Lemma \ref{lem: pattern}. The key difference is that the former assumes $np \ge C \log n$, for some large constant $C$, which allows  to use Chernoff bound to conclude that given any set of columns $J \subset [n]$ of appropriate size, there is a large number of rows for which there exists exactly one non-zero entry per row in the columns indexed by $J$. When $np \le C \log n$ this simply does not hold for all $J \subset [n]$ as there are light columns. Moreover, for such choices of $p$ the Chernoff bound is too weak to yield any non-trivial bound on the number of rows with the desired property. Therefore we need to use several structural properties of our matrix $A_n$, derived in Lemma \ref{lem: typical structure}, to obtain a useful lower bound on the number of such rows.
\end{rmk}

 \begin{proof}[Proof of Lemma \ref{lem: pattern}]


 Fixing $\kappa \le c^\star_{\ref{lem: pattern}}(p \sqrt{pn})^{-1}$, for some constant $c^\star_{\ref{lem: pattern}}$ to be determined during the course of the proof, we let $J \in \binom{[n]}{\kappa}$. Let $I^{1}(J)$ be the set of all rows of $\fold(A_n)$ containing exactly one non-zero entry in the columns corresponding to $J$. More precisely,
 \[
 I^{1}(J):= \Big\{i \in [\gn]: \ga_{i,j_i}\ne 0 \text{ for some } j_i \in J, \text { and } \ga_{i,j}=0 \text{ for all } j \in J \setminus \{j_i\}\Big\}.\]
Similarly for a set $J' \in \binom{[n]}{\gm}$ we define
\[I^0(J'):= \Big\{ i \in [\gn]\setminus (\bar J \cup \bar J'): \ga_{i,j}=0 \text{ for all } j \in J'\Big\}. \]
Note that we have deleted the rows in $\bar J \cup \bar J'$ while defining $I^0(J')$. This is due to the fact that matrices satisfying Assumption \ref{ass:matrix-entry} allow some dependencies among its entries. Later, in the proof we will require $I^1(J)$ and $I^0(J')$ to be independent for disjoint $J$ and $J'$.

 To estimate $|I^{1}(J)|$ we let $\sE:=\left(\cup_{j \in J} \supp(\col_j(\fold(A_n)))\right)$ and define a function $f: \sE \to \N$ by
 \[
  f(i): = \sum_{j \in J} \bI\{i \in {\supp(\col_j(\fold(A_n)))}\}, \quad i \in \sE.
 \]
 We note that $I^{1}(J)=\{i \in \sE : \ f(i)=1 \}$. Hence,
 \beq\label{eq:I^1-lbd}
  |I^{1}(J)| \ge 2|\sE|-\sum_{i \in \sE} f(i)
  =\sum_{i \in \sE} f(i) - 2 \left( \sum_{i \in \sE} f(i)-|\sE| \right).
 \eeq
If $J \cap \cL(A_n) = \varnothing$ (recall that $\cL(A_n)$ is the set of light columns of $A_n$), then by property \eqref{item:diff} of the event $\Omega_{\ref{lem: typical structure}}$ we have
$$\sum_{i \in \sE} f(i)=\sum_{j \in J} |\supp(\col_j(\fold(A_n))) | \ge \sum_{j \in J} |\supp(\col_j(A_n))| - \f{\d_0}{8}np |J|\ge \f{7 \d_0}{8} np|J|.$$
Thus, by  property \eqref{item: extension} of the event $\Omega_{\ref{lem: typical structure}}$, it follows that
   \[
   \sum_{i \in \sE} f(i) - |\sE| =  \sum_{j \in J} |\supp(\col_j(A_n))| - |\sE|
    \le  \frac{\d_0}{16} np |J|.
   \]
   Therefore, from \eqref{eq:I^1-lbd} we deduce that
   \begin{equation}\label{eq: I^1(J)}
     |I^1(J)| \ge \f{\d_0}{2} np |J|
   \end{equation}
on the event $\Omega_{\ref{lem: typical structure}}$  for any $J \subset [n]$ such that $J \cap \cL(A_n)=\emptyset$.

Using the above lower bound on the cardinality of $I^1(J)$ we now show that it has a large intersection with $I^0(J')$. Therefore we can set the desired collection of rows to be the intersection of $I^1(J)$ and $I^0(J')$. However, the caveat with this approach is that the collection of rows just described does not satisfy the property \eqref{eq:row-does-not-sect}. To take care of this obstacle, we define
\[
\bar I^1(J) :=I^1(J)\setminus \sT(A_n), \quad \text{ where }\quad \sT(A_n):= \bigcup_{j \in \cL(A_n)} \supp(\col_j(\fold(A_n))) 
\]
From the definition of $\sT(A_n)$ it is evident that any subset $I \subset \bar I^1(J)$ now satisfies the property \eqref{eq:row-does-not-sect}. 
We further note that
\[
|I^1(J)\cap \sT(A_n)|  \le \sum_{j \in J} \left| \supp(\col_j(\fold(A_n))) \cap \sT(A_n)\right|
  \le \f{\d_0}{16}np |J|,
\]
where in the last step we have used the property \eqref{item: normal column} of $\Omega_{\ref{lem: typical structure}}$. Thus we proved that  on the event $\Omega_{\ref{lem: typical structure}}$,
\beq\label{eq: bar-I^1(J)}
|\bar I^1(J)| \ge \f{7 \d_0}{16} np |J|
\eeq
for any $J \subset [n]$ satisfying $J \cap \cL(A_n) = \varnothing$.

 It remains to show that $I^0(J') \cap \bar I^1(J)$ has a large cardinality, for any choice of $J'$ disjoint from $J$, with high probability.
To prove it, we recall that $I^0(J') \subset [\gn] \setminus (\bar J \cup \bar J')$. Thus, using Assumption \ref{ass:matrix-entry} we find that for any $J' \subset \binom{[n]}{\gm}$ and any $i \in [\gn]\setminus (\bar J \cup \bar J')$
 \[
 \P(i \in I^0(J') )  \ge (1-2p)^{|J'|} \ge 1-2p\gm.
 \]
  Hence, for a given $I \subset [\gn]$, $\E |I \setminus I^0(J')| \le 2 p \gm \cdot |I| \le |I|/4$ by the assumptions on $\kappa$ and $\gm$.
 So, by Chernoff's inequality (see e.g.,  \cite[Theorem 2.3.1]{V0})
 \[
  \P\left(|I \setminus I^0(J')| \ge \frac{1}{2}|I| \right)
  \le \exp\left( - 2 p \gm |I| - \f{|I|}{2} \log \left(\f{1}{4 p \gm e}\right) \right)
  \le \exp \left( - \frac{ |I|}{16} \log \left( \frac{1}{8 p \gm} \right)\right),
 \]
for any $c^\star_{\ref{lem: pattern}} \le 1/16$.
 Therefore, for any $I \subset [\gn]$ such that $|I| \ge \frac{\d_0}{4}\kappa np$, we deduce that
 \begin{align*}
  & \P \Big( \exists J' \in \binom{[n]}{m} \text{ such that }
         |I^0(J') \cap I| \le \frac{\d_0}{8} \kappa np \Big) \\
   \le& \sum_{J' \in \binom{[n]}{\gm}} \P (|I \setminus I^0(J')| \ge \frac{1}{2}|I| )
   \\
   \le& \binom{n}{ \gm} \cdot \exp \left( - \frac{ |I|}{16} \log \left( \frac{1}{8 p \gm} \right)\right) \\
   \le&  \exp \left( \gm \cdot \log \left( \frac{e n}{\gm} \right) - \frac{\d_0}{64}\kappa np \cdot \log \left( \frac{1}{8 p \gm} \right) \right)
    = \exp (- \kappa  np \cdot U),
 \end{align*}
 where
\[
 U:= \frac{\d_0}{64}\log\left(\frac{1}{8p \gm}\right) - \frac{\gm}{\kappa  n p} \log \left(\frac{e n}{\gm}\right).
 \]
We now need to find a lower bound on $U$ for which we split the ranges of $p$. First let us consider $p \in (0,1)$ such that $c^\star_{\ref{lem: pattern}} (p \sqrt{np})^{-1} \ge 1$. For such choices of $p$ we will show that for any $\kappa \le c^\star_{\ref{lem: pattern}} \cdot (p \sqrt{np})^{-1}$, with $c^\star_{\ref{lem: pattern}}$ sufficiently small, and $\gm = \kappa \sqrt{np}$ we have $U \ge 3$. To this end,
 denote
 \beq\label{eq:lbd-a}
   \alpha: = \frac{1}{8 \kappa p \sqrt{pn}} \ge \f{1}{8 c^\star_{\ref{lem: pattern}}}.
 \eeq
 Note that
 \begin{align*}
  U
  =  \frac{\d_0}{64}   \log \left( \frac{1}{8 \kappa p\sqrt{np}}\right)
    - \frac{1}{\sqrt{np}}  \log \left(\frac{en}{\kappa \sqrt{np}} \right) &=  \frac{\d_0}{64}  \log \alpha -  \frac{1}{\sqrt{np}}  \log (8e pn \alpha) \\
  &=\frac{\d_0}{64}\log\alpha-\frac{1}{\sqrt{np}}\log \alpha -\frac{1}{\sqrt{pn}}\big(\log(8e)+ \log (np)\big)\\
  & \ge \f{\d_0}{128} \log \alpha,
 \end{align*}
 for all large $n$, where the last step follows upon noting that by the assumption on $p$ we have $np \to \infty$ as $n \to \infty$, using the fact that $x^{-1/2}\log x \to 0$ as $x \to \infty$, and the lower bound on $\alpha$ (see \eqref{eq:lbd-a}). Choosing $c^\star_{\ref{lem: pattern}}$ sufficiently small and using \eqref{eq:lbd-a} again we deduce that $U \ge 3$, for all large $n$.

 Now let us consider $p \in (0,1)$ such that $c^\star_{\ref{lem: pattern}} (p \sqrt{np})^{-1} < 1$. For such choices of $p$ we have that $\kappa =1$ and $\gm= c^\star_{\ref{lem: pattern}} p^{-1}$. Therefore, recalling the definition of $U$ we note that
 \begin{align*}
 U = \f{\d_0}{64} \log \left( \f{1}{8 c^\star_{\ref{lem: pattern}}}\right) - \f{c^\star_{\ref{lem: pattern}}}{n p^2} \log (e(c^\star_{\ref{lem: pattern}})^{-1} np)  & \ge \f{\d_0}{64} \log \left( \f{1}{8 c^\star_{\ref{lem: pattern}}}\right) - (c^\star_{\ref{lem: pattern}})^{-1/3} n^{-1/3}  \log (e(c^\star_{\ref{lem: pattern}})^{-1} n)\\
 & \ge \f{\d_0}{128} \log \left( \f{1}{8 c^\star_{\ref{lem: pattern}}}\right) \ge 3,
 \end{align*}
 where the first inequality follows from the assumption that $c^\star_{\ref{lem: pattern}} (p \sqrt{np})^{-1} < 1$.

 This proves that, for any $p \in (0,1)$ such that $np \ge \log (1/\bar C p)$ and any $I \subset [\gn]$ with $|I|\ge \frac{\d_0}{4}\kappa  np$, we have
 \[
   \P \left( \exists J' \in \binom{[n]}{ \gm} \text{ such that } |I^0(J') \cap I| \le \frac{\d_0}{8} \kappa np \right)
   \le \exp (-3 \kappa np).
 \]
To finish the proof, for a set $J \in \binom{[n]}{\kappa}$ we define
 \begin{multline*}
  p_J:=\P \bigg( \left\{\exists J' \in \binom{[n]}{\gm} \text{ such that } J' \cap J= \emptyset, \
        |\bar I^{1}(J) \cap I^0(J')| < \frac{\d_0}{8} \kappa np \right\} \\
        \cap \{J \cap \cL(A_n)=\emptyset\}\cap  \Omega_{\ref{lem: typical structure}} \bigg).
 \end{multline*}
Since $J$ and $J'$ are disjoint and $I^0(J') \subset [\gn]\setminus (\bar J \cup \bar J')$, it follows from Assumption \ref{ass:matrix-entry} that the random subsets $\bar I^{1}(J)$ and $I^0(J')$ are independent. Using \eqref{eq: bar-I^1(J)} we obtain that for any $J \subset [n]$,
 \begin{align}
    p_J
    &\le
     \sum_{I \subset [\gn] }
        \P \left( \left\{\bar I^{1}(J)=I, \, J \cap \cL(A_n) = \emptyset\right\} \cap\Omega_{\ref{lem: typical structure}} \cap \left\{\exists J' \in \binom{[n]}{\gm} \text{ such that }
         |I^0(J') \cap I| \le \frac{\d_0}{8} \kappa np \right\} \right)
        \notag\\
    &\le
     \sum_{I \subset [\gn] ,\ |I| > \frac{\d_0}{4}\kappa np } \P(\bar I^{1}(J)=I)
        \P \Big( \exists J' \in \binom{[n]}{\gm} \text{ such that }
         |I^0(J') \cap I| \le \frac{\d_0}{8} \kappa np \Big)
        \notag\\
    &\le
     \exp(-3\kappa np) \sum_{I \subset [\gn] ,\ |I| > \frac{\d_0}{4}\kappa np } \P(\bar I^{1}(J)=I) \le \exp(-3 \kappa  n p),\label{eq:conc_large_set}
 \end{align}
for all large $n$.
%
%
The rest of the proof consists of taking union bounds.
 First, using the union bound over $J \in \binom{[n]}{\kappa}$ satisfying $J \cap \cL(A_n) = \varnothing$, setting ${c}_{\ref{lem: pattern}}=\d_0/16$, we get that
 \begin{align*}
  \P \left(\bigcup_{\substack{J \in \binom{[n]}{\kappa} \\ J \cap \cL(A_n) = \varnothing }} \ \bigcup_{J' \in \binom{[n]}{\gm}, \, J \cap J'=\emptyset} ({\mathcal{A}}^{J, J'}_{{c}_{\ref{lem: pattern}, \kappa}})^c
 \cap \Omega_{\ref{lem: typical structure}} \right)   \le \binom{n}{\kappa} \exp(-3 \kappa  n p)
 & \le \exp( \kappa \log n - 3 \kappa  n p)\\
  & \le \exp( -  \kappa  n p).
 \end{align*}
 Finally taking another union bound over $\kappa \le c^\star_{\ref{lem: pattern}} (p \sqrt{np})^{-1} \vee 1$ we obtain the desired result.
 \end{proof}

\vskip5pt

Note that in \eqref{eq: pattern} we could only consider $J \subset [n]$ such that $J \cap \cL(A_n)=\varnothing$. As we will see later, when we apply Lemma \ref{lem: pattern} to establish the invertibility over vectors that are close to sparse, we have to know that $\|A_n x_{[n]\setminus\cL(A_n)}\|_2$ is large for $x \in S^{n-1}$ close to very sparse vectors. So, one additionally needs to show that  $\|A_n x\|_2$ and $\|x_{[n]\setminus\cL(A_n)}\|_2$ cannot be small at the same time. The following lemma does this job. Its proof again uses the structural properties of $A_n$ derived in Lemma \ref{lem: typical structure}.

Before stating the next lemma let us recall that $\Omega_0^c$ is the event that the matrix $A_n$ has neither zero columns  nor zero rows (see \eqref{eq:Omega-0}).


\begin{lem}  \label{lem: normal coordinates}
Let $A_n$ satisfy Assumption \ref{ass:matrix-entry}. Fix a realization of $A_n$ such that the event $\Omega_0^c \cap \Omega_{\ref{lem: typical structure}}$ occurs. Let $x \in S^{n-1}$ be such that $\norm{A_n x}_2 < 1/4$.
  Then
  \[
   \norm{x_{[n] \setminus \cL(A_n)} }_2
   \ge \frac{1}{C_{\ref{lem: normal coordinates}}np},
  \]
  for some absolute constant $C_{\ref{lem: normal coordinates}}$.
\end{lem}

\begin{proof}
 We may assume that $\norm{x_{\cL(A_n)}}_2 \ge 3/4$, since otherwise there is nothing to prove.
 For any $j \in \cL(A_n)$, we choose $i:=i_j$ such that $a_{i_j,j}=1$. Such a choice is possible since we have assumed that $\Omega_0^c$ occurs. Using the property \eqref{item: disjoint support} of the event $\Omega_{\ref{lem: typical structure}}$ we see that any such function $i: \cL(A_n) \to [n]$ is an injection.

 Set
 \[J_0:=  \{j \in \cL(A_n): \ |(A_n x)_{i_j}| \ge (1/2) |x_j| \}.\]
 If $\norm{x_{J_0}}_2 \ge 1/2$, then
 \[
  \norm{A_n x}_2 \ge \left( \sum_{j \in J_0} ( (A_n x)_{i_j} )^2 \right)^{1/2} \ge \frac{1}{4},
 \]
 which contradicts our assumption $\|A_n x \|_2 < \f14$.
 Hence, denoting $J_1:=\cL(A_n) \setminus J_0$, we may assume that $\norm{x_{J_1}}_2 \ge 1/4$. We then observe that for any $j \in J_1$,
 \begin{align*}
   \frac{1}{2}|x_j|
   \ge |(A_nx)_{i_j}|
  & \ge |a_{i_j,j} x_j| - \left| \sum_{k \neq j} a_{i_j,k} x_k \right| \\
  & \ge |x_j|- |\supp(\row_{i_j}(A_n))| \cdot \max_{k \in \supp(\row_{i_j}(A_n)) \setminus \{j\} } |x_k| \\
    & \ge |x_j|- C_{\ref{lem: typical structure}}np\cdot \max_{k \in \supp(\row_{i_j}(A_n)) \setminus \{j\} } |x_k|,
 \end{align*}
 where the last inequality follows upon using the property \eqref{item: no heavy} of the event $\Omega_{\ref{lem: typical structure}}$.

This shows that for any $j \in J_1$ there exists $k \in [n]\setminus \{j\}$ such that $a_{i_j,k}=1$ and
 \[
  |x_k| \ge \frac{1}{2 C_{\ref{lem: typical structure}}np} |x_j|.
 \]
 Choose one such $k$ and denote it by $k(j)$. Using the property \eqref{item: disjoint support} of the event $\Omega_{\ref{lem: typical structure}}$ again, we deduce that $k(j) \in [n] \setminus \cL(A_n)$.
 Therefore,
 \[
   \frac{1}{16}
   \le \sum_{j \in J_1} x^2_j
   \le (2 C_{\ref{lem: typical structure}} np)^2 \sum_{j \in J_1} x^2_{k(j)}
   \le (2 C_{\ref{lem: typical structure}} np)^2 r_0 \sum_{k \in [n] \setminus \cL(A_n)} x^2_k.
 \]
 Here, the last inequality follows since by \eqref{item: bounded multiplicity} of the event $\Omega_{\ref{lem: typical structure}}$, we have that for any $k \in [n] \setminus \cL(A_n)$,
 \[
   | \{ j \in \cL(A_n): \ k(j)=k \} | \le r_0.
 \]
This finishes the proof of the lemma.
\end{proof}

\vskip10pt
We see that Lemma \ref{lem: normal coordinates} provides a lower bound on $\|x_{[n]\setminus \cL(A_n)}\|_2$ that deteriorates as $p$ increases. We show below that for large $p$ the set of light columns is empty with high probability. Hence, in that regime we can work with $x$ instead of $x_{[n]\setminus \cL(A_n)}$. Furthermore, during the course of the proof of Proposition \ref{l: sparse vectors-2} we will see that to deduce that $x_{[n]\setminus \cL(A_n)}$ itself is close to sparse vectors we need bounds on $|\cL(A_n)|$ for all $p$ satisfying $np \ge \log(1/p)$. Both these statements are proved in the following lemma.

\begin{lem}\label{lem:light-col-card}
Let $A_n$ satisfies Assumption \ref{ass:matrix-entry}. If $np \ge \log(1/\bar{C}p)$ for some $\bar{C}\ge 1$ then
\[
\P(|\cL(A_n)| \ge n^{\f13}) \le n^{-\f19},
\]
for all large $n$. Moreover, there exists an absolute constant $C_{\ref{lem:light-col-card}}$ such that if $np \ge C_{\ref{lem:light-col-card}} \log n$ then
\[
\P(\cL(A_n) =\emptyset) \ge 1- 1/n.
\]
\end{lem}

\vskip10pt

Proof of Lemma \ref{lem:light-col-card} follows from standard concentration bounds and is postponed to Appendix \ref{sec:appendix}. Equipped with all the relevant ingredients we are now ready to state the main result of this section.

  \begin{prop}[Invertibility over very sparse vectors]    \label{l: sparse vectors-2}
   Let $A_n$ satisfies Assumption \ref{ass:matrix-entry} where $p$ satisfies the inequality
   \[
   np \ge \log(1/p).
   \]
  Fix $K, C_\star \ge 1$ and let
   \beq\label{eq:ell_0_dfn}
   \ell_0:= \left \lceil \frac{\log \left(\f{c^\star_{\ref{lem: pattern}}}{ p}\right)}{\log \sqrt{pn}} \right \rceil.
  \eeq
  Then there exist constants $0< c_{\ref{l: sparse vectors-2}}, \wt{c}_{\ref{l: sparse vectors-2}}< \infty$, depending only on $\delta_0$, and an absolute constant $\ol{c}_{\ref{l: sparse vectors-2}}$ such that for any $y_0 \in \R^n$ with $|\supp(y_0)| \le C_\star np$, we have
  \begin{multline*}
   \P\left(\left\{\exists x \in V_0 \text{ such that } \norm{A_n x -y_0}_2 \le \rho \sqrt{np} \text{ and } \norm{A_n-\E A_n} \le K \sqrt{np}\right\} \cap \Omega_0^c \cap \Omega_{\ref{lem: typical structure}}\right)
   \le n^{-\bar{c}_{\ref{l: sparse vectors-2}}}, 
  \end{multline*}
  for all large $n$,
  where
  \beq\label{eq:V-0}
  V_0:= \text{\rm Dom}\left(c^\star_{\ref{lem: pattern}}p^{-1}, c_{\ref{l: sparse vectors-2}}K^{-1}\right)\cup {\rm{Comp}}(c^\star_{\ref{lem: pattern}}p^{-1}, \rho), \quad \text{ and } \quad \rho:= (\wt c_{\ref{l: sparse vectors-2}}/K)^{2\ell_0+1}.
  \eeq
 \end{prop}
 
 \begin{rmk}
Let us add that during the proof of Proposition \ref{l: sparse vectors-2} it will be shown that on the event $\mathcal{E}_{\ref{lem: pattern}}^c \cap \mathcal{E}_{\ref{lem:light-col-card}}^c \cap \Omega_0^c \cap \Omega_{\ref{lem: typical structure}} \cap \left\{\norm{A_n-\E A_n} \le K \sqrt{np}\right\}$, for any $y_0 \in \R^n$ with $|\supp(y_0)| \le C_\star np$, one deterministically has that $\|A_n x - y_0\|_2 > \rho \sqrt{np}$ for all $x \in V_0$, where $\mathcal{E}_{\ref{lem: pattern}}$ and  $\mathcal{E}_{\ref{lem:light-col-card}}$ are the {\em bad events} as identified in Lemmas \ref{lem: pattern} and \ref{lem:light-col-card}, respectively. The probability bound in Proposition \ref{l: sparse vectors-2} is a consequence of probability bounds obtained in those two lemmas.
 \end{rmk}

 \vskip10pt

 \begin{proof}[Proof of Proposition \ref{l: sparse vectors-2}]
Since we do not have any control on the vector $y_0$ except the cardinality of its support, to remove the effect of $y_0$ we will show that the $\ell_2$-norm of the vector $A_n x - y_0$ restricted to the complement of $\supp(y_0)$ has a uniform bound over $x \in V_0$. To this end, for ease of writing, we define $\bar A_n$ to be the sub-matrix of $A_n$ of dimension $\bar n \times n$, where $\bar n := n - |\supp(y_0)|$, obtained by deleting the rows indexed by $ \supp(y_0)$. We have that $\|\bar A_n x\|_2 \le \|A_n x - y_0\|_2$.

Next we observe that for any $x \in \R^n$ an application of the triangle inequality implies that
\[
\|\fold(\bar A_n) x \|_2^2 \le 2 \|\bar A_n x \|_2^2.
\]
Furthermore
\begin{align}
 \|\fold (\bar A_n)\| & \le \|\bar A_{n,1} - \E \bar A_{n,1}\|+\|\bar A_{n,2} - \E \bar A_{n,2}\| + \|\E \bar A_{n,1} - \E \bar A_{n,2}\| \notag\\
 & \le 2 \|A_n- \E A_n\| + \|\E A_{n,1} - \E A_{n,2}\| \le 2 \|A_n- \E A_n\| + 2\sqrt{np}, \label{eq:fold-norm-0}
\end{align}
where in the last step we have used the fact that
\beq\label{eq:fold-norm}
\|\E  A_{n,1} -  \E A_{n,2} \| \le \sqrt{2n} \cdot p \le 2\sqrt{np}.
\eeq
To establish \eqref{eq:fold-norm} we note that Assumption \ref{ass:matrix-entry} implies that there at most two non-zero entries  per row in the matrix $\E {A}_{n,1} - \E {A}_{n,2}$ each of which has absolute value less than or equal to $p$. Therefore, each of the entries of $(\E  A_{n,1} - \E A_{n,2}) (\E  A_{n,1} - \E  A_{n,2})^*$ is bounded by $2p^2$ and hence by the Gershgorin circle theorem we deduce \eqref{eq:fold-norm}.

Therefore, in light of \eqref{eq:fold-norm-0}, recalling $K \ge 1$, it is enough to find a bound on the probability of the event
\[
\gV_{V_0}:=\left\{\exists x \in V_0: \norm{\fold(\bar A_n) x}_2 \le 2 (\wt c_{\ref{l: sparse vectors-2}}/K)^{2 \ell_0+1} \sqrt{np} \right\}\cap \Omega_K \cap \Omega_0^c \cap \Omega_{\ref{lem: typical structure}},
\]
where
\[
\Omega_K:= \left\{\norm{\fold(A_n)} \le 4 K \sqrt{np}\right\}.
\]
We will show that
 \begin{multline}\label{eq:dom-bd-1}
   \P\bigg(\left\{\exists x \in \text{\rm Dom}\left(c^\star_{\ref{lem: pattern}}p^{-1}, c_{\ref{l: sparse vectors-2}}K^{-1}\right):  \norm{\fold(\bar A_n) x }_2 \le (\wt c_{\ref{l: sparse vectors-2}}/K)^{2 \ell_0} \sqrt{np}\right\}\\
   \cap \Omega_K \cap \Omega_0^c \cap \Omega_{\ref{lem: typical structure}}\bigg)
   \le n^{-\bar{c}_{\ref{l: sparse vectors-2}}}.
   \end{multline}
 First let us show that $\P(\gV_{V_0}) \le n^{-\bar{c}_{\ref{l: sparse vectors-2}}}$ assuming \eqref{eq:dom-bd-1}. To this end, denoting $m:=c^\star_{\ref{lem: pattern}}p^{-1}$, we note that for any $x \in \text{\rm Comp}(m,\rho)$
  \begin{align}\label{eq:dom-to-comp}
\left| \|\fold(\bar A_n) x \|_2- \norm{\fold(\bar A_n) \f{x_{[1:m]}}{\|x_{[1:m]}\|_2}}_2\right|
& \le  \|\fold(\bar A_n)\| \cdot \left( \norm{x_{[1:m]}-\f{x_{[1:m]}}{\|x_{[1:m]}\|_2}}_2  +  \|x_{[m+1:n]}\|_2\right)\notag\\
&  \le 4 K\sqrt{np} \cdot \left( 1- \|x_{[1:m]}\|_2  +   \|x_{[m+1:n]}\|_2\right)\notag\\
& \le 8K \rho \sqrt{np} =8 \wt c_{\ref{l: sparse vectors-2}} \cdot (\wt c_{\ref{l: sparse vectors-2}}/K)^{2\ell_0} \sqrt{np},
 \end{align}
 on the event $\Omega_K$. For $x \in V_0$ we have that $x_{[1:m]}/\|x_{[1:m]}\|_2 \in {\rm Sparse}(m) \cap S^{n-1} \subset \text{\rm Dom}\left(m, c_{\ref{l: sparse vectors-2}}K^{-1}\right)$ we see from \eqref{eq:dom-bd-1} that
 \[
 \norm{\fold(\bar A_n) \f{x_{[1:m]}}{\|x_{[1:m]}\|_2}}_2 \ge (\wt c_{\ref{l: sparse vectors-2}}/K)^{2\ell_0} \sqrt{np}
 \]
 with the desired high probability. Therefore, upon shrinking $\wt c_{\ref{l: sparse vectors-2}}$ such that $10 \wt c_{\ref{l: sparse vectors-2}} \le 1$, and recalling that $K \ge 1$, we deduce from \eqref{eq:dom-to-comp} that $\P(\gV_{V_0}) \le n^{-\bar{c}_{\ref{l: sparse vectors-2}}}$. Thus, it now suffices to prove \eqref{eq:dom-bd-1}.


 Turning to this task, we split the proof into three parts depending on the sparsity level of the matrix $A_n$, determined by $p$. First let us consider the case $\log(1/p) \le np \le C_{\ref{lem:light-col-card}} \log n$.

Fix  $ x \in \text{\rm Dom}\left(m, c_{\ref{l: sparse vectors-2}}K^{-1}\right)$ and define $\tilde x \in \R^n$ to be vector obtained from $x$ after setting the coordinates belonging to the set $\cL(\bar A_n)$ to be zero. That is,
\[
\tilde x_i := {x_i}\cdot \bI(i \in [n]\setminus \cL(\bar A_n)), \, i \in [n].
\]
Further set $\hat x$ to be the normalized version of $\tilde x$. So
\[
\hat x:= \tilde x /{\|{x_{[n]\setminus \cL(\bar A_n)}}\|_2}.
\]
During the course of the remainder of the proof we will see that to prove \eqref{eq:dom-bd-1} it suffices to consider only those $x$ for which ${\|{x_{[n]\setminus \cL(\bar A_n)}}\|_2} \ne 0$. Therefore, $\hat x$ is well defined.

Let us rearrange the magnitudes of the coordinates of $\hat x$ and group them in blocks of lengths $(pn)^{\ell/2}$, where $\ell=1 \etc \ell_0$.
  More precisely, set
  \beq\label{eq:hatz-1}
    \hat z_\ell:=\hat x_{[(pn)^{(\ell-1)/2}+1: (pn)^{\ell/2}]}\footnote{when $\ell =1$ by a slight abuse of notation we take $\hat z_1=x_{[1: \sqrt{np}]}$.},
   \eeq
   and
   \beq\label{eq:hatz-2}
    \hat z_{\ell_0+1}:=\hat x_{[(pn)^{\ell_0/2}+1 : n]}.
   \eeq
 For clarity of presentation, let us assume that $m=(pn)^{\ell_0/2}$, i.e.~the integer part in the definition of $\ell_0$ is redundant. Recalling the definition of $\tilde x$ we see that it matches with $x$ except $|\cL(\bar A_n)|$ coordinates. Therefore, for any $x \in \text{\rm Dom}(m, c_{\ref{l: sparse vectors-2}}K^{-1})$ we find that
\begin{align*}
\|\tilde x_{[m+1:n]}\|_2 \le  \|x_{[m+1:n]}\|_2  \le \, & c_{\ref{l: sparse vectors-2}}K^{-1} \sqrt{m} \|x_{[m+1:n]}\|_\infty \\
\le \, &  \sqrt{\f{m}{m/2- |\cL(\bar A_n)|}} \cdot c_{\ref{l: sparse vectors-2}}K^{-1} \|x_{[m/2+|\cL(\bar A_n)|+1:m]}\|_2 \\
\le \, &  2 c_{\ref{l: sparse vectors-2}}K^{-1} \|\tilde x_{[m/2+1:m]}\|_2,
\end{align*}
on the event
\[
\Omega_\cL:= \left\{ |\cL(\bar A_n)| \le n^{1/3}\right\},
\]
for all large $n$, where in the last step we have used the fact that for $p = O(\frac{\log n}{n})$ we have $n^{1/3} = o(m)$. This further implies that
\beq\label{eq:ell-0-ell-0+1}
\|\hat z_{\ell_0+1}\|_2 \le 2 c_{\ref{l: sparse vectors-2}}K^{-1} \|\hat x_{[m/2+1:m]}\|_2 \le 2 c_{\ref{l: sparse vectors-2}}K^{-1} \|\hat z_{\ell_0}\|_2,
\eeq
on the event $\Omega_\cL$, where the last inequality is a consequence of the fact that the condition $np \to \infty$ as $n \to \infty$ implies that the support of $\hat z_{\ell_0}$ contains that of $\hat x_{[m/2+1:m]}$.

Since $\sum_{\ell=1}^{\ell_0+1}\|\hat z_\ell\|_2^2 =1$, we deduce from \eqref{eq:ell-0-ell-0+1} that
\[
\sum_{\ell=1}^{\ell_0} \|\hat z_\ell\|_2^2 \ge 1 - 4 c_{\ref{l: sparse vectors-2}}^2K^{-2}.
\]
Hence, choosing $c_{\ref{l: sparse vectors-2}}$ sufficiently small we obtain that there exists $\ell \le \ell_0$ such that $\norm{\hat z_\ell}_2 \ge (c_{\ref{l: sparse vectors-2}}/K)^{\ell}$. Let $\ell_\star$ be the largest index having this property, and set $u:=\sum_{\ell=1}^{\ell_\star} \hat z_\ell, \ v=:\sum_{\ell=\ell_\star+1}^{\ell_0+1} \hat z_\ell$. First consider the case when $\ell_\star < \ell_0$. Then by the triangle inequality we have that
  \beq\label{eq:bound v}
   \norm{v}_2 \le \sum_{\ell=\ell_\star+1}^{\ell_0+1} \norm{\hat z_m}_2 \le 4 (c_{\ref{l: sparse vectors-2}}/K)^{(\ell_\star+1)},
  \eeq
  where we have used the inequality \eqref{eq:ell-0-ell-0+1}.

Let $\kappa=(np)^{(\ell_\star-1)/2}$.
Note that
    \[
     \kappa \le (np)^{(\ell_0-1)/2} \le \frac{1}{c^\star_{\ref{lem: pattern}} p  \sqrt{pn}}.
    \]
%
To finish the proof we now apply Lemma \ref{lem: pattern} with this choice of $\kappa$.  Using the fact that $|\supp(y_0)| \le  C_\star np$ we see that
\beq\label{eq:np-lbd}
\bar n p \ge \log(1/\bar C p),
\eeq
for some large constant $\bar C$, whenever $p \le c$ for some small constant $c$. Therefore, we can apply Lemma \ref{lem: pattern} to the rectangular matrix $\bar A_n$ to find the desired uniform bound on $\| \fold(\bar A_n) x\|_2$. To this end, we split the support of $u$ into $\sqrt{np}$ blocks of equal size $\kappa$ and define $L_{\ell_\star}:= \pi_{\hat x}^{-1}([1, (np)^{\ell_\star/2}])$, where $\pi_{\hat x}$ is the permutation of absolute values of the coordinates of $\hat x$ in an non-increasing order. For $s \in [\sqrt{np}]$, define $J_s:=\pi_{\hat x}^{-1}([(s-1)\kappa+1, s\kappa])$, and set $J_s':=L_{\ell_\star} \setminus J_s$.
Using Lemma \ref{lem: pattern}, for any $s \in [\sqrt{np}]$, we will show that there is a substantial number of rows of $\bar A_n$ which have one non-zero entry in the block $J_s$ and no such entries in $J_s'$. Let us check it.
On the event $\Omega_\cL$,
\beq\label{eq:L_star-new}
|L_{\ell_\star}| \le \frac{1}{c^\star_{\ref{lem: pattern}} p} \le \f{2 n}{c^\star_{\ref{lem: pattern}}\log n} \le n - n^{1/3} \le n - |\cL(\bar  A_n)|,
\eeq
 where the second inequality uses  assumption \eqref{eq:np-lbd}, which in particular implies that $p \ge \log n /(2n)$. Since $\hat x_{\cL(\bar A_n)} =0$, and $L_{{\bm \ell}_\star} $ contain the coordinates of $\hat x$ with the largest absolute value, it implies that $L_{{\bm \ell}_\star}\cap \cL(\bar A_n)=\emptyset$. Otherwise $|L_{\ell_\star}| > |\cL(\bar A_n)^c|$, yielding a contradiction to \eqref{eq:L_star-new}. Hence we also get that $J_s \cap \cL(\bar A_n)=\emptyset$. Moreover, $|J_s'| \le |L_{\ell_\star}| = \kappa \sqrt{pn}$. Therefore we now apply Lemma \ref{lem: pattern} to get a set $\cA$ such that $\cA^c \cap \Omega_{\ref{lem: typical structure}}$ has a small probability and on $\cA \cap \Omega_\cL$ there exist subsets of rows $I_s \subset [\bar n]$ with $|I_s| \ge {c}_{\ref{lem: pattern}} \kappa np$ for all $s \in [\sqrt{pn}]$, such that for every $i \in I_s$, we have $|\ga_{i,j_i}|  = 1$ for only one index $j_i \in J_s$ and $\ga_{i,j}=0$ for all $j \in (J_s \cup J_s' )\setminus \{j_0\}$. This means that $I_1,I_2,\ldots,I_{\sqrt{pn}}$ are disjoint subsets. Moreover $\{I_s\}_{s \in [\sqrt{np}]}$ satisfy the property \eqref{eq:row-does-not-sect}. That is,
\beq
\supp(\row_i(\fold(\bar A_n))) \cap \cL(\bar A_n) = \emptyset,  \quad \text{ for all } i \in I_s, \text{ and } s \in [\sqrt{np}].
\eeq

Therefore, for $s \in [\sqrt{np}]$ and $i \in I_s$,
\[
\left(\fold (\bar A_n) \cdot (x_{\cL(\bar A_n)})\right)_i =0
\]
and thus denoting $w:= \f{x_{\cL(\bar A_n)}}{\|x_{[n]\setminus \cL(\bar A_n)}\|_2}$ we deduce that
\[
\left| \left(\fold(\bar A_n) \cdot\left(u+ w \right)\right)_i\right|=|(\fold(\bar A_n) u)_i| = |u_{j_i}| \ge |\hat x(\pi_{\hat x}^{-1}(s \kappa))|,
\]
where the inequality follows from the monotonicity of the sequence $\{|\hat x(\pi_{\hat x}^{-1}(k))|\}_{k=1}^n$.
Hence
  \begin{align}
    \norm{\fold(\bar A_n) \cdot (u+w)}_2^2
    \ge \sum_{s=1}^{\sqrt{np}} \sum_{i \in I_s} \big( (\fold(A_n) u)_i \big)^2
   & \ge {{c}_{\ref{lem: pattern}}np} \sum_{s=1}^{(pn)^{1/2}}\kappa (\hat x(\pi_{\hat x}^{-1}(s\kappa)))^2 \notag\\
    &\ge {{c}_{\ref{lem: pattern}}np} \sum_{k=(np)^{(\ell_\star-1)/2}}^{(np)^{\ell_\star/2}} (\hat x(\pi_{\hat x}^{-1}(k)))^2 \notag \\
    &= {{c}_{\ref{lem: pattern}}np}\norm{\hat z_{\ell_\star}}_2^2  \ge {{c}_{\ref{lem: pattern}}np} \cdot (c_{\ref{l: sparse vectors-2}}/K)^{2 \ell_\star}. \label{i: z_l}
\end{align}
Note that all but the last step above continues to hold even when $\ell_\star=\ell_0$. Combining \eqref{i: z_l} with the bound on $\norm{v}_2$ (see \eqref{eq:bound v}), we deduce that
  \begin{align}
      \norm{\fold(\bar A_n) \cdot  \f{ x}{\|x_{[n]\setminus \cL(\bar A_n)}\|_2}}_2
      &\ge     \norm{\fold(\bar A_n) (u+w)}_2 - \norm{\fold (\bar A_n)} \cdot \norm{v}_2  \notag\\
      &\ge \sqrt{{{c}_{\ref{lem: pattern}}np} } \cdot (c_{\ref{l: sparse vectors-2}}/K)^{ \ell_\star}-  4 K\sqrt{np} \cdot 4 (c_{\ref{l: sparse vectors-2}}/K)^{(\ell_\star+1)}  \ge (\wt c_{\ref{l: sparse vectors-2}}/K)^{\ell_\star}\sqrt{np}, \label{eq:lbd-ell_0-1}
  \end{align}
on the set $\cA \cap \Omega_\cL \cap \Omega_K \cap \Omega_{\ref{lem: typical structure}}$, where the last inequality follows upon choosing $c_{\ref{l: sparse vectors-2}}$ and $\wt c_{\ref{l: sparse vectors-2}}$ sufficiently small (independently of $\ell_\star$).

 Now it remains to consider the case $\ell_\star=\ell_0$. Proceeding similarly as in \eqref{i: z_l} we have that
\[\norm{\fold(\bar A_n) \cdot (u+w) }_2 \ge \sqrt{{c}_{\ref{lem: pattern}}np } \cdot \norm{\hat z_{\ell_0}}_2,\]
and from \eqref{eq:ell-0-ell-0+1}, we have $\norm{v}_2=\norm{\hat z_{\ell_0+1}}_2 \le 2 c_{\ref{l: sparse vectors-2}}K^{-1}  \norm{\hat z_{\ell_0}}_2$. Therefore proceeding  as before, on $\cA \cap \Omega_K \cap \Omega_\cL \cap \Omega_{\ref{lem: typical structure}}$, we obtain
\beq\label{eq:lbd-ell_0-2}
\norm{\fold(\bar A_n) \cdot \f{x}{\| x_{[n]\setminus \cL(\bar A_n)}\|_2}}_2 \ge (\wt c_{\ref{l: sparse vectors-2}}/K)^{\ell_\star}\sqrt{np}.
\eeq
Since $np \le C_{\ref{lem:light-col-card}} \log n$, using Lemma \ref{lem: normal coordinates} we also have that
\[
\|x_{[n]\setminus \cL(A_n)}\|_2 \ge \f{1}{C_{\ref{lem: normal coordinates}} np} \ge \f{1}{C_{\ref{lem: normal coordinates}} C_{\ref{lem:light-col-card}} \log n} \ge (c_{\ref{l: sparse vectors-2}}/K)^{\ell_\star},
\]
on the set $\Omega_0^c \cap \Omega_{\ref{lem: typical structure}}$, for all large $n$. This lower bound on $\|x_{[n]\setminus \cL(A_n)}\|_2$ shows in particular  in particular that $\hat x$ is well defined.

Combining \eqref{eq:lbd-ell_0-1}-\eqref{eq:lbd-ell_0-2}, and using Lemma \ref{lem: pattern} and Lemma \ref{lem:light-col-card} we establish \eqref{eq:dom-bd-1} for all $p \in (0,1/2]$ such that $\log(1/p) \le np \le C_{\ref{lem:light-col-card}} \log n$.

Next we consider the case when $C_{\ref{lem:light-col-card}} \log n \le np \le (c^\star_{\ref{lem: pattern}} n)^{2/3}$. For such choices of $p \in (0,1)$ we use Lemma \ref{lem:light-col-card} to obtain that $\{\cL(\bar A_n) = \emptyset\}$ with high probability. Using this fact one proceeds similarly as in the previous case to arrive at \eqref{eq:dom-bd-1}. Below is a brief outline.

Similarly to $\{\hat z_\ell\}_{\ell=1}^{\ell_0+1}$ defined in \eqref{eq:hatz-1}-\eqref{eq:hatz-2}, we first define $\{z_\ell\}_{\ell=1}^{\ell_0+1}$ by rearranging the magnitudes of the coordinates of $x$ and grouping them in blocks of length $(np)^{\ell/2}$ for $\ell=1,2,\ldots, \ell_0$ and $z_{\ell_0+1}$ being the remaining block. Next, we  define $\ell_\star$ to be the largest $\ell\le  \ell_0$ such that $\norm{z_\ell}_2 \ge (c_{\ref{l: sparse vectors-2}}/K)^{\ell}$. Equipped with the definition of $\ell_\star$ we then define $\{J_s, J_s'\}_{s \in [\sqrt{np}]}$ similarly as in the previous case. On the event $\{\cL(\bar A_n)=\emptyset\}$ the requirement that $J_s \cap \cL(\bar A_n)=\emptyset$ trivially follows. Since $np \ge C_{\ref{lem:light-col-card}} \log n$ by Lemma \ref{lem:light-col-card} we have that $\cL(\bar A_n)= \emptyset$ with high probability. This allows us to use Lemma \ref{lem: pattern} to find disjoint subsets of rows $\{I_s\}_{s \in [\sqrt{np}]}$ with the desired properties and hence by repeating the same computations as in the previous case we arrive at \eqref{eq:lbd-ell_0-2}. Now noting that $\|x_{[n]\setminus \cL(\bar A_n)}\|_2=1$, on a set with high probability, we obtain the desired bound in \eqref{eq:dom-bd-1}.

It remains to provide a proof of \eqref{eq:dom-bd-1} for $p \in (0,1/2]$ such that $np \ge (c^\star_{\ref{lem: pattern}} n)^{2/3}$. For this range of $p$ we do not need the elaborate chaining argument of the previous two cases. It follows from the following simpler argument.

Fixing $x \in \text{\rm Dom}\left(m, c_{\ref{l: sparse vectors-2}}K^{-1}\right)$, for $k \in \supp(x_{[1:m]})$ we define $J_k = \{k\}$ and $J_k' = \supp(x_{[1:m]})\setminus \{k\}$. Applying Lemma \ref{lem: pattern} with $\kappa=1$ and $\gm=m$ we find disjoint subsets of rows $\{I_k\}_{k \in \supp(x_{[1:m]})}$ such that $|I_k| \ge c_{\ref{lem: pattern}} np$ for all $k \in \supp(x_{[1:m]})$ (note that by Lemma \ref{lem:light-col-card} $\cL(\bar A_n) =\emptyset$ with high probability). Therefore, proceeding similarly as in  \eqref{i: z_l} we obtain that
\begin{align*}
 \left\| \fold(\bar A_n) \cdot \f{x_{[1:m]}}{\|x_{[1:m]}\|_2}\right\|_2^2 & \ge \sum_{k \in \supp(x_{[1:m]})} \sum_{i \in I_k} \left(\f{(\fold(\bar A_n) x_{[1:m]})_i}{\|x_{[1:m]}\|_2}\right)^2 \\
 & \ge c_{\ref{lem: pattern}} np \sum_{k \in \supp(x_{[1:m]})} \f{|x_k|^2}{\|x_{[1:m]}\|_2} = c_{\ref{lem: pattern}} np
 \end{align*}
 on set $\cA \cap \{\cL(\bar A_n) = \emptyset\}\cap \Omega_{\ref{lem: typical structure}}$ such that $\cA^c \cap \Omega_{\ref{lem: typical structure}}$ has a small probability. Using the fact that $x \in \text{\rm Dom}\left(m, c_{\ref{l: sparse vectors-2}}K^{-1}\right)$ we observe that
 \[
 \left\| \f{x_{[1:m]}}{\|x_{[1:m]}\|_2} - x_{[1:m]}\right\|_2 = 1- \|x_{[1:m]}\|_2 \le \|x_{[m+1:n]}\|_2 \le c_{\ref{l: sparse vectors-2}}K^{-1} \sqrt{m} \|x_{[m+1:n]}\|_\infty \le c_{\ref{l: sparse vectors-2}}K^{-1}.
 \]
 Thus applying the triangle inequality we deduce that for any $x \in \text{\rm Dom}\left(m, c_{\ref{l: sparse vectors-2}}K^{-1}\right)$
 \begin{align*}
 \|\fold(\bar A_n) x\|_2 & \ge \left\| \fold(\bar A_n) \f{x_{[1:m]}}{\|x_{[1:m]}\|_2}\right\|_2 - 4 K \sqrt{np} \left(  \left\| \f{x_{[1:m]}}{\|x_{[1:m]}\|_2} - x_{[1:m]}\right\|_2 + \|x_{[m+1:n]}\|_2  \right) \\
 &\ge  \sqrt{c_{\ref{lem: pattern}} np} - 8 c_{\ref{l: sparse vectors-2}} \sqrt{np} \ge \sqrt{\f{c_{\ref{lem: pattern}} np}{2}},
 \end{align*}
 on the event $\cA \cap \Omega_K \cap \{\cL(\bar A_n)=\emptyset\}$, whenever $c_{\ref{l: sparse vectors-2}} \le \f{1}{16} \sqrt{c_{\ref{lem: pattern}}}$.
 This together with Lemma \ref{lem:light-col-card} proves \eqref{eq:dom-bd-1} for all $p \in (0,1/2]$ such that $np \ge (c^\star_{\ref{lem: pattern}} n)^{2/3}$ and it finishes the proof of the proposition.
 \end{proof}

 \subsection{Invertibility over vectors close to moderately sparse}\label{sec:sparse-2}
 In this section we extend the uniform bound of Proposition \ref{l: sparse vectors-2} for vectors close to moderately sparse vectors.  The following is the main result of this section.
 \begin{prop}\label{p: spread vectors}
 Let A$_n$ be as in Assumption \ref{ass:matrix-entry}, $V_0$,and $\rho$ be as in Proposition \ref{l: sparse vectors-2}, and $p \ge c_1 \f{\log n}{n}$ for some constant $c_1>0$. Fix $K\ge 1$. Then there exist constants $0< c_{\ref{p: spread vectors}}, \wt{c}_{\ref{p: spread vectors}}, {c}^*_{\ref{p: spread vectors}}, \bar{c}_{\ref{p: spread vectors}}< \infty$, depending only on $K$, such that for any $M$ with $p^{-1} \le M \le {c}^*_{\ref{p: spread vectors}}n$ and $y_0 \in \R^n$ we have
  \begin{align*}
   \P&\Big(\exists x \in \text{\rm Dom}\left(M, c_{\ref{p: spread vectors}}K^{-1}\right)\cup {\rm{Comp}}(M, \rho )\setminus V_0  \text{ such that } \norm{A_n x -y_0}_2 \le \wt{c}_{\ref{p: spread vectors}}\rho \sqrt{np} \\
   &\hskip 3in \text{ and } \norm{A_n- \E A_n} \le K \sqrt{np}\Big) \le \exp(-\bar{c}_{\ref{p: spread vectors}}n).
  \end{align*}
 \end{prop}

 As outlined in Section \ref{sec:prelim} the key to the proof of Proposition \ref{p: spread vectors} will be to obtain an estimate on the small ball probability. This will be achieved by deriving bounds on the L\'{e}vy concentration function (recall Definition \ref{dfn:levy}). The necessary bound is derived in the lemma below.
 \begin{lem}     \label{l: spread vector-levy}
 Let $C_n$ be a $n_1 \times n_2$ matrix, where $n_1,n_2 \ge \gn$ (recall $\gn:=\lfloor n/2 \rfloor$), with i.i.d.~$\dBer(p)$ entries.  Then  for any $\alpha >1$, there exist $\be, \gamma >0$, depending only on $\alpha$ such that for $x \in \R^{n_2}$, satisfying
  $\norm{x}_\infty/\norm{x}_2 \le \alpha \sqrt{p}$, we have
 \[
   \cL \left(C_n x, \beta \cdot \sqrt{np} \norm{x}_2\right)
   \le \exp (-\gamma n ).
 \]
 \end{lem}

\vskip10pt
Lemma \ref{l: spread vector-levy} is a consequence of \cite[Corollary 3.7]{BR-invertibility}. The difference between Lemma \ref{l: spread vector-levy} and \cite[Corollary 3.7]{BR-invertibility} is that the latter has been proved for matrices whose entries have zero mean and obey a certain product structure. The key to the proof of \cite[Corollary 3.7]{BR-invertibility} is \cite[Lemma 3.5]{BR-invertibility}. Upon investigating the proof of \cite[Lemma 3.5]{BR-invertibility} it becomes evident that neither the zero mean condition nor the product structure of the entries are essential to its proof. So, repeating the proof of \cite[Lemma 3.5]{BR-invertibility} under the current set-up and following the proofs of \cite[Lemma 3.6, Corollary 3.7]{BR-invertibility} we derive Lemma \ref{l: spread vector-levy}. Further details are omitted.

We additionally borrow the following fact from the proof of \cite[Lemma 3.8]{BR-invertibility}.
\begin{fact}\label{fact:net-bd}
Fix $M_1 <M_2 < n$ and for any $x \in S^{n-1}$ define
\[
u_x:=u(x,M_1):= x_{[1:M_1]}, \quad v_x:=v(x,M_1,M_2):= x_{[M_1+1:M_2]},  \text{ and } r_x:=r(x,M_2):= x_{[M_2+1:n]}.
\]
Then, given any $\vep, \tau >0$ and a set $\cS \subset S^{n-1}$ there exists a set $\cM \subset \cS$ such that given any $x \in \cS$ there exists a $\bar{x} \in \cM$ such that
\beq\label{eq:net-approx}
\|u_x -u_{\bar{x}}\|_2 \le \vep, \quad \norm{\f{v_x}{\|v_x\|_2}- \f{v_{\bar{x}}}{\|v_{\bar{x}}\|_2}}_2 \le \tau, \quad \text{ and } \quad |\|v_x\|_2 - \|v_{\bar{x}}\|_2| \le \vep.
\eeq
and
\beq\label{eq:cM-bd}
|\cM| \le  {n \choose M_1}{n-M_1 \choose M_2-M_1} \cdot \left(\f{6}{\vep}\right)^{M_1+1} \cdot \left(\f{6}{\tau}\right)^{M_2-M_1}.
\eeq
\end{fact}

\vskip10pt

The proof of Fact \ref{fact:net-bd} follows from volumetric estimates. Indeed, one first fixes the choice of the supports of $u_x$ and $v_x$, and constructs standard nets for $u_x$, $v_x/\|v_x\|_2$, and  $\|v_x\|_2$ of desired precision. Then bounds on the cardinality of follows by taking a union over the set of all possible choices of the supports of $u_x$ and $v_x$. We omit further details.

Now we are ready to prove Proposition \ref{p: spread vectors}.

\begin{proof}[Proof of Proposition \ref{p: spread vectors}]
Fix $M$ with $p^{-1} \le M \le c^*_{\ref{p: spread vectors}} n$ and for ease of writing let us denote
 \beq\label{eq:def-V_1}
 V_1:= \text{\rm Dom}\left(M, c_{\ref{p: spread vectors}}K^{-1}\right)\cup {\rm{Comp}}(M, \rho),
 \eeq
 where $c_{\ref{p: spread vectors}}$ and $c^*_{\ref{p: spread vectors}}$ to be determined during the course of the proof. We will show that for any $y \in \R^n$
\begin{align}\label{eq:moderate-pre-prop}
\P\left(\left\{\exists x \in V_1\setminus V_0: \norm{(A_n-p {\bm J}_n ) x -y}_2 \le 2\wt{c}_{\ref{p: spread vectors}}\rho \sqrt{np}\right\} \cap \Omega_K^0\right) \le \exp(-2\bar{c}_{\ref{p: spread vectors}}n),
\end{align}
where ${\bm J}_n$ is the $n \times n$ matrix of all ones and
\beq\label{eq:Omega-K-0}
\Omega_K^0:=\left\{ \norm{A_n- \E A_n} \le K \sqrt{np}\right\}.
\eeq
First let us show that the proposition follows from \eqref{eq:moderate-pre-prop}. To this end, denote
\[
\cY_n:= \{ y \in \R^n: y = y_0+ \lambda {\bm 1}, |\lambda| \le \sqrt{n} p\}.
\]
It easily follows that $\cY_n$ has a net $\cY_n'$ of mesh size $\wt{c}_{\ref{p: spread vectors}}\rho \sqrt{np}$ with cardinality at most $O(\sqrt{p}/\rho) = \exp(O(\log n))$. Therefore, noting that for any $y_1, y_2 \in \R^n$
\[
\left| \inf_{x \in V_1 \setminus V_0} \|(A_n -p {\bm J}_n ) x - y_1\|_2 -  \inf_{x \in V_1 \setminus V_0} \|(A_n -p {\bm J}_n ) x - y_2\|_2 \right| \le \|y_1 - y_2\|_2,
\]
taking a union over $y \in \cY_n'$ we deduce from \eqref{eq:moderate-pre-prop} that
\beq\label{eq:moderate-pre-prop-1}
\P\left(\left\{ \inf_{x \in V_1\setminus V_0, y \in \cY_n} \|(A_n - {\bm J}_np)x - y\|_2 \le \wt{c}_{\ref{p: spread vectors}}\rho \sqrt{np}\right\} \cap \Omega_K^0\right) \le \exp(-\bar{c}_{\ref{p: spread vectors}}n).
\eeq
Since $y_0 -{\bm J}_n x \in \cY_n$ for all $x \in S^{n-1}$ we further note that
\beq\label{eq:moderate-pre-prop-2}
\inf_{x \in V_1\setminus V_0, y \in \cY_n} \|(A_n - {\bm J}_np)x - y\|_2 \le \inf_{x \in V_1 \setminus V_0} \|A_n x - y_0\|_2.
\eeq
This together with \eqref{eq:moderate-pre-prop-1} yields the desired conclusion. Thus, it remains to establish \eqref{eq:moderate-pre-prop}.

To this end, we fix any $x \in V_1\setminus V_0, y \in \R^n$ and write
\[
A_n:=\begin{bmatrix} A_n^{1,1} & A_n^{1,2}\\ A_n^{2,1} & A_n^{2,2} \end{bmatrix}, \quad x:=\begin{pmatrix} x_1 \\ x_2 \end{pmatrix}, \quad \text{ and } \quad y:=\begin{pmatrix} y_1 \\ y_2 \end{pmatrix},
\]
where $A_n^{1,1}$ and $A_n^{2,2}$ are $\gn \times \gn$ and $(n-\gn) \times (n-\gn)$ matrices, respectively, $A_n^{1,2}$, and $(A_n^{2,1})^*$ are $\gn \times (n-\gn)$ matrices, $x_1, y_1$ are vectors of length $\gn$, and $x_2,y_2$ are vectors of length $(n-\gn)$. Similarly we define $\{{\bm J}_n^{i,j}\}_{i,j=1}^2$. With these notations we see that
\begin{multline*}
\|(A_n- p {\bm J}_n) x -y\|_2^2 =\|(A_n^{1,1}- p {\bm J}_n^{1,1}) x_1 + (A_n^{1,2}- p {\bm J}_n^{1,2}) x_2 - y_1\|_2^2 \\
+ \|(A_n^{2,1}-p {\bm J}_n^{2,1}) x_1 + (A_n^{2,2}- p {\bm J}_n^{2,2}) x_2 - y_2\|_2^2.
\end{multline*}
Also note that by Assumption \ref{ass:matrix-entry} both $A_n^{1,2}$ and $A_n^{2,1}$ are matrices with i.i.d.~Bernoulli entries independent of $A_n^{1,1}$ and $A_n^{2,2}$, respectively. Since $x \notin V_0$ we have $\|x_{[m+1:n]}\|_\infty/ \|x_{[m+1:n]}\|_2 \le c_{\ref{l: sparse vectors-2}}^{-1}K m^{-1/2}$. Further, for $j \in [n]$, let us define
\[
x_1'(j):= x_{[m+1:n]}(j)\cdot \bI(j \in [\gn]), \quad \text{ and } x_2'(j):= x_{[m+1:n]}(j)\cdot \bI(j \notin [\gn]).
\]
Therefore there exists $i \in \{1,2\}$ such that $\|x_i'\|_2 \ge \|x_{[m+1:n]}\|_2/\sqrt{2}$. Without loss of generality let us assume $i=1$. This implies that
\(
\|x_1'\|_\infty/ \|x_1'\|_2 \le 2 c_{\ref{l: sparse vectors-2}}^{-1} K m^{-1/2} = 2 c_{\ref{l: sparse vectors-2}}^{-1} (c^\star_{\ref{lem: pattern}})^{-1/2} K \sqrt{p}
\).
Hence applying Lemma \ref{l: spread vector-levy} we see that for a sufficiently small  $\wt{c}_{\ref{p: spread vectors}}$, we have
\begin{align}\label{eq:levy-spread}
\P\left(\|(A_n - p {\bm J}_n) x - y\|_2 \le  4\wt{c}_{\ref{p: spread vectors}} \|x_{[m+1:n]}\|_2 \sqrt{np}\right) \le \cL\left(A_n^{2,1}x_1', 8 \wt{c}_{\ref{p: spread vectors}}\|x_1'\|_2 \sqrt{np}\right) \le \exp(-3 \ol{c} n),
\end{align}
for some $\ol{c}>0$.

To finish the proof we now use a $\vep$-net argument. Applying Fact \ref{fact:net-bd} we see that there exists $\cM \subset V_1 \setminus V_0$ such that for any $x \in V_1\setminus V_0$, there exists $\bar{x} \in \cM$ so that \eqref{eq:net-approx} holds. Thus
\begin{align}\label{eq:v-diff}
 \norm{v_x-v_{\bar{x}}}_2 \le  \norm{\f{v_x}{\norm{v_x}_2}-\f{v_{\bar{x}}}{\norm{v_{\bar{x}}}_2}}_2 \norm{v_{\bar{x}}}_2 + \norm{v_{{x}}}_2 \left| 1- \frac{\norm{v_{\bar{x}}}_2}{\norm{v_{{x}}}_2}\right| \le \vep +\tau \norm{v_{\bar{x}}}_2.
\end{align}
Since $\bar x \in V_1$ we also observe that
\beq\label{eq:r-v-relation}
\|r_{\bar{x}}\|_2\le c_{\ref{p: spread vectors}} K^{-1} \sqrt{M} \|r_{\bar{x}}\|_\infty \le 2c_{\ref{p: spread vectors}} K^{-1} \|v_{\bar{x}}\|_2,
\eeq
where the last inequality follows from the facts that the coordinates of $r_{\bar{x}}$ have smaller magnitudes than the non-zero coordinates of $v_{\bar{x}}$ and $m \le M/2$. Since $\bar x \notin V_0$, we have
\beq\label{eq:r-v-relation-2}
\|\bar x_{[m+1:n]}\|_2 = \sqrt{\|v_{\bar{x}}\|_2^2+\|r_{\bar{x}}\|_2^2} \ge \rho.
\eeq
Therefore, it follows from above that $\|v_{\bar{x}}\|_2 \ge \|r_{\bar{x}}\|_2$, whenever $c_{\ref{p: spread vectors}}$ chosen sufficiently small, which further implies that $\|v_{\bar{x}}\|_2 \ge \rho/\sqrt{2}$. Hence, choosing $\vep \le \rho/\sqrt{2}$, and proceeding similarly as in \eqref{eq:r-v-relation} we deduce
\beq\label{eq:r-v-relation-1}
\norm{r_x}_2 \le  2c_{\ref{p: spread vectors}} K^{-1} (\norm{v_{\bar{x}}}_2+\vep) \le 4c_{\ref{p: spread vectors}} K^{-1}  \norm{v_{\bar{x}}}_2.
   \eeq
 Further note that
 \[
 \|A_n - p {\bm J}_n\| \le \|A_n - \E A_n\| + \| \E A_n - p {\bm J}_n\| \le \|A_n - \E A_n\| +1,
 \]
 where the last step follows from Assumption \ref{ass:matrix-entry}. So, using the triangle inequality, \eqref{eq:net-approx}, \eqref{eq:v-diff}-\eqref{eq:r-v-relation}, and \eqref{eq:r-v-relation-1} we deduce
\begin{align}\label{eq:net-triangle}
\|(A_n - p {\bm J}_n) \bar{x} - y\|_2  & \le  \|(A_n- p {\bm J}_n) x - y\|_2 \notag\\
& \qquad \qquad \qquad + \norm{A_n- p {\bm J}_n}  \Big( \norm{u_x-u_{\bar{x}}}_2 +
           \norm{v_x-v_{\bar{x}}}_2 +\norm{r_x}_2 +\norm{r_{\bar{x}}}_2  \Big) \notag\\
           & \le  \|(A_n- p {\bm J}_n) x - y\|_2 + 4K \sqrt{np}\cdot  \vep + 2K \sqrt{np} \cdot \tau \cdot \|v_{\bar x}\|_2 + 12 c_{\ref{p: spread vectors}} \sqrt{np} \cdot  \norm{v_{\bar{x}}}_2.
\end{align}
Thus setting
\beq\label{eq:choose-vep-tau}
\vep=\f{c_{\ref{p: spread vectors}}\rho}{4 K} \quad \text{ and } \quad \tau= \f{c_{\ref{p: spread vectors}}}{2K},
\eeq
and shrinking $c_{\ref{p: spread vectors}}$ further, from \eqref{eq:levy-spread} and \eqref{eq:r-v-relation-2} we derive that
\begin{align}\label{eq:inf-to-union-moderate}
&\P\left(\exists {x} \in V_1\setminus V_0: \|(A_n- p {\bm J}_n) {x} -y \|_2 \le 2\wt{c}_{\ref{p: spread vectors}}\rho \sqrt{np} \right) \notag\\
\le \, &  \P\left(\exists \bar{x} \in \cM: \|(A_n- p{\bm J}_n) \bar{x} -y \|_2 \le 4 \wt{c}_{\ref{p: spread vectors}} \|\bar{x}_{[m+1:n]}\|_2 \sqrt{np} \right) \le |\cM| \cdot \exp(-3\bar{c} n).
\end{align}
With the above choices of $\vep$ and $\tau$, and any $M \le c^*_{\ref{p: spread vectors}} n$, from \eqref{eq:cM-bd} we have that
  \begin{align*}
  |\cM| \le  \bar{C}^{2M} \binom{n}{m}\binom{n}{M}\ \cdot \left(\f{1}{\rho}\right)^{m+1}  & \le \bar{C}^{2c^*_{\ref{p: spread vectors}} n}\Big(\f{en}{m}\Big)^m \Big(\f{en}{M}\Big)^M \left(\f{1}{\rho}\right)^{m+1} \\
  & \le ({e (c^\star_{\ref{lem: pattern}})^{-1}np})^{m} \Big(\bar{C}^2 \f{e}{c^*_{\ref{p: spread vectors}}}\Big)^{c^*_{\ref{p: spread vectors}} n}\left(\f{1}{\rho}\right)^{m+1} ,
   \end{align*}
 for some constant $\bar{C}$ depending only on $K$.  Recalling the definition of $\rho$ and $m$, it is easy to note that
 \[
m\log \left( \f{np}{\rho}\right) = o(n),
 \]
 for all $p$ satisfying $np\ge \f{\log n}{\sqrt{\log \log n}}$. This implies that for $c^*_{\ref{p: spread vectors}}$ sufficiently small, if $M \le c^*_{\ref{p: spread vectors}}n$ then we have $|\cM| \le \exp(\bar{c}n)$. In combination with \eqref{eq:inf-to-union-moderate}, this yields \eqref{eq:moderate-pre-prop}. The proof of the proposition is complete.
\end{proof}

\subsection{Invertibility over sparse vectors with a large spread component}\label{sec:sparse-3}
Combining Proposition \ref{l: sparse vectors-2} and Proposition \ref{p: spread vectors} we see that we have a uniform lower bound on $\|A_n x\|_2$ for $x \in V_1$ (recall the definition $V_1$ from \eqref{eq:def-V_1}) with $M= c^*_{\ref{p: spread vectors}}n$. As seen from the proof of Proposition \ref{p: spread vectors}, the positive constant $c^*_{\ref{p: spread vectors}}$ is  small. On the other hand, as we will see in Section \ref{sec:incomp}, to obtain a uniform lower bound on $\|A_n x\|_2$ over incompressible and non-dominated vectors $x$ in the case when $A_n$ is the adjacency matrix of a directed Erd\H{o}s-R\'{e}nyi graph, we first need to prove a uniform lower bound on the same for
$x \in V_{c^*, c}$, where
\beq\label{eq:V-c-*-c}
V_{c^*, c}:=\text{\rm Dom}(c^* n, c K^{-1}) \cup \text{\rm Comp}(c^* n, \rho),
\eeq
with the constant $c^*$ close to one (in fact $c^* >\f34$ will do) and $c>0$ some another constant. This is not immediate from Proposition \ref{p: spread vectors} and it will be the main result of this short section.

\begin{prop}\label{p: spread vectors-1}
Let A$_n$ be as in Assumption \ref{ass:matrix-entry}, $\rho$ as in Proposition \ref{l: sparse vectors-2}, and $p \ge c_1 \f{\log n}{n}$ for some constant $c_1 \in (0,1)$. Fix $K\ge 1$ and $c_0^* \in (0,1)$. Let $M_0:=\f{n \sqrt{\log \log n}}{\log n}$. Then there exist constants $0< c_{\ref{p: spread vectors-1}}, \wt{c}_{\ref{p: spread vectors-1}}, \bar{c}_{\ref{p: spread vectors-1}}< \infty$, depending only on $c_0^*$ and $K$, such that for any $y \in \R^n$ we have
\begin{multline}\label{eq:spread-prop-pre-bd}
\P\left(\left\{\exists x \in V_{c_0^*,{c}_{\ref{p: spread vectors-1}}}\setminus V_{M_0}: \norm{(A_n-p {\bm J}_n ) x -y}_2 \le 4\wt{c}_{\ref{p: spread vectors-1}}\|x_{[M_0+1:c_0^*n]}\|_2 \sqrt{np}\right\} \cap \Omega_K^0\right) \\
\le \exp(-2\bar{c}_{\ref{p: spread vectors-1}}n),
\end{multline}
 for all large $n$, where
  \beq\label{eq:V_M-0}
V_{M_0}:= \text{\rm Dom}(M_0, {c}_{\ref{p: spread vectors}} K^{-1}) \cup \text{\rm Comp}(M_0, \rho).
\eeq
Consequently, for any $y_0 \in \R^n$ we have
  \begin{multline}\label{eq:spread-prop-bd}
   \P\Big(\exists x \in V_{c_0^*, c_{\ref{p: spread vectors-1}}}\setminus V_{M_0} \text{ such that } \norm{A_n x -y_0}_2 \le \wt{c}_{\ref{p: spread vectors-1}}\rho \sqrt{np} \\
    \text{ and } \norm{A_n- \E A_n} \le K \sqrt{np}\Big) \le \exp(-\bar{c}_{\ref{p: spread vectors-1}}n).
  \end{multline}
\end{prop}

\vskip10pt
As the set  of sparse vectors that have a large spread component  has a higher metric entropy compared to that of the set of vectors considered in Section \ref{sec:sparse-2}, the small ball probability estimate derived in Lemma \ref{l: spread vector-levy} will be insufficient to accommodate a union bound.
To obtain a useful bound on the small ball probability we use the following result. Before stating the lemma let us introduce a notation: for any $v \in \R^n$ and $J \subset [n]$ we write $v_J$ to denote the vector in $\R^n$ obtained from $v$ by setting $v_i=0$ for all $i \in J^c$.
\begin{lem}[Bound on L\'{e}vy concentration function]\label{lem:bound-levy}
Let $v \in \R^n$ be a fixed vector and ${\bm x} \in \R^n$ be a random vector with i.i.d.~$\dBer(p)$ for some $p \in (0,1)$. Then there exists an absolute constant $C_{\ref{lem:bound-levy}}$ such that for every $\vep >0$ and $J \in [n]$,
\begin{align*}
\cL\left(\langle {\bm x}, v\rangle, p^{1/2}(1-p)^{1/2}\|v_J\|_2 \vep  \right) & \le \cL\left(\langle {\bm x}_J, v_J\rangle, p^{1/2}(1-p)^{1/2}\|v_J\|_2 \vep  \right)\\
& \le C_{\ref{lem:bound-levy}} \left( \vep + \f{\|v_J\|_{\infty}}{{p^{1/2}(1-p)^{1/2}}\|v_J\|_2}\right).
\end{align*}
\end{lem}
\vskip10pt

 The proof of Lemma \ref{lem:bound-levy} is a simple consequence of the well known Berry-Ess\'{e}en theorem and is similar to that of \cite[Proposition 3.2]{LPRT}. Hence further details are omitted.

To utilize the bound from Lemma \ref{lem:bound-levy} we recall that any vector belonging to the third set has a large spread component. This means that one can find a $J \subset [n]$ such that $\|v_J\|_\infty/\|v_J\|_2$ is small with the Euclidean norm of $v_J$ being not too small. 

The proof of Proposition \ref{p: spread vectors-1} is similar to that of Proposition \ref{p: spread vectors}. Recall a key to the proof of Proposition \ref{p: spread vectors} is the anti-concentration bound of Lemma \ref{l: spread vector-levy} where the latter is a consequence of Paley-Zygmund inequality (see the proof of \cite[Corollary 3.7]{BR-invertibility}). To prove Proposition \ref{p: spread vectors-1} we need a better anti-concentration bound.
To this end, we note that any $x \notin V_{M_0}$ has a large spread component, i.e.~a large non-dominated part. It allows us to use Lemma \ref{lem:bound-levy} instead of Paley-Zygmund inequality. For matrices with independent rows, this together with standard tensorization techniques produces a sharp enough anti-concentration probability bound suitable for the proof of Proposition \ref{p: spread vectors-1}. For matrices satisfying Assumption \ref{ass:matrix-entry} we additionally need to show that one can find a sub-matrix of $A_n$ with jointly independent entries, such that the coordinates of $x$ which correspond to the columns of this sub-matrix  form a vector with a large spread component and a sufficiently large norm to carry out the scheme described above. Since the proof of Proposition \ref{p: spread vectors-1} is an adaptation of that of Proposition \ref{p: spread vectors} with these couple of modifications it is deferred to Appendix \ref{sec:invert-large-spread}.

\begin{rmk}\label{rmk:wt-V}
Proposition \ref{p: spread vectors-1} shows that $\|A_n x\|_2$ has uniform lower bound when $x \in V_{c_0^*, c_{\ref{p: spread vectors-1}}}\setminus V_{M_0}$. Its proof reveals that the same bound continues to hold when $V_{M_0}$ is replaced by
\[
\bar V_{M_0}:= {\rm Dom}(M_0, c K^{-1}) \cup {\rm Comp}(M_0,\rho),
\]
for some small constant $c < {c}_{\ref{p: spread vectors}}$. Changing the constant ${c}_{\ref{p: spread vectors}}$ to $c$ only shrinks the constants $c_{\ref{p: spread vectors-1}}, \wt c_{\ref{p: spread vectors-1}}, \bar c_{\ref{p: spread vectors-1}}$. We will use this generalization in the proof of Lemma \ref{lem:non-dominated-J}.
\end{rmk}

\begin{rmk}\label{rmk:prop-non-square}
Propositions \ref{l: sparse vectors-2}, \ref{p: spread vectors}, and \ref{p: spread vectors-1} have been proved for $n \times n$ matrices.  It can be checked that the conclusions of these propositions continue to hold for $(n-1) \times n$ matrices, with slightly worse constants. In particular, they hold for the matrix $\wt A_n$ such that its rows are any $(n-1)$ columns of the matrix $A_n$ satisfying Assumption \ref{ass:matrix-entry}. We will need this generalization to prove the desired lower bound on the smallest singular value of the adjacency matrix of a random bipartite graph or equivalently for the random matrix with i.i.d.~Bernoulli entries (as noted in Remark \ref{eq:iid-bipartite}). To keep the presentation of this paper simple we refrain from providing the proof for this generalization. It follows from a simple adaptation of the proof of the same for square matrices.
\end{rmk}

\subsection{Structure of $A_n^{-1} u$}\label{sec:kernel-unstructured}
As mentioned in Section \ref{sec:prelim}, to deduce invertibility over non-dominated and incompressible vectors we also need to show that, given any $u \in \R^n$, the random vector $A_n^{-1} u$ must be non-dominated and incompressible with high probability. Since we will apply this result with coordinates of $u$ being i.i.d.~$\dBer(p)$, we may and will assume that $u$ does not have a large support. With some additional work, the results of Sections \ref{sec:sparse-1}-\ref{sec:sparse-3} yield this.

Moreover, as we will see in Section \ref{sec:incomp}, to treat the non-dominated and incompressible vectors when $A_n$ is the adjacency matrix of a directed Erd\H{o}s-R\'{e}nyi graph, we further need to establish that
given any $J \subset [n]$ with $|J| \approx \f{n}{2}$, one can find $I \subset J$ such that the vector $(A_n^{-1} u)_I$ contains a considerable proportion the non-dominated and incompressible components of the random vector $A_n^{-1} u$. The proof of the latter crucially uses Proposition \ref{p: spread vectors-1}. These two results are the content of this section.

We first begin with the corollary which shows that $A_n^{-1} u$ is neither compressible nor dominated with high probability.


\begin{cor}\label{cor:combine}
Let A$_n$ be as in Assumption \ref{ass:matrix-entry}, where $p$ satisfies the inequality
   \[
   np \ge \log(1/p),
   \]
and $\rho$ be as in Proposition \ref{l: sparse vectors-2}. Fix $K\ge 1$, $c_0^* \in (0,1)$, and $y_0 \in \R^n$ such that $|\supp(y_0)| \le C_\star np$ for some $C_\star >0$. Then there exist constants $0< \wt{c}_{\ref{cor:combine}}, \bar{c}_{\ref{cor:combine}}< \infty$, depending only on $c_0^*$ and $K$, such that
  \begin{align*}
   \P&\Big(\Big\{\exists x \in \R^n \text{ such that } {x}/{\|x\|_2} \in V_{c_0^*, c_{\ref{p: spread vectors-1}}}, \, \norm{A_n x -y_0}_2 \le \wt{c}_{\ref{cor:combine}}\rho \sqrt{np}\cdot  \|x\|_2 \\
   &\hskip 3in \text{ and } \norm{A_n- \E A_n} \le K \sqrt{np}\Big\} \cap \Omega_0^c\Big) \le n^{-\bar{c}_{\ref{cor:combine}}},
  \end{align*}
  where we recall the definition of $V_{c_0^*, c_{\ref{p: spread vectors-1}}}$ from \eqref{eq:V-c-*-c} and the definition of $\Omega_0$ from \eqref{eq:Omega-0}.
  \end{cor}

\begin{proof}
Recalling the definition of $V_0$ from \eqref{eq:V-0}, we first show that
 \begin{multline}\label{eq:bound-on-V-0-new}
   \P\left(\left\{\exists x \in \R^n \text{ such that } {x}/{\|x\|_2} \in V_0, \, \norm{A_n x -y_0}_2 \le \wt{c}_{\ref{cor:combine}}\rho \sqrt{np}\cdot  \|x\|_2 \right\} \cap \Omega_K^0 \cap \Omega_0^c\right)\\ \le n^{-2\bar{c}_{\ref{cor:combine}}},
  \end{multline}
  for all large $n$, where we recall the definition of $\Omega_K^0$ from \eqref{eq:Omega-K-0}. We remind the reader that to prove Proposition \ref{l: sparse vectors-2} we defined $\bar A_n$ to be the sub-matrix of $A_n$ obtained upon deleting the rows in $\supp(y_0)$ and showed that $\|\bar A_n x\|_2$ is uniformly bounded below, with high probability, for all $x \in V_0$. As $\|\bar A_n x\|_2 \le \|A_n x - y_0\|_2$ this yielded the desired result. Since the proof does not involve $y_0$, except for the cardinality of its support, we therefore can carry out the exact same steps and use the bound on the probability of $\Omega_{\ref{lem: typical structure}}^c$, derived in Lemma \ref{lem: typical structure}, to obtain \eqref{eq:bound-on-V-0-new}.

It remains to show that
 \begin{multline}\label{eq:bound-on-V-1-new}
   \P\left(\left\{\exists x \in \R^n \text{ such that } {x}/{\|x\|_2} \in V_{c_0^*, c_{\ref{p: spread vectors-1}}}\setminus V_0, \, \norm{A_n x -y_0}_2 \le \wt{c}_{\ref{cor:combine}}\rho \sqrt{np}\cdot  \|x\|_2 \right\} \cap \Omega_K^0\right) \\
   \le \exp(-\bar{c} n),
  \end{multline}
  for some $\bar c >0$. If there exists an $x \in \R^n$ such that $\|A_n x - y_0\| \le \wt{c}_{\ref{cor:combine}}\rho \sqrt{np}\cdot  \|x\|_2$, then using triangle inequality we find that
 \[
 \frac{\|y_0\|_2}{\|x\|_2} \le \|A_n -p {\bm J}_n\|  + p \|{\bm J}_n\|+ \f{\|A_n x - y_0\|_2}{\|x\|_2} \le 2K np,
 \]
  on $\Omega_K^0$. Further let us recall that for any $x \in S^{n-1}$, we have $p {\bm J}_n x = \lambda {\bm 1}$ for some $\lambda \in \R$ with $|\lambda| \le \sqrt{n} p$. Therefore, using triangle inequality once more we see that
 \begin{multline*}
 \left\{ \exists x \in \R^n \text{ such that } {x}/{\|x\|_2} \in V_{c_0^*, c_{\ref{p: spread vectors-1}}}\setminus V_0, \, \norm{A_n x -y_0}_2 \le \wt{c}_{\ref{cor:combine}}\rho \sqrt{np}\cdot  \|x\|_2 \right\}\cap \Omega_K^0 \\
 \subset \left\{\inf_{y \in \cY_\star} \inf_{x \in V_{c_0^*, c_{\ref{p: spread vectors-1}}}\setminus V_0} \norm{(A_n - p {\bm J}_n) x - y}_2 \le \wt{c}_{\ref{cor:combine}}\rho \sqrt{np}\right\} \cap \Omega_K^0,
 \end{multline*}
 where
 \[
 \cY_\star:= \left\{ \gamma \cdot \f{y_0}{\|y_0\|_2} + \lambda {\bf 1}; \ \gamma, \lambda \in \R \text{ with } |\gamma|\le 2K {np}, |\lambda| \le \sqrt{n} p  \right\}.
 \]
 Since $\cY_\star$ admits a net $\cN_\star$ of mesh size $\wt{c}_{\ref{cor:combine}}\rho \sqrt{np}$ with cardinality at most $(8K\sqrt{np}/\wt{c}_{\ref{cor:combine}}\rho)^2$, by a union bound we see that it suffices to show that
 \beq\label{eq:bound-on-V-1-new-1}
  \P\left(\left\{\exists x \in V_{c_0^*, c_{\ref{p: spread vectors-1}}}\setminus V_0, \, \text{ such that } \norm{(A_n - p {\bm J}_n) x -y}_2 \le 2 \wt{c}_{\ref{cor:combine}}\rho \sqrt{np} \right\} \cap \Omega_K^0\right) \le \exp(-2\bar{c} n),
 \eeq
 for any $y \in \R^n$. Arguing similarly as in \eqref{eq:r-v-relation}, we note that for any $x \in {\rm Dom}(c_0^*n, c_{\ref{p: spread vectors-1}}K^{-1})$
 \(
 \|x_{[M_0+1: c_0^*n]}\|_2 \ge \|x_{[c_0^*n+1:n]}\|_2,
 \)
and hence, for $x \in {\rm Dom}(c_0^*n, c_{\ref{p: spread vectors-1}}K^{-1})\setminus V_{M_0}$ we obtain that $\|x_{[M_0+1: c_0^*n]}\|_2 \ge \rho/\sqrt{2}$. Therefore, \eqref{eq:bound-on-V-1-new-1} follows from \eqref{eq:moderate-pre-prop} and \eqref{eq:spread-prop-pre-bd}. This yields \eqref{eq:bound-on-V-1-new} and combining this with \eqref{eq:bound-on-V-0-new} now finishes the proof of the corollary.
\end{proof}

Building on Corollary \ref{cor:combine} we now prove that for any  $J \subset [n]$ with $|J| \approx \f{n}{2}$, there exists a large set $I \subset J$ such that $(A_n^{-1} u)_I$ has non-dominated tails and a substantial Euclidean norm.


\begin{lem}\label{lem:non-dominated-J}
Let A$_n$ be as in Assumption \ref{ass:matrix-entry}, where $p$ satisfies the inequality
   \[
   np \ge \log(1/p).
   \]
Fix $K \ge 1$, $J \subset [n]$ such that $\f{3n}{8} \le |J| \le \f{5n}{8}$, and $y_0 \in \R^n$ with $\|y_0\|_2 \in [1, \bar C np]$ and $|\supp(y_0)| \le C_\star np$ for some $\bar C, C_\star >0$. Then there exist constants $0< c_{\ref{lem:non-dominated-J}}, \bar c_{\ref{lem:non-dominated-J}} < \infty$, depending only on $K$, such that
\begin{multline*}
\P\left(\left\{\exists x \in \R^n: A_n x =y _0 \text{ and either } \f{\|x_{[\f{n}{4}+1:n]\cap J}\|_\infty}{\|x_{[\f{n}{4}+1:n]\cap J}\|_2} \ge \f{1}{c_{\ref{lem:non-dominated-J}} \sqrt{n}} \text{ or } \frac{\|x_{[\f{n}{4}+1:n]\cap J}\|_2}{\norm{x}_2} \le \rho \right\} \cap \Omega_K^0 \cap \Omega_0^c \right)\\
 \le n^{-\bar{c}_{\ref{lem:non-dominated-J}}},
\end{multline*}
for all large $n$.
\end{lem}

\begin{proof}
Let $x \in \R^n$ be such that $A_n x =y_0$.
Let us show first that the event $ \|x_{[\f{n}{4}+1:n]\cap J}\|_2 \le \rho$ can occur with probability at most $n^{-\bar c}$ for some constant $\bar c >0$.
Since $|J| \ge \f{3n}{8}$, we have
\[
\|x_{J\cap[\f{n}{4}+1:n]}\|_2 \ge \|x_{[\f{7}{8}n+1: n]}\|_2,
\]
where by a slight abuse of notation, for $m < m' <n$, we write
\[
x_{J \cap[m:m']}(i) = x_{[m;m']}(i) \cdot {\bm 1}(i \in J).
\]
{Applying Corollary \ref{cor:combine} with $c_0^*=\f78$ we note that $\|x_{[\f{7}{8}n+1: n]}\|_2 {\le \rho \norm{x}_2}$ with probability at least $1 - n^{-\bar c}$. Thus our claim on the upper bound on  $ \|x_{[\f{n}{4}+1:n]\cap J}\|_2$ follows.}

Next,  assume that
\beq\label{eq:x-non-dominated-J}
\left\| x_{J \cap [\f{n}{4}+1:n]}\right\|_2  \le c_{\ref{lem:non-dominated-J}} \sqrt{n} \left\| x_{J \cap [\f{n}{4} +1:n]}\right\|_\infty.
\eeq
We will prove that if $c_{\ref{lem:non-dominated-J}}$ is chosen sufficiently small then this can hold only on a set of small probability as well. This will complete the proof.

Denote $w:= x_{J^c \cup [1: \frac{n}{4}]}$ and $z := x - w  = x_{J \cap [\frac{n}{4}+1:n]}$.
Then $w \neq 0$ and the assumption $|J| \ge \f{3n}{8}$ implies that  $w \in  {\rm Sparse}({7n}/{8})$.
We will show that the vector $A_n w/\norm{w}_2$ is close to some set $\cY$ having a small $\varepsilon$-net. As $w /\|w\|_2 \in {\rm Sparse}({7n}/{8}) \cap S^{n-1}$, the desired probability estimate will then follow from Proposition \ref{p: spread vectors-1} and the union bound over the net.

Turning to carry out the above task, we note that the inequality \eqref{eq:x-non-dominated-J} shows that
\[
\|z\|_2=\left\| x_{J \cap [\f{n}{4}+1:n]}\right\|_2 \le 3 c_{\ref{lem:non-dominated-J}} \norm{x_{[M_0+1:\frac{n}{4}]}}_2,
\]
where we recall $M_0:= \f{n\sqrt{\log \log n}}{\log n}$.
 Since $A_n x =y_0$,  this implies
\[
\|A_n w - y_0+ p{\bm J}_n  z \|_2 = \| (A_n - p{\bm J}_n) z\|_2 \le \| A_n - p {\bm J}_n \| \cdot \|z\|_2 \le 6 c_{\ref{lem:non-dominated-J}} K \sqrt{np} \cdot \norm{x_{[M_0+1:\frac{n}{4}]}}_2,
\]
on the event $\Omega_K^0$, where we recall its definition from \eqref{eq:Omega-K-0}.

For ease of writing let us denote $w^\star:=w/\norm{w}_2$ and set $y^\star:=(y_0- p{\bm J}_n  z)/\norm{w}_2$.
With this notation, the previous inequality reads
\beq\label{eq:w-bd-2}
\|A_n w^\star - y^\star\|_2
\le 6 c_{\ref{lem:non-dominated-J}} K \sqrt{np} \cdot \f{\|x_{[M_0+1:\f{n}{4}]}\|_2}{\|w\|_2}\\
 =  6 c_{\ref{lem:non-dominated-J}} K \sqrt{np} \cdot \norm{w^\star_{[M_0+1:\f{n}{4}]}}_2,
\eeq
where we used that  $w_{[M_0+1:n/4]}=x_{[M_0+1:n/4]}$ to derive the last equality.

The inequality \eqref{eq:w-bd-2} already shows that $A_n w/\|w\|_2$ is close to $y^\star$. From the definition of $y^\star$ we further note that $y^\star = \lambda y_0 + \gamma {\bm 1}$ for some $\lambda, \gamma \in \R$. This indicates that the natural choice for the set $\cY$ is the collection of all vectors of the form $ \lambda y_0 + \gamma {\bm 1}, \lambda, \gamma \in \R$. To show that $\cY$ admits a net of small cardinality we need bounds on $\lambda$ and $\gamma$.

We claim that $y^\star \in \cY$, where
\[
\cY:= \{y \in \R^n: y = \lambda y_0 + \gamma {\bm 1}, \text{ for some } \lambda \in (0, 4Knp] \text{ and } \gamma \in [-3\sqrt{n}p, 3\sqrt{n}p]\}.
\]
To see this we observe that the assumption $A_n x=y_0$ implies that
\[
\|x\|_2 \ge \f{\|A_n x\|_2}{\|A_n\|} \ge \f{\|y_0\|_2}{2K np} \ge \f{1}{2K np},
\]
as $\|y_0\|_2 \ge 1$ and $\|A_n\| \le 2K np$ on the event $\Omega_K^0$. Therefore,
\[
\|w\|_2 \ge \|x_{[1:\frac{n}{4}]}\|_2 \ge \f12 \|x\|_2 \ge 1/(4K np).
\]
From this and the inequality $\|z\|_2 \le 2 \|w\|_2$ the required claim follows.

Since $\|y_0\|_2 \le \bar C np$ it is also immediate that the set $\cY$ admits a $(\wt{c}_{\ref{p: spread vectors-1}}\rho \sqrt{np})$-net $\mathcal{N}$ of cardinality at most $O(((np)^2/\rho)^2)$.

We next claim that $w^\star \notin \wt V_{M_0}$, with high probability, where
\beq\label{eq:wt-V}
\wt V_{M_0}:= {\rm Dom} (M_0, \f{{c}_{\ref{p: spread vectors}}}{3}K^{-1}) \cup {\rm Comp}(M_0, \rho).
\eeq
Proving \eqref{eq:wt-V} will put us in a position to apply Proposition \ref{p: spread vectors-1}.

To this end, using Corollary \ref{cor:combine}, we can assume that $x/\|x\|_2 \notin V_{15/16,c_{\ref{p: spread vectors-1}}}$, with high probability. Hence, recalling the definition of $w$ and using the monotonicity of the non-zero coordinates of $x_{[M_0+1:n]}$ we have
\beq\label{eq:w-comp-bd}
\f{\|w_{[M_0+1:n]}\|_2}{\|w\|_2} \ge \f12 \cdot \f{\|x_{[n/8+1:n/4]}\|_2}{\|x\|_2} \ge \f{\|x_{[15n/16+1:n]}\|_2}{\|x\|_2} \ge \rho.
\eeq
Moreover, $w_{[M_0+1:n/4]}=x_{[M_0+1:n/4]}$, and $\|x_{[M_0+1:n/4]}\|_2 \ge \frac{1}{3} \|x_{[M_0+1:n]}\|_2$, so
\beq\label{eq:w-dom-bd}
\f{\|w_{[M_0+1:n]}\|_\infty}{\|w_{[M_0+1:n]}\|_2} \le \f{\|x_{[M_0+1:n]}\|_\infty}{\|x_{[M_0+1:n/4]}\|_2} \le 3 \cdot \f{\|x_{[M_0+1:n]}\|_\infty}{\|x_{[M_0+1:n]}\|_2} \le \f{3 K}{{c}_{\ref{p: spread vectors}} \sqrt{M_0}},
\eeq
where we applied Corollary \ref{cor:combine} again in assuming that $x/\|x\|_2 \notin V_{M_0}$, with high probability, with $V_{M_0}$ as in \eqref{eq:V_M-0}. Inequalities \eqref{eq:w-comp-bd} and \eqref{eq:w-dom-bd}  confirm that $w^\star \notin \wt V_{M_0}$ with high probability.

Now, setting $c_0^*=\f78$ in Proposition \ref{p: spread vectors-1} (see also Remark \ref{rmk:wt-V}), combining  \eqref{eq:spread-prop-pre-bd} with the union bound over the net $\mathcal{N} \subset \cY$, and applying triangle inequality we  derive that
\begin{multline*}
\P\left( \left\{\inf_{y \in \cY} \inf_{w^\star \in {\rm Sparse}({7n}/{8}) \cap S^{n-1}\setminus \wt V_{M_0}} \|(A_n -p {\bm J}_n) w_\star - y\|_2 \le 4 \wt{c}_{\ref{p: spread vectors-1}}\|w^\star_{[M_0+1:\f{7n}{8}]}\|_2 \sqrt{np} \right\} \cap \Omega_K^0 \right) \\
\le \exp(-\bar{c}_{\ref{p: spread vectors-1}}n).
\end{multline*}
Thus recalling \eqref{eq:w-bd-2}, and as $y^\star \in \cY$, we see that for a sufficiently small $c_{\ref{lem:non-dominated-J}}$, the inequality \eqref{eq:x-non-dominated-J} can hold only on a set of small probability.
This finishes the proof of the lemma.
\end{proof}

 \section{Invertibility over incompressible and non-dominated vectors}\label{sec:incomp}
 In this section our goal is to obtain a uniform lower bound on $\|A_n x \|_2$ over non-dominated and incompressible
 vectors $x$, with large probability. 
 As the set of such vectors possesses a large metric entropy, one cannot replicate the approach of Section \ref{sec:inv-dom-comp}. As outlined in Section \ref{sec:prelim}, we find a uniform lower bound over the set of such vectors by relating it to the average of the distance of a column of $A_n$ from the subspace spanned by the rest of the columns. To this end, we use the following Lemma from \cite{RV1} (see Lemma 3.5 there).
 \begin{lem}[Invertibility via distance]  \label{l: via distance}
     For $j \in [n]$, let $\tilde{A}_{n,j} \in \R^n$ be the $j$-th column of $\tilde{A}_n$, and let $\tilde H_{n,j}$ be the subspace of $\R^n$ spanned by $\{\tilde{A}_{n,i}, i \in [n]\setminus\{j\}\}$. Then for any $\vep, \rho>0$, and $M<n$,
   \beq\label{eq:invertibility_distance}
    \P\left(  \inf_{x \in  \text{\rm Incomp}(M, \rho)} \norm{\tilde{A}_nx}_2 \le \vep \rho^3 \sqrt{\frac{p}{n}} \right)
    \le \frac{1}{M} \sum_{j=1}^n \P \left( \dist(\tilde{A}_{n,j},\tilde H_{n,j}) \le \rho^2\sqrt{p} \vep \right).
   \eeq
 \end{lem}
 \begin{rmk}\label{rmk:l via distance}
 Lemma \ref{l: via distance} can be extended to the case when the event in the \abbr{LHS} of \eqref{eq:invertibility_distance} is intersected with an event $\Omega$. In that case Lemma \ref{l: via distance} continues to hold if the \abbr{RHS} of \eqref{eq:invertibility_distance} is replaced by intersecting each of the event under the summation sign with the same event $\Omega$. In the proof of Theorem \ref{thm:s-min-graphs}, we will use this slightly more general version of Lemma \ref{l: via distance}. Since the proof of this general version of Lemma \ref{l: via distance} is a straightforward adaptation of the proof of \cite[Lemma 3.5]{RV1}, we omit the details.
 \end{rmk}
 Lemma \ref{l: via distance} shows that it is enough to find bounds on $\dist(A_{n,j}, H_{n,j})$ for $j \in [n]$, where $A_{n,j}$ is the $j$-th column of $A_n$ and $H_{n,j}$ is the subspace spanned by the rest of the columns. Furthermore, from the assumption on the entries of $A_n$ it follows that one only needs to consider $j=1$. For $j \in [n]\setminus\{1\}$ one can essentially repeat the same argument.

 For  a matrix $A_n$ of i.i.d.~Bernoulli entries, the first column $A_{n,1}$ is independent of $H_{n,1}$, so the desired bound on the distance  essentially follows from Berry-Ess\'{e}en theorem (see Lemma \ref{lem:bound-levy}), 
 upon showing that any vector in the kernel of a random matrix must be both non-dominated and incompressible. This easier case is therefore deferred to Section \ref{sec:proof-main-thm} and is dealt with during the course of the proof of Theorem \ref{thm:s-min-graphs}. Here we will only obtain a bound on $\dist(A_{n,1}, H_{n,1})$ when $A_n$ is the adjacency matrix of either a directed or a undirected Erd\H{o}s-R\'{e}nyi graph.

To obtain a bound on $\dist(A_{n,1}, H_{n,1})$ we derive an alternate expression for the same which is more tractable. This is done in the following extension of  \cite[Proposition 5.1]{V}.

 \begin{prop}[Distance via quadratic forms]\label{prop:distance-quadratic}
 Let $ \tilde A_n, \tilde A_{n,j}$ and $\tilde H_{n,j}$ be as in Lemma \ref{l: via distance}. Denote by $C_n$  the $(n-1)\times (n-1)$ sub-matrix of $\tilde A_n^{{\sf T}}$ obtained after removing the first row and column of $\tilde A_n$. Furthermore, let ${\bm x}^{\sf T}, {\bm y} \in \R^{n-1}$ denote the first row and column of $\tilde A_n$ with the first entry $a_{11}$ removed, respectively. Then we have the following:

\begin{enumeratei}

\item If $C_n$ is non-invertible then
\[
\dist(\tilde A_{n,1}, \tilde H_{n,1}) \ge \sup_{v \in {\rm Ker}(C_n) \cap S^{n-1}} |\langle {\bm y}, v\rangle |,
\]
where ${\rm Ker}(C_n):= \{u \in \R^{n-1}: C_n u =0\}$.

\item If $C_n$ is invertible then
 \beq\label{eq:dist-quad}
 \dist(\tilde A_{n,1},\tilde H_{n,1}) = \f{\left|\langle C_n^{-1} {\bm x}, {\bm y}\rangle-a_{11} \right|}{\sqrt{1+\left\|C_n^{-1}{\bm x}\right\|_2^2}}.
 \eeq

 \end{enumeratei}
 \end{prop}

  \begin{proof}
 It follows from the definition that
 \[
 \dist(\tilde A_{n,1}, \tilde H_{n,1}) \ge \sup_{{\bm h}} |\langle \tilde A_{n,1}, {\bm h} \rangle |,
 \]
 where the supremum is taken over all unit vectors ${\bm h}$ that are normal to the subspace $\tilde H_{n,1}$. To prove part (i) we only need to show that if $v \in {\rm Ker}(C_n)$ then the vector $\binom{0}{v}$ is a vector normal to $\tilde H_{n,1}$. This is immediate from the definition of $C_n$ and $\tilde H_{n,1}$.

 When $C_n$ is invertible and $\tilde A_n$ is a symmetric matrix,  Proposition \ref{prop:distance-quadratic} was proved in \cite{V}. A simple adaptation of this proof yields Proposition \ref{prop:distance-quadratic} for any square matrix $\tilde A_n$. We omit the details here.
 \end{proof}


From Proposition \ref{prop:distance-quadratic} we see that the relevant distance
has two different expressions depending on whether the $(n-1)\times (n-1)$ sub-matrix of $A_n$ obtained upon removing the first row and column is invertible or not.
In the latter case one can again use Lemma \ref{lem:bound-levy} to deduce the desired bound. Hence the treatment of that case is postponed to Section \ref{sec:proof-main-thm}.

 Thus the main technical result of this section is the following. For ease of writing we formulate and prove the relevant result for $(n+1) \times (n+1)$ matrices. With no loss of generality this extends to $n \times n$ matrices, possibly with slightly worse constants.

 \begin{prop}[Distance bound]\label{prop:dist-bd}
Let ${\sf A}_n$, a matrix of size $(n+1)\times (n+1)$, be the adjacency matrix of either a directed or an undirected Erd\H{o}s-R\'{e}nyi graph. Let $A_n$ be the $n \times n$ sub-matrix of ${\sf A}_n$ obtained upon deleting the first row and column of ${\sf A}_n$. Denote ${\bm x}^{\sf T}$ and ${\bm y}$ to be the first row and column of ${\sf A}_n$ with the first common entry removed. Define
\[
\Omega_+:=\{A_n \text{ is invertible}\},
\]
and fixing $K \ge 1$ we let
\[
\Omega_K^0:=\left\{\|A_n -  \E A_n \| \le K \sqrt{np}\right\}.
\]
Then there exist an absolute constant $c_{\ref{prop:dist-bd}}$ and another large constant $C_{\ref{prop:dist-bd}}$, depending only on $K$, such that for any $u \in \R$ and $\vep >0$ we have
\beq\label{eq:dist-bd}
\P \left(\left\{\f{\left| \langle A_n^{-1} {\bm x}, {\bm y}\rangle -u\right|}{\sqrt{1+\|A_n^{-1} {\bm x}\|_2^2}} \le c_{\ref{prop:dist-bd}}\vep \rho^2 \sqrt{p} \right\} \cap \Omega_K^0 \cap \Omega_+  \right) \le  \vep^{1/5} + \f{C_{\ref{prop:dist-bd}}}{\sqrt[4]{np}}.
\eeq
\end{prop}

\begin{rmk}\label{eq:sharper-bd}
It is believed that the optimal exponent of $\vep$ in the \abbr{RHS} of \eqref{eq:dist-bd} is one. As ${\bm x}$ and ${\bm y}$ are not independent, to obtain a bound on the probability of the event on the \abbr{LHS} of \eqref{eq:dist-bd}, we need to use a decoupling argument (see Lemma \ref{lem:decoupling} below). Even in the case of independent ${\bm x}$ and ${\bm y}$, to apply Lemma, \ref{lem:bound-levy} one still needs to replace the denominator by some constant multiple of $\|A_n^{-1} {\bm x}\|_2$. This amounts to showing that  $\|A_n^{-1} {\bm x}\|_2 \ge c$ for some $c >0$. As the entries of $A_n$ have a non-zero mean this poses an additional technical difficulty. These two steps together result a sub-optimal exponent of $\vep$ in the \abbr{RHS} of \eqref{eq:dist-bd}.

It is further believed that the second term in the probability bound of \eqref{eq:dist-bd} can be improved to $\exp(-\bar c np)$ for some $\bar c>0$. To improve this bound, one needs to obtain a strong estimate of the L\'{e}vy concentration function of  $A_n v$  for $v \in S^{n-1}$. Such an estimate is impossible for a vector with rigid arithmetic structure. On the other hand the set of such vectors has a low metric entropy. Therefore, one needs to show that this metric entropy precisely balances the estimate on the L\'{e}vy concentration function. Putting these two pieces together, the desired better bound on the probability was obtained in \cite{BR-invertibility} for sparse matrices with i.i.d.~entries, for $np \ge C \log n$, for some large $C>1$, and in \cite{wei} for symmetric sparse matrices when $p \ge n^{-c}$ for some $c \in (0,1)$. To achieve the same here for all $p$ satisfying $np \ge \log(1/p)$ one requires new ideas. We refrain from pursuing this direction.
\end{rmk}

\begin{rmk}
We point out to the reader that results analogous to Proposition \ref{prop:dist-bd} were used in \cite{V} and \cite{wei} to control the invertibility over incompressible vectors for dense and sparse symmetric random matrices, respectively. Here to prove Proposition \ref{prop:dist-bd} we encounter additional technical difficulties to tackle the adjacency matrix of the directed Erd\H{o}s-R\'{e}nyi graph and also to handle the non-zero mean assumption on the entries.
\end{rmk}

\vskip10pt
Before proceeding to the proof of Proposition \ref{prop:dist-bd} let us describe the idea behind it. We note that if ${\bm x}$ and ${\bm y}$ were independent vectors with i.i.d.~Bernoulli entries and if the vector $A_n^{-1}{\bm x}$ was neither dominated nor compressible then, on the event that $A_n$ is invertible, the probability of the event
\beq\label{eq:non-dom-levy-form}
\left\{\f{|\langle A_n^{-1} {\bm x}, {\bm y} \rangle - u |}{\|A_n^{-1} {\bm x}\|_2} \le \vep \rho^2 \sqrt{p}\right\}
\eeq
would have been a consequence of Lemma \ref{lem:bound-levy}. Therefore, applying Proposition \ref{prop:distance-quadratic}(ii) we see that it is enough to reduce the \abbr{RHS} of \eqref{eq:dist-quad} to an expression similar to the above. This consists of several critical steps. The first is a decoupling argument. This is done via the following lemma.

 \begin{lem}[Decoupling]\label{lem:decoupling}
 Fix any $n \times n$ matrix $B_n$. Suppose ${\bm z}$ 
 and $\hat {\bm z}$ 
 are random vectors of length $n$, with independent coordinates but not necessarily independent of each other. Further assume that for every $ J \subset [n]$, ${\bm z}_J$ is independent of $\hat {\bm z}_{J^c}$. Let $({\bm z}', \hat{\bm z}')$ be an independent copy of $({\bm z}, \hat{\bm z})$. Then, for any $J \subset[n]$,
 \[
 \cL\left(\langle B_n {\bm z}, \hat{\bm z}\rangle, \vep\right)^2 \le \P\left(\left| \langle B_n({\bm z}_{J^c} -{\bm z}'_{J^c}),  \hat{\bm z}_J\rangle +  \langle B_n^*(\hat{\bm z}_{J^c} -\hat{\bm z}'_{J^c}), {\bm z}_J\rangle- v\right| \le 2 \vep \right), 
 \]
 where $v$ is some random vector depending on the $J^c \times J^c$ minor of $B_n$, and the random vectors ${\bm z}_{J^c}, {\bm z}'_{J^c}, \hat{\bm z}_{J^c}$, and $\hat{\bm z}'_{J^c}$.
 \end{lem}

 \vskip10pt
Using \cite[Lemma 14]{costello},  in \cite{V} (see Proposition 5.1 there), a version of Lemma \ref{lem:decoupling} was proved when $B_n$ a symmetric matrix and ${\bm x}={\bm y}$. The same proof, with appropriate changes, works for a general matrix $B_n$ and with the stated assumptions on the joint law of ${\bm x}$ and ${\bm y}$. We omit the details.

Recall that in Section \ref{sec:inv-dom-comp} the invertibility over compressible and dominated vectors was proved under the general Assumption \ref{ass:matrix-entry}, and as seen in Remark \ref{rmk:skew symm}, the assumption can be further relaxed to include skew-symmetric matrices. Since skew-symmetric matrices of odd dimension are always singular, one cannot expect to have a unified proof for all matrices satisfying this general assumption. As we will see below, the proofs for the directed and the undirected Erd\H{o}s-R\'{e}nyi graphs differ in choosing $J \subset [n]$ in Lemma \ref{lem:decoupling}.

So, first let us consider the case when $A_n$ is the adjacency matrix of a directed Erd\H{o}s-R\'{e}nyi graph. We see that to apply Lemma \ref{lem:decoupling} one needs to condition on $A_n$ (notice that the matrix $B_n$ in Lemma \ref{lem:decoupling} is a deterministic matrix). Once we show that
\beq\label{eq:ell-2-to-hs}
1+ \|A_n^{-1} {\bm x}\|_2 \sim  p^{1/2} \|A_n^{-1}\|_{{\rm HS}} = \Omega(1),
\eeq
with large probability, we can replace the denominator of the \abbr{RHS} of \eqref{eq:dist-quad} by $p^{1/2} \|A_n^{-1}\|_{{\rm HS}}$. This allows us to condition on $A_n$ and then apply the decoupling lemma with an appropriate choice of the set $J \subset [n]$.

To this end, we first show that \eqref{eq:ell-2-to-hs} holds when ${\bm x}$ is replaced by its centered version $\bar{\bm x}$.
To tackle the additional difficulty of the non-zero mean we then show that the eigenvector corresponding to the largest eigenvalue of $A_n$ is close to the vector of all ones so that $|\|A_n^{-1} {\bm x}\|_2- \|A_n^{-1} \bar{\bm x}\|_2|$ is small.

Let us state the lemma showing that $\|A_n^{-1} \bar{\bm x}\|_2 \sim p^{1/2} \|A_n^{-1}\|_{{\rm HS}}$.

\begin{lem}\label{lem:prep-bounds}
Let $A_n$ satisfies Assumption \ref{ass:matrix-entry}, with $p$ such that \(
   np \ge \log(1/p)
   \). Let ${\bm x} \in \R^n$ be a random vector with i.i.d.~$\dBer(p)$ entries, and ${\bm x}'$ be an independent copy of ${\bm x}$. Denote $\bar {\bm x} : = {\bm x} - \E{\bm x}$.Then we have the following:

\begin{enumerate}[(i)]
\item For every $\vep_\star >0$,
\[
\P_{{\bm x}}\left(\|A_n^{-1} \bar{\bm x}\|_2 \le \vep_\star^{-1/2}p^{1/2} (1-p)^{1/2}  \|A_n^{-1} \|_{{\rm HS}} \right) \ge 1-\vep_\star,
\]
where $\P_{{\bm x}}(\cdot)$ denotes the probability under the law of ${\bm x}$.

\item Fix $K \ge 1$. Then, for every $\vep_\star >0$,
\beq\label{eq:HS_lbd}
\P \left(\left\{\|A_n^{-1}({\bm x} - {\bm x}')\|_2 \le \vep_\star p^{1/2} \rho \|A_n^{-1}\|_{{\rm HS}}\right\} \cap \Omega_K^0 \right) \le 4 C_{\ref{lem:bound-levy}} \left( \vep_\star + \f{K c_{\ref{p: spread vectors-1}}^{-1} }{\sqrt{np}}\right) + 2 n^{-\bar{c}_{\ref{cor:combine}}},
\eeq
where $\rho$ as in Proposition \ref{l: sparse vectors-2}.
\end{enumerate}
\end{lem}

\vskip10pt
The proof of Lemma \ref{lem:prep-bounds} is deferred to the end of this section. The next lemma shows that $A_n$ has a large eigenvalue and the eigenvector corresponding to that eigenvalue is close to the vector of all ones.

\begin{lem}\label{lem:large-eig}
Let $A_n$ be an (possibly random) $n \times n$ matrix and for $K \ge 1$, let
\[
\bar\Omega_K:= \left\{\|A_n - p {\bm J}_n\| \le  K \sqrt{np}\right\},
\]
for some $p \in (0,1/2]$ such that $np \to \infty$ as $n \to \infty$, where we recall that ${\bm J}_n$ is the $n \times n$ matrix of all ones. Then on the event $\bar\Omega_K$, for all large $n$, the following hold:

\begin{enumeratei}

\item There exists a real eigenvalue $\lambda_0$ of $A_n$ such that $|\lambda_0| \ge  \f{np}{2}$.

\item We further have
\[
  \norm{v_0- \frac{1}{\sqrt{n}} \mathbf{1} }_2 \le \frac{16K}{\sqrt{np}},
 \]
where $v_0 \in S^{n-1}$ is the eigenvector corresponding to the eigenvalue $\lambda_0$.
\end{enumeratei}
\end{lem}

\vskip10pt
Equipped with Lemmas \ref{lem:prep-bounds} and \ref{lem:large-eig}, one obtains \eqref{eq:ell-2-to-hs}, which together with Lemma \ref{lem:decoupling} implies that one has to find a probability of the event
\beq\label{eq:intermediate-outline}
\left\{\f{|\langle A_n^{-1} ({\bm x}_{J^c} - {\bm x}'_{J^c}), {\bm y}_J \rangle  + \langle (A_n^{-1})^* ({\bm y}_{J^c} - {\bm y}'_{J^c}), {\bm x}_J \rangle- v |}{\|A_n^{-1}\|_{{\rm HS}}} \le c \vep \rho^2p\right\},
\eeq
{for some appropriate choice of $J \subset [n]$. Here the set $J$ is at our disposal,  $v$ is as in Lemma \ref{lem:decoupling}, $({\bm x}', {\bm y}')$ is an independent copy of $({\bm x}, {\bm y})$ and $c$ is some small constant}. As  $A_n$ is not symmetric and the first row and column, after removing the first diagonal entry, are dependent, to obtain  a bound on the probability of the event in \eqref{eq:intermediate-outline} we need a bound on the L\'{e}vy concentration function of a sum of two correlated random variables. A natural solution would be to take a $J \subset J_0 \subset [n]$ such that ${\bm x}_{J_0} \equiv 0$ so that second term in the numerator of \eqref{eq:intermediate-outline} vanishes.
Having used this trick, in order to be able to apply Lemma \ref{lem:bound-levy} we finally need to show that $v^\star_J$ has a large spread component and $\|v^\star_J\|_2$ is not too small, where
\[
v^\star:= \f{A_n^{-1}({\bm x}_{J^c} - {\bm x}'_{J^c})}{\|A_n^{-1}({\bm x}_{J^c} - {\bm x}'_{J^c})\|_2}.
\]
The existence of a $J \subset J_0 \subset [n]$ so that $v^\star$ has the desired properties is guaranteed by Lemma \ref{lem:non-dominated-J}. Putting these pieces together one then completes the proof. Below we expand on this idea to complete the proof of Proposition \ref{prop:dist-bd} for the adjacency matrix of a directed Erd\H{o}s-R\'{e}nyi graph.

\begin{proof}[Proof of Proposition \ref{prop:dist-bd} for directed Erd\H{o}s-R\'{e}nyi graph]
As mentioned above, to apply Lemma \ref{lem:decoupling} we need to show that \eqref{eq:ell-2-to-hs} holds with high probability. To this end, we begin by noting that
\beq
\sqrt{1+\|A_n^{-1} {\bm x}\|_2^2} \le  2\sqrt{1+\|A_n^{-1} \bar{\bm x}\|_2^2} + 2 p \|A_n^{-1} {\bm 1}\|_2, \notag
\eeq
where $\bar{\bm x}= \{\bar x_i\}_{i=1}^n \in \R^n$ and $\bar x_i = x_i - \E x_i$ for $i \in [n]$. Therefore denoting
\[
\cF_n:= \left\{ p^{1/2} \|A_n^{-1} {\bm 1} \|_2\le \vep_1^{-1/2} \|A_n^{-1} \|_{{\rm HS}}\right\}
\]
and
\[
\cG_n:=\left\{ \|A_n^{-1} \bar{\bm x}\|_2 \le \vep_1^{-1/2} p^{1/2}(1-p)^{1/2} \|A_n^{-1}\|_{{\rm HS}}\right\},
\]
we have that
\beq\label{eq:norm-triangle}
\sqrt{1+\|A_n^{-1} {\bm x}\|_2^2}  \le 4 \vep_1^{-1/2} p^{1/2} \|A_n^{-1}\|_{{\rm HS}} +2
\eeq
 on the event $\cF_n \cap \cG_n$,
where
 $\vep_1>0$ to be determined later during the course of the proof.

We claim that $\Omega_K^0 \subset \cF_n$. Indeed, using Lemma \ref{lem:large-eig} we have that
\begin{align}\label{eq:F-n-whp}
p^{1/2}\|A_n^{-1} {\bm 1}\|_2  & \le \sqrt{np} \cdot \|A_n^{-1} v_0\|_2 + \sqrt{np} \cdot \norm{A_n^{-1} \left(\f{1}{\sqrt{n}}{\bm 1} -v_0\right)}_2\\
& \le \f{\sqrt{np}}{|\lambda_0|} + \sqrt{np} \cdot \|A_n^{-1}\|_{{\rm HS}} \cdot \norm{\f{1}{\sqrt{n}}{\bm 1} -v_0}_2 \le \f{2}{\sqrt{np}} +  32 K \|A_n^{-1}\|_{{\rm HS}}, \notag
\end{align}
where in the last step we have used the fact that $\Omega_K^0 \subset \bar \Omega_{2K}$ for all large $n$.
Using Jensen's inequality applied to the empirical measure of the square of the singular values of $A_n$ (or equivalently using \abbr{AM-HM} inequality) we see that
\[
\|A_n\|_{{\rm HS}}^2 \cdot \|A_n^{-1}\|_{{\rm HS}}^2 \ge n^2.
\]
Since
\[
\|A_n\|_{{\rm HS}}^2 \le 2 \|A_n - \E A_n\|_{{\rm HS}}^2 + 2 \|\E A_n \|_{{\rm HS}}^2 \le 2n \cdot (K^2 np ) + 2 (np)^2 \le 4 K^2 n^2 p,
\]
on the event $\Omega_K^0$, we deduce that
\beq\label{eq:HS-lbd}
\|A_n^{-1}\|_{{\rm HS}} \ge 1/ (2K p^{1/2}).
\eeq
Plugging this bound in \eqref{eq:F-n-whp} and setting $\vep_1 \le 10^{-4} K^{-2}$ we derive that $\Omega_K^0 \subset \cF_n$.

Hence, estimating $\P (\cG_n) $ by Lemma \ref{lem:prep-bounds}(i) and using \eqref{eq:HS-lbd} again, we obtain from \eqref{eq:norm-triangle} that
\begin{multline}\label{eq:dist-bound-split-1}
\P \left(\left\{\f{\left| \langle A_n^{-1} {\bm x}, {\bm y}\rangle -u\right|}{\sqrt{1+\|A_n^{-1} {\bm x}\|_2^2}} \le \vep \rho^2 p^{1/2}\right\} \cap \Omega_K^0  \right) \\
\le \P \left(\left\{{\left| \langle A_n^{-1} {\bm x}, {\bm y}\rangle -u\right|} \le  5 \vep \vep_1^{-1/2} p \rho^2 \|A_n^{-1}\|_{{\rm HS}}\right\} \cap \Omega_K^0 \cap \cG_n  \right) + \vep_1.
\end{multline}
Therefore, to complete the proof it remains to find a bound on the first term in the \abbr{RHS} of \eqref{eq:dist-bound-split-1}.

Now we will apply Lemma \ref{lem:decoupling}. Recalling the fact that ${\bm x}^{\sf T}$ and ${\bm y}$ are the first row and column of ${\sf A}_n$, after removing the first diagonal entry, using the representation \eqref{eq:dir-erdos} we note that
\[
x_i = \theta_i \cdot \gamma_i \qquad \text{ and } \qquad y_i = (1-\theta_i) \cdot \varpi_i,
\]
where $\{\gamma_i\}, \{\varpi_i\}$, and $\{\theta_i\}$ are sequences of independent $\dBer(2p)$, $\dBer(2p)$, and $\dBer(1/2)$ random variables, respectively.  Set $J:= \{i \in [n]: \theta_i =0\}$. Upon conditioning on ${\bm \theta}:=\{\theta_i\}_{i \in [n]}$ we see that $ {\bm x}_J \equiv 0$, and $ {\bm y}_J$ and $ {\bm x}_{J^c}$ are distributed as i.i.d.~sequence of $\dBer(2p)$ random variables.
Denote
\[
\Upsilon_n:=\P \left(\left\{{\left| \langle A_n^{-1} {\bm x}, {\bm y}\rangle -u\right|} \le  5 \vep \vep_1^{-1/2} p\rho^{2} \|A_n^{-1}\|_{{\rm HS}}\right\} \cap \Omega_K^0 \right).
\]
An application  of Jensen's inequality and Lemma \ref{lem:decoupling} yields that
\begin{multline}\label{eq:upsilon-bound}
\Upsilon_n^2 \le \E\left[ \cL\left(\langle A_n^{-1} {\bm x}, {\bm y}\rangle, 5 \vep \vep_1^{-1/2} p \rho^2 \|A_n^{-1}\|_{{\rm HS}} \rangle \; \Big| \; A_n, {\bm \theta}\right)^2\bI_{\Omega_K^0}\right] \\
\le \E\left[\P\left( \left|\langle A_n^{-1} ({\bm \gamma} - {\bm \gamma}'), {\bm \varpi}\rangle - v\right| \le 10 \vep \vep_1^{-1/2} p \rho^2 \|A_n^{-1}\|_{{\rm HS}} \; \Big| \;  A_n, {\bm \theta}\right) \bI_{\Omega_K^0}\right],
\end{multline}
where
\[
{\bm \varpi}:= \left\{\begin{array}{ll} \varpi_i & i \in J\\ 0 & i \in J^c \end{array} \right. , \qquad {\bm \gamma}:= \left\{\begin{array}{ll} \gamma_i & i \in J^c\\ 0 & i \in J \end{array} \right. ,
\]
${\bm \gamma}'$ an independent copy of ${\bm \gamma}$, and $v$ is some random vector depending only the $J^c \times J^c$ minor of $A_n^{-1}$, ${\bm \gamma}$, and ${\bm \gamma}'$.

Estimating the \abbr{RHS} of \eqref{eq:upsilon-bound} relies on Lemma \ref{lem:bound-levy}. To apply it, we need to  bound the probability appearing there by the L\'evy concentration function from this Lemma. We show that this can be done after discarding two events of a small probability.

To this end, denoting
\[
\widehat \cG_n:= \left\{\|A_n^{-1}({\bm \gamma} - {\bm \gamma}')\|_2 \ge \vep_1 p^{1/2} \rho \|A_n^{-1}\|_{{\rm HS}}\right\},
\]
we see that
\beq\label{eq:hat-G-n}
\left\{ \left|\langle A_n^{-1} ({\bm \gamma} - {\bm \gamma}'), {\bm \varpi}\rangle - v\right| \le 10 \vep \vep_1^{-1/2} p\rho^{2} \|A_n^{-1}\|_{{\rm HS}} \right\} \cap  \widehat \cG_n \subset \left\{ \left|\langle {\bm \xi} , {\bm \varpi} \rangle - \wt v \right| \le 10 \vep \vep_1^{-3/2} p^{1/2}\rho\right\},
\eeq
where
\[
{\bm \xi}:= \f{A_n^{-1} ({\bm \gamma} - {\bm \gamma}')}{\|A_n^{-1} ({\bm \gamma} - {\bm \gamma}')\|_2} \in S^{n-1} \qquad \text{ and } \qquad \wt v:= \f{v}{\|A_n^{-1} ({\bm \gamma} - {\bm \gamma}')\|_2}.
\]
As ${\bm \varpi}_{J^c} \equiv 0$, to be able to apply Lemma \ref{lem:bound-levy}, we have to select a set  $I \subset J$ such that ${\bm \xi}_I$ has a substantial Euclidean norm and is non-dominated, with large probability.

So, we define
\[
\wt \cG_n:= \left\{\f{\|{\bm \xi}_{[\f{n}{4}+1:n]\cap J}\|_\infty}{\|{\bm \xi}_{[\f{n}{4}+1:n]\cap J}\|_2} \le \f{1}{c_{\ref{lem:non-dominated-J}} \sqrt{n}}, \, \|{\bm \xi}_{[\f{n}{4}+1:n]\cap J}\|_2 \ge \rho \right\}.
\]
As the coordinates of ${\bm \varpi}_J$ are i.i.d.~$\dBer(2p)$, setting $I:= \supp({\bm \xi}_{[\f{n}{4}+1:n] \cap J})$, and
using Lemma \ref{lem:bound-levy} we find that
\begin{multline}
\P\left(\left|\langle {\bm \xi} , {\bm \varpi} \rangle - \wt v \right| \le 10 \vep \vep_1^{-3/2} p^{1/2} \rho \; \Big| \; A_n, {\bm \theta}, {\bm \gamma}, {\bm \gamma'}\right) \bI_{\wt \cG_n} \\
\le \cL\left(\langle {\bm \xi}_I, {\bm \varpi}_I\rangle, \ 20 \vep \vep_1^{-3/2} p^{1/2}(1-p)^{1/2}\|{\bm \xi}_I\|_2  \right)\bI_{\wt \cG_n}
 \le 20 C_{\ref{lem:bound-levy}} \left( \vep \vep_1^{-3/2}+ \f{1}{c_{\ref{lem:non-dominated-J}}\sqrt{np}}\right).
\end{multline}
To complete the proof it remains to show that both $\wt \cG_n$ and $\widehat \cG_n$ have large probabilities. First let us show that $\P(\wt \cG_n^c)$ is small. By Chernoff's bound, there exists a set $\Omega_\gamma$ such that on that set $|\supp({\bm \gamma} - {\bm \gamma}')| \le C_\star np$ and $\|{\bm \gamma} - {\bm \gamma}'\|_2 \in [1, \bar C np]$ for some $C_\star, \bar C >0$, with $\P(\Omega_\gamma) \ge 1- \exp(-\bar c np)$, for some constant $\bar c >0$. Moreover, by Chernoff's bound again, there exists a set $\Omega_\theta$ with probability at least  $1 - \exp(-c_* n)$, for some $c_*>0$, such that $\f{3n}{8} \le |J| \le \f{5n}{8}$ on $\Omega_\theta$. Hence, applying Lemma \ref{lem:non-dominated-J} we find that
\begin{multline}
\P(\wt \cG_n^c \cap \Omega_K^0) \le \E\left[\P\left(\wt \cG_n^c \cap \Omega_K^0 \; \Big| \; {\bm \theta}, {\bm \gamma}, {\bm \gamma}'\right) {\bI}_{\Omega_\gamma \cap \Omega_\theta} \right] + \P(\Omega_\gamma^c) + \P(\Omega_\theta^c) \\
\le n^{-\bar c_{\ref{lem:non-dominated-J}}}+ \exp(-\bar c np) + \exp(-c_* n).
\end{multline}
Next, let us show that $\widehat \cG_n $ has a large probability.
Recall that $\bm \gamma$ is a random vector with independent $\dBer (p)$ coordinates, and $\bm \gamma'$ is an independent copy of $\bm \gamma$.
 Using Lemma \ref{lem:prep-bounds}(ii) we  obtain that
\begin{multline}\label{eq:hat-G-n-pr}
\P(\widehat \cG_n^c \cap \Omega_K^0)= \P\left(\left\{\|A_n^{-1}({\bm x} - {\bm x}')\|_2 \le \vep_1 p^{1/2} \rho \|A_n^{-1}\|_{{\rm HS}}\right\} \cap \Omega_K^0\right) \\
\le 8 C_{\ref{lem:bound-levy}} \left( \vep_1 + \f{K c_{\ref{p: spread vectors-1}}^{-1}}{\sqrt{np}}\right) + 2 n^{-\bar{c}_{\ref{cor:combine}}},
\end{multline}
where  ${\bm x}'$ is an independent copy of ${\bm x}$, establishing $\widehat \cG_n$ has a large probability.

Now, combining \eqref{eq:hat-G-n}-\eqref{eq:hat-G-n-pr}, from \eqref{eq:upsilon-bound} we derive that
\[
\Upsilon_n^2 \le C(\vep_1 + \vep \vep_1^{-3/2}) + \f{\bar C}{\sqrt{np}}+n^{-c},
\]
for some large constants $C, \bar C$, and some small constant $c$. This together with \eqref{eq:dist-bound-split-1} now implies that
\[
\P \left(\left\{\f{\left| \langle A_n^{-1} {\bm x}, {\bm y}\rangle -u\right|}{\sqrt{1+\|A_n^{-1} {\bm x}\|_2^2}} \le \vep \rho^2 p^{1/2}\right\} \cap \Omega_K^0  \right) \le C(\vep_1 + \vep_1^{1/2} + \vep^{1/2} \vep_1^{-3/4}) + \f{\bar C}{(np)^{1/4}}+ n^{-\f{c}{2}}.
\]
Finally choosing $\vep_1 = \vep^{\f{2}{5}}$ and replacing $\vep$ by $\vep/C^5$ the proof completes.
\end{proof}

Next we carry out the proof for the adjacency matrix of a undirected Erd\H{o}s-R\'{e}nyi graph. It follows from simple modification of the same for the directed case. Hence, we only provide an outline indicating the necessary changes.

\begin{proof}[Proof of Proposition \ref{prop:dist-bd} for undirected Erd\H{o}s-R\'{e}nyi graph]
Since in the undirected case ${\bm x}={\bm y}$, proceeding similarly to the steps leading to \eqref{eq:dist-bound-split-1} we derive that
\begin{multline}\label{eq:dist-bound-split-1-und}
\P \left(\left\{\f{\left| \langle A_n^{-1} {\bm x}, {\bm x}\rangle -u\right|}{\sqrt{1+\|A_n^{-1} {\bm x}\|_2^2}} \le \vep \rho^2 p^{1/2}\right\} \cap \Omega_K^0  \right) \\
\le \P \left(\left\{{\left| \langle A_n^{-1} {\bm x}, {\bm x}\rangle -u\right|} \le  5 \vep \vep_1^{-1/2} p \rho^2 \|A_n^{-1}\|_{{\rm HS}}\right\} \cap \Omega_K^0 \cap \cG_n  \right) + \vep_1.
\end{multline}
Next set $J:=\{i \in [n]: \gd_i = 0\}$ where $\{\gd_i\}_{i \in [n]}$ are i.i.d.~$\dBer(\f12)$. Using this choice of $J$ we then apply Lemma \ref{lem:decoupling} to see that
\begin{equation}\label{eq:upsilon-bound-und}
\wt \Upsilon_n^2
\le \E\left[\P\left( \left|\langle A_n^{-1} ({\bm x}_{J^c} - {\bm x}'_{J^c}), {\bm x}_J\rangle - v\right| \le 5 \vep \vep_1^{-1/2} p \rho^2 \|A_n^{-1}\|_{{\rm HS}} \, \Big| \, A_n, J \right) \bI_{\Omega_K^0}\right],
\end{equation}
where
\[
\wt \Upsilon_n:=\P \left(\left\{{\left| \langle A_n^{-1} {\bm x}, {\bm x}\rangle -u\right|} \le  5 \vep \vep_1^{-1/2} p\rho^{2} \|A_n^{-1}\|_{{\rm HS}}\right\} \cap \Omega_K^0 \right),
\]
and $v$ is some vector depending on the $J^c \times J^c$ sub-matrix of $A_n$, ${\bm x}_{J^c}$, and ${\bm x}'_{J^c}$. As the entries of the random vector ${\bm x}_{J^c}$ are i.i.d.~$\dBer(\f{p}{2})$, using Lemma \ref{lem:prep-bounds}(ii) we find that
\begin{equation}\label{eq:hat-G-n-pr-1}
\P(\overline \cG_n^c \cap \Omega_K^0)
\le 16 C_{\ref{lem:bound-levy}} \left( \vep_1 + \f{K c_{\ref{p: spread vectors-1}}^{-1}}{\sqrt{np}}\right) + 2 n^{-\bar{c}_{\ref{cor:combine}}},
\end{equation}
where
\[
\overline \cG_n:= \left\{\|A_n^{-1}({\bm x}_{J^c} - {\bm x}_{J^c}')\|_2 \ge \vep_1 p^{1/2} \rho \|A_n^{-1}\|_{{\rm HS}}\right\}.
\]
As
\beq
\left\{ \left|\langle A_n^{-1} ({\bm x}_{J^c} - {\bm x}'_{J^c}), {\bm x}_J\rangle - v\right| \le 5 \vep \vep_1^{-1/2} p\rho^{2} \|A_n^{-1}\|_{{\rm HS}} \right\} \cap  \overline \cG_n \subset \left\{ \left|\langle \overline{\bm \xi} , {\bm x}_J \rangle - \overline v \right| \le 5 \vep \vep_1^{-3/2} p^{1/2}\rho\right\}, \notag
\eeq
where
\[
\overline {\bm \xi}:= \f{A_n^{-1} ({\bm x}_{J^c} - {\bm x}'_{J^c})}{\|A_n^{-1} ({\bm x}_{J^c} - {\bm x}'_{J^c})\|_2} \in S^{n-1} \qquad \text{ and } \qquad \overline v:= \f{v}{\|A_n^{-1} ({\bm x}_{J^c} - {\bm x}'_{J^c})\|_2},
\]
proceeding similarly as in the proof in the directed case the remainder of this proof can be completed. We leave the details to the reader.
\end{proof}

We end this section with proofs of Lemmas \ref{lem:prep-bounds} and \ref{lem:large-eig}.

\begin{proof}[Proof of Lemma \ref{lem:prep-bounds}]
The proof of part (i) is essentially an application of Markov's inequality. To this end, we note that
\beq\label{eq:norm-split}
\| A_n^{-1} \bar{\bm x}\|_2^2 = \sum_{k=1}^n \langle A_n^{-1}\bar{\bm x}, e_k \rangle^2 = \sum_{k=1}^n \langle (A_n^{-1})^{\sf T} e_k, \bar{\bm x} \rangle^2 = \sum_{k=1}^n \| (A_n^{-1})^{\sf T} e_k\|_2^2 \langle w_k, \bar{\bm x}\rangle^2,
\eeq
where
\[
w_k:= \f{(A_n^{-1})^{\sf T} e_k}{\| (A_n^{-1})^{\sf T} e_k\|_2^2},
\]
and $e_k$ is the $k$-th canonical basis vector. Since $\|w_k\|_2=1$ and the random vector $\bar{\bm x}$ has zero mean with i.i.d.~coordinates we have
\[
\E_{{\bm x}} \left[ \langle w_k, \bar{\bm x}\rangle^2 \right] = \Var_{{\bm x}} \left( \langle w_k, {\bm x}\rangle \right) = \Var (x_1)=p(1-p),
\]
which in turn implies that
\[
\E_{\bm x} \left[ \| A_n^{-1} \bar{\bm x}\|_2^2\right] = p(1-p) \sum_{k=1}^n \| (A_n^{-1})^{\sf T} e_k\|_2^2= p(1-p) \|A_n^{-1}\|_{{\rm HS}}^2,
\]
where $\E_{{\bm x}}$ and $\Var_{{\bm x}}$ denote the expectation and the variance with respect to the randomness of ${\bm x}$. The conclusion of part (i) now follows upon using Markov's inequality.

Turning to prove (ii), we denote $p_k := \|(A_n^{-1})^{\sf T} e_k\|_2^2/\|A_n^{-1}\|_{{\rm HS}}^2$. As $\sum_{k=1}^n p_k=1$, proceeding as in \eqref{eq:norm-split}, and applying \cite[Lemma 8.3]{V} we note that
\begin{multline}\label{eq:HS-split}
\P \left(\|A_n^{-1}({\bm x} - {\bm x}')\|_2 \le \vep_\star p^{1/2} \rho \|A_n^{-1}\|_{{\rm HS}} \; \Big| \; A_n\right) = \P\left(\sum_{k=1}^n p_k \langle w_k, {\bm x} - {\bm x}'\rangle^2 \le \vep_\star^2 p \rho^2 \; \Big| \; A_n\right) \\
\le 2 \sum_{k=1}^n p_k \P\left( \langle w_k, {\bm x} - {\bm x}'\rangle^2 \le 2\vep_\star^2 p \rho^2 \; \Big| v\;  A_n\right).
\end{multline}
As $\sum_{k=1}^n p_k=1$, the advantage of working with the \abbr{RHS} of \eqref{eq:HS-split} is that it is enough to find the maximum of the probabilities under the summation. To find such a bound we would like to use Lemma \ref{lem:bound-levy}. This  requires to show that $w_k$ is neither dominated nor compressible with high probability.

Turning to this task, recall that $w_k = \f{(A_n^{-1})^{\sf T} e_k}{ \|(A_n^{-1})^{\sf T} e_k\|_2} $.
Since $A_n^{\sf T}$ also satisfies Assumption \ref{ass:matrix-entry}, applying Corollary \ref{cor:combine} with $c_0^*=1/2$, we obtain that
\begin{multline}\label{eq:levy-bd-condition}
\P\left( \left\{\langle w_k, {\bm x} - {\bm x}'\rangle^2 \le 2\vep_\star^2 p \rho^2\right\} \cap \Omega_K^0 \right) \\
\le \E\left[\P\left( \langle w_k, {\bm x} - {\bm x}'\rangle^2 \le 2\vep_\star^2 p \rho^2 \Big| A_n\right)\bI(w_k \notin V_{1/2, c_{\ref{p: spread vectors-1}}})\right] + n^{-\bar{c}_{\ref{cor:combine}}}.
\end{multline}
If $w_k \notin V_{1/2, c_{\ref{p: spread vectors-1}}}$ then
\[
\|(w_k)_{[n/2+1:n]}\|_2 \ge \rho \qquad \text{ and } \qquad \f{\|(w_k)_{[n/2+1:n]}\|_\infty}{\|(w_k)_{[n/2+1:n]}\|_2} \le \f{2K}{c_{\ref{p: spread vectors-1}}\sqrt{n}}.
\]
So now we apply Lemma \ref{lem:bound-levy} to find that
\[
\P\left(\left| \langle w_k, {\bm x} - {\bm x}' \rangle\right| \le 2 \vep_\star p^{1/2}\rho \Big| A_n \right) \bI(w_k \notin V_{1/2, c_{\ref{p: spread vectors-1}}}) \le 4 C_{\ref{lem:bound-levy}} \left( \vep_\star + \f{K c_{\ref{p: spread vectors-1}}^{-1} }{\sqrt{np}}\right),
\]
where we have used the fact that $p \le \f34$. This, together with \eqref{eq:levy-bd-condition}, upon taking an average over $A_n$, in \eqref{eq:HS-split}, such that $\Omega_K^0$ holds, yields the bound \eqref{eq:HS_lbd}. This completes the proof of the lemma.
\end{proof}

\begin{proof}[Proof of Lemma \ref{lem:large-eig}]
Denote
 \[
  W:=\left \{ w \in \R^n: \ \norm{w- \mathbf{e}}_2 \le \frac{8K}{\sqrt{np}} \quad \text{ and } \quad \ \pr{w}{\mathbf{e}}=1 \right \},
 \]
 where for brevity we write ${\bf e}:= \f{1}{\sqrt{n}}{\bm 1}$. Define the function $F: W \to \R^n$ by
 \[
  F(x):= \frac{A_nx}{\pr{A_n x}{\mathbf{e}}}.
 \]
 We claim that $F(W) \subset W$. We will see below that proving this claim will imply that $A_n$ has a large eigenvalue and  the eigenvector corresponding to that large eigenvalue is close to ${\bf e}$.

 To check the claim, note that for any $x \in \R^n$ we have
 \(
 \pr{F(x)}{{\bf e}}=1
 \). Therefore it remains to show that
 \beq\label{eq: first-0}
 \|F(x) - {\bf e}\|_2 \le \frac{8K}{\sqrt{np}}, \quad \text{ for all } x \in W.
 \eeq
 To this end, for any $x \in W$ we write $x=\mathbf{e}+y$ where from the definition of the set $W$ it follows that $\norm{y}_2 \le  \frac{8K}{\sqrt{np}}$. As $\pr{x}{{\bf e}}=1$ and $\norm{{\bf e}}_2=1$ we further have that $\pr{y}{{\bf e}}=0$, which in turn implies that
 \(
 A_n y = (A_n - p {\bm J}_n)y.
 \)
 As
 \[
 (A_n - p {\bm J}_n) {\bf e} = A_n {\bf e} - np {\bf e},
 \]
 we deduce that
 \[
 \| A_n {\bf e} - np {\bf e}\|_2 \le K \sqrt{np}
 \]
 on the event $\bar \Omega_K$. So we obtain that
  \begin{equation} \label{eq: first}
 \norm{A_n x- np \mathbf{e}}_2
 \le \norm{A_n \mathbf{e}- np \mathbf{e}}_2+\norm{A_n - p {\bm J}_n} \cdot \norm{y}_2
 \le K \sqrt{np} \left(1+\frac{8K}{\sqrt{np}} \right)
 \end{equation}
 on the event $\bar\Omega_K$. Thus using Cauchy-Schwarz inequality
 \beq\label{eq: first-1}
  |\pr{A_n x}{\mathbf{e}}-np| \le K \sqrt{np}\left(1+\frac{8K}{\sqrt{np}} \right).
 \eeq
 Using the fact that $np \to \infty$ as $n \to \infty$, and the triangle inequality we also see from above that
 \beq\label{eq: first-2}
 |\pr{A_n x}{{\bf e}}| \ge \f{np}{2},
 \eeq
 for all large $n$. Combining \eqref{eq: first}-\eqref{eq: first-2}, and using the triangle inequality once more, we derive that on the event $\bar\Omega_K$,
 \begin{align*}
  \norm{F(x)-\mathbf{e}}_2 \le \f{\|A_n x - np {\bf e}\|_2}{|\pr{A_n x}{{\bf e}}|} + \f{\|(\pr{A_n x}{{\bf e}} - np) {\bf e}\|_2}{|\pr{A_n x}{{\bf e}}|} \le \f{4K}{\sqrt{np}}\left( 1+ \f{8K}{\sqrt{np}}\right) \le \f{8K}{\sqrt{np}},
 \end{align*}
 for all large $n$. This proves \eqref{eq: first-0} and hence we have the claim that $F(W) \subset W$.

Now to show that the claim implies the existence of a real large eigenvalue we apply Brouwer fixed point theorem. It implies that there exists $w \in W$ such that
 \[
 A_n w= \pr{A_n w}{{\bf e}} w.
 \]
 Equivalently, $w$ is an eigenvector of $A_n$ corresponding to the eigenvalue $\lambda_0:= \pr{A_n w}{{\bf e}}$. The lower bound on $|\lambda_0|$ follows from \eqref{eq: first-2}. To complete the proof of the lemma we note that
 \[
\left| \|w\|_2-1\right| \le \|w- {\bf e}\|_2 \le \f{8K}{\sqrt{np}}
 \]
Therefore setting $v_0:=w/\|w\|_2$ we obtain
\[
\|v_0 - {\bf e}\|_2 \le \|w - {\bf e}\|_2 + \|w\|_2 \cdot \left| \f{1}{\|w\|_2} -1 \right| \le \f{16K}{\sqrt{np}}.
\]
This finishes the proof of the lemma.
\end{proof}

\section{Proofs of Theorems \ref{thm:bernoulli} and \ref{thm:s-min-graphs}}\label{sec:proof-main-thm}

In this section we prove Theorems \ref{thm:bernoulli} and \ref{thm:s-min-graphs}. First let us prove part (ii) of Theorem \ref{thm:s-min-graphs}. We will show that the conclusion of Theorem \ref{thm:s-min-graphs}(ii) holds under a more general set-up, namely when the entries of $A_n$ satisfy Assumption \ref{ass:matrix-entry}.

\begin{proof}[Proof of Theorem \ref{thm:s-min-graphs}(ii)]
The proof of (ii) is standard and is provided for a reader's convenience.
We begin by noting that if $A_n$ satisfies Assumption \ref{ass:matrix-entry} it is enough to show that
\beq\label{eq:Omega-0-bd}
\P(\Omega_{0, \col}^c) \le \f{\bar{C}_{\ref{thm:s-min-graphs}}}{2\log n},
\eeq
where $\Omega_{0,\col}$ is the event that there exists zero columns in $A_n$.

To prove \eqref{eq:Omega-0-bd} we use Chebychev's inequality. We will show that $\Var(\sN) \approx \E [\sN]$, where $\sN$ is the number of zero columns in $A_n$. This observation, together with the fact $\E[\sN] \to \infty$ as $n \to \infty$, whenever $np \le \log (1/p)$, will show that $\sN$ cannot deviate too much from its expectation with large probability. Then, noting that $\Omega_{0,\col}^c = \{\sN = 0\}$, the desired probability bound on $\Omega_{0,\col}^c$ follows. Below we carry out this task.

To this end, denote $\bI_i:=\bI_i(A_n)$ to be the indicator of the event that the $i$-th column of $A_n$ is zero and therefore $\sN=\sum_{i=1}^n \bI_i$. It is easy to note that under Assumption \ref{ass:matrix-entry} we have
\beq\label{eq:Omega-0-exp}
\E\left[\sN\right] = n \P(\bI_1=1) \ge  n(1-p)^n.
\eeq
On the other hand, we see that
\beq\label{eq:Omega-0-var-1}
\Var(\bI_i) \le  \E\, \bI_i  \le (1-p)^{n-1}, \quad i \in [n].
\eeq
Using the fact that the entries of $A_n$ satisfy Assumption \ref{ass:matrix-entry} we further observe that for any $i \ne j \in [n]$ the entries of the sub-matrix of $A_n$ with rows $([n]\setminus\{i,j\})$ and columns $\{i,j\}$ are i.i.d.~$\dBer(p)$ random variables. Therefore
\begin{align}\label{eq:Omega-0-var-2}
\Cov(\bI_i,  \bI_j) & = \E(\bI_i \bI_j) - \E(\bI_i) \cdot \E(\bI_j) \notag\\
& \le \P\left(a_{k,\ell}=0, (k,\ell) \in ([n]\setminus\{i,j\}\right) \times \{i,j\}) - (1-p)^{2n} \notag\\
& = (1-p)^{2(n-2)} - (1-p)^{2n} \le C p (1-p)^{2n},
\end{align}
for some absolute constant $C$, whenever $p \le 1/2$. Thus combining \eqref{eq:Omega-0-exp}-\eqref{eq:Omega-0-var-2} and using Chebychev's inequality we deduce that
\begin{align}\label{eq:Omega-0-chebychev}
\P\left( \left| \sN - \E \sN\right| \ge \f12 \E \sN\right)  & \le 4\f{n(1-p)^{n-1}+C n^2 p(1-p)^{2n}}{(\E \sN)^2} \notag\\
& \le \f{4}{(1-p)\cdot \E [\sN]} + C p \le \f{4(1+C e^{-1})}{(1-p)\cdot \E [\sN]},
\end{align}
where the last step follows from the fact that
\[
p(1-p) \E [\sN] \le np (1-p)^n \le np e^{-np} \le \sup_{x \in (0,\infty)} x e^{-x} = e^{-1}.
\]

To complete the argument it remains to find a suitable lower bound on $\E[\sN]$. To this end, we note that the assumption $np \le \log(1/p)$ implies that $p \le 2\log n/n$. Therefore using the inequality $\log(1-x) \ge - x-x^2$ for $x \in (0,1/2]$ we obtain that
\[
 \E[\sN] \ge n(1-p)^{n} \ge  np \cdot  e^{-n(p+p^2)} \cdot p^{-1} \ge np \cdot e^{-\f{4(\log n)^2}{n}} \ge np \left(1-\f{4(\log n)^2}{n}\right) \ge \f{np}{2},
\]
for all large $n$, where in the third inequality above we have again used the assumption $np \le \log(1/p)$. Thus noting that
\[
\Omega_{0,\col}^c = \{\sN = 0\} \subset \left\{  \left| \sN - \E \sN\right| \ge \f12 \E \sN\right\},
\]
and using \eqref{eq:Omega-0-chebychev} we arrive at \eqref{eq:Omega-0-bd} when $p \ge \f{\log n}{2n}$. If $p \le \f{\log n}{2n}$, we use a different bound on $\E[\sN]$:
\[
\E [\sN] \ge np \cdot e^{-n(p+p^2)} \cdot p^{-1}  \ge np \cdot \f{2n}{\log n} \cdot \f{1}{2\sqrt{n}} \ge \f{\sqrt{n}}{\log n}.
\]
Proceeding as above and combining these two cases completes the proof.
\end{proof}

\vskip5pt

Next combining results of Sections \ref{sec:inv-dom-comp}-\ref{sec:incomp} we finish the proof of Theorem \ref{thm:s-min-graphs}(i). Upon recalling Remark \ref{eq:iid-bipartite} we note that Theorem \ref{thm:s-min-graphs}(i) for $A_n = \Adj(\bGraph(n,p) $ follows from Theorem \ref{thm:bernoulli}. Therefore we prove Theorem \ref{thm:s-min-graphs}(i) only $A_n=\Adj(\Graph(n,p_n))$ or   $\Adj(\dGraph(n,p_n))$.

\begin{proof}[Proof of Theorem \ref{thm:s-min-graphs}(i)]
Recalling that $\Omega_K^0= \{\norm{{A}_n - \E A_n} \le K \sqrt{np}\}$, we note that for any $\vartheta>0$,
 \begin{align}\label{eq:inf_split}
 & \P\Big( \{s_{\min}({A}_n) \le \vartheta\}  \cap  \Omega_K^0 \Big) \notag\\
  \le& \,  \P\Big( \Big\{\inf_{x \in V^c} \norm{{A}_nx}_2 \le \vartheta \Big\}  \cap\Omega_K^0 \Big)
  + \P\Big( \Big\{\inf_{x \in V}\norm{{A}_nx}_2 \le \vartheta  \Big\} \cap  \Omega_K^0 \Big),
 \end{align}
 where
  \beq\label{eq:def-V}
  V:=S^{n-1} \setminus \Big( \text{Comp}(n/2, \rho) \cup \text{Dom}(n/2, (c_{\ref{p: spread vectors-1}}K^{-1}) \Big),
 \eeq
 and $\rho$ as in Proposition \ref{l: sparse vectors-2}. Since $A_n$, the adjacency matrix of any of the three random graph models under consideration, satisfies Assumption \ref{ass:matrix-entry}, using Propositions \ref{l: sparse vectors-2}, \ref{p: spread vectors}, and \ref{p: spread vectors-1}, setting $y_0=0$,  we obtain that
 \beq\label{eq:V^c-inf-bd}
   \P\Big( \inf_{x \in V^c} \norm{{A}_nx}_2 \le \min\{\wt c_{\ref{p: spread vectors}}, \wt c_{\ref{p: spread vectors-1}}\} \rho \sqrt{np},  \,  \norm{{A}_n - \E A_n} \le K \sqrt{np} \Big)
    \le 3 n^{-\bar c_{\ref{l: sparse vectors-2}}}.
 \eeq
Hence, it remains to find an upper bound on the second term in the \abbr{RHS} of \eqref{eq:inf_split}. Using Lemma \ref{l: via distance}, we see that to find an upper bound of
\[
 \P\Big( \Big\{\inf_{x \in V} \norm{{A}_nx}_2 \le c_{\ref{thm:s-min-graphs}}\vep \rho^3 \sqrt{\frac{p}{n}}\Big\} \cap \Omega_K^0\Big)
 \]
it is enough to find the same for
 \begin{equation}\label{eq:dist-new1}
 \P \Big( \Big\{\dist({A}_{n,j},H_{n,j}) \le c_{\ref{thm:s-min-graphs}}\rho^2 \sqrt{p} \vep\Big\} \cap \Omega_K^0 \Big) \text{ for a fixed } j,
 \end{equation}
 where ${A}_{n,j}$ are now columns of ${A}_n$ and $H_{n,j}:={\rm Span}\{A_{n,i}, \, i \in[n]\setminus \{ j\}\}$ (see also Remark \ref{rmk:l via distance}). As $A_n$ satisfies Assumption \ref{ass:matrix-entry}, it suffices consider only  $j=1$.

 Turning to bound $\dist(A_{n,1}, H_{n,1})$, we denote $C_n$ to be the $(n-1) \times (n-1)$ matrix obtained from $A_n^{\sf T}$ upon deleting its first row and column. For the adjacency matrices of directed and undirected Erd\H{o}s-R\'{e}nyi graphs our strategy will change depending on whether $C_n$ is invertible or not.

Using Proposition \ref{prop:distance-quadratic}(i) we see that
 \begin{multline}\label{eq:Omega-+-1}
 \P \left( \left\{ \dist(A_{n,1}, H_{n,1}) \le c_{\ref{thm:s-min-graphs}}\rho^2 \sqrt{p} \vep\right\} \cap \Omega_K^0 \cap \Omega_+ \right)\\
 =\P\left( \left\{ \f{\left|\langle C_n^{-1} {\bm x}, {\bm y}\rangle-a_{11} \right|}{\sqrt{1+\left\|C_n^{-1}{\bm x}\right\|_2^2}} \le c_{\ref{thm:s-min-graphs}}\rho^2 \sqrt{p} \vep \right\} \cap \Omega_K^0 \cap \Omega_+\right),
 \end{multline}
 where by a slight abuse of notation we write $\Omega_+:=\{C_n \text{ is invertible}\}$, ${\bm x}^{\sf T}$, ${\bm y}$ are the first row and column of $A_n$, respectively, with the $(1,1)$-th entry $a_{11}$ removed. As $\|C_n - \E C_n\| \le \|A_n - \E A_n\|$, using Proposition \ref{prop:dist-bd}, setting $c_{\ref{thm:s-min-graphs}} \le c_{\ref{prop:dist-bd}}$, we see that the \abbr{RHS} of \eqref{eq:Omega-+-1} is bounded by
 \beq\label{eq:Omega-+}
 \vep^{1/5} + \f{C_{\ref{prop:dist-bd}} }{\sqrt[4]{np}}.
 \eeq
 This yields the desired bound on the event that $\dist(A_{n,1}, H_{n,1})$ is small on the event $\Omega_+$. It remains to find the same on the event $\Omega_+^c$. Turning to do this task, we apply Proposition \ref{prop:distance-quadratic}(i) to obtain that
 \begin{multline}\label{eq:Omega--1}
 \P \Big( \Big\{\dist({A}_{n,1},H_{n,1}) \le c_{\ref{thm:s-min-graphs}}\rho^2 \sqrt{p} \vep\Big\} \cap \Omega_K^0 \cap \Omega_+^c \Big) \\
 \le \P\Big(\Big\{\exists v \in {\rm Ker}(C_n) \cap S^{n-2}: \ |\pr{\bm y}{v}| \le c_{\ref{thm:s-min-graphs}} \vep \rho^2 \sqrt{p}\Big\}\cap  \Omega_K^0 \Big)
 \end{multline}
 As $A_n$ satisfies Assumption \ref{ass:matrix-entry}, so does $C_n$. Therefore using Propositions \ref{l: sparse vectors-2}, \ref{p: spread vectors}, and \ref{p: spread vectors-1} again we obtain that
 \beq\label{eq:Omega--2}
 \P\left(\{\exists v \in S^{n-2} \cap {\rm Ker}(C_n): v \in V^c\} \cap \Omega_K^0\right) \le 3 n^{-\bar c_{\ref{l: sparse vectors-2}}},
 \eeq
 where we recall the definition of $V$ from \eqref{eq:def-V}. Note that to obtain \eqref{eq:Omega--2} we need to apply the propositions for a $(n-1) \times (n-1)$ matrix. This only slightly worsens the constants.

 Next, by Assumption \ref{ass:matrix-entry} the matrix $C_n$ and the random vector ${\bm y}$ are independent and the coordinates of ${\bm y}$ are i.i.d.~$\dBer(p)$. Moreover, if $v \in V$ then from the definition of $V$ it follows that $v$ is neither dominated nor compressible. Hence, conditioning on such a realization of $v \in S^{n-2} \cap {\rm Ker}(C_n)$, applying Lemma \ref{lem:bound-levy}, and finally taking an average over such choices of $v$ we obtain
 \beq\label{eq:Omega--3}
 \P\Big(\Big\{\exists v \in {\rm Ker}(C_n) \cap V: \ |\pr{\bm y}{v}| \le c_{\ref{thm:s-min-graphs}} \vep \rho^2 \sqrt{p}\Big\}\cap  \Omega_K^0 \Big) \le C_{\ref{lem:bound-levy}} \left( \vep + \f{2K}{c_{\ref{p: spread vectors-1}}\sqrt{np}}\right).
 \eeq
 To complete the proof for the adjacency matrices of the directed and undirected Erd\H{o}s-R\'{e}nyi graphs we simply take $c_0=\f12$ and $C_0=1$ in Theorem \ref{thm:s-max-general} and then set $K= C_{\ref{thm:s-max-general}}$. Now combining \eqref{eq:V^c-inf-bd}-\eqref{eq:Omega--3}, applying Theorem \ref{thm:s-max-general}, and substituting the bounds in \eqref{eq:inf_split} we arrive at \eqref{eq:s-min-bd} when $A_n=\Adj(\Graph(n,p_n))$ or   $\Adj(\dGraph(n,p_n))$.
\end{proof}

We end this section with the proof of Theorem \ref{thm:bernoulli}. 

\begin{proof}[Proof of Theorem \ref{thm:bernoulli}]
Let $A_n$ be the matrix with i.i.d.~$\dBer(p)$ entries. Recall from the above that it suffices to derive the desired bound for \eqref{eq:dist-new1} for $j=1$. Note that 
 $\dist(A_{n,1}, H_{n,1}) \ge |\pr{A_{n,1}}{v}|$ for any $v \in {\rm Ker}(\wt A_n) \cap S^{n-1}$ where $\wt A_{n}$ is the $(n-1) \times n$ matrix whose rows are the columns $A_{n,2}, \ldots, A_{n,n}$. Since the entries of $A_n$ are independent, we apply Propositions \ref{l: sparse vectors-2} and \ref{p: spread vectors} for the matrix $\wt A_n$ (although these were proved for square matrices, they have a simple extension for $\wt A_n$; see also {Remark \ref{rmk:prop-non-square}}) to conclude that any $v \in {\rm Ker}(\wt A_n) \cap S^{n-1}$ must be in $\wt V$ with probability at least $1 - n^{-c}$, for some $c >0$, where 
 \[
 \wt V:= S^{n-1} \setminus \Big( \text{Comp}(C p^{-1}, \rho) \cup \text{Dom}(C p^{-1}, (c_{\ref{p: spread vectors-1}}K^{-1}) \Big),
 \] 
 and $C < \infty$ is some constant, to be specified below. Upon using Theorem \ref{thm:s-max-general} we observe that it remains to argue that 
\begin{equation}\label{eq:Omega-new1}
 \P\Big(\Big\{\exists v \in {\rm Ker}(\wt A_n) \cap \wt V: \ |\pr{A_{n,1}}{v}| \le c_{\ref{thm:s-min-graphs}} \vep \rho^2 \sqrt{p}\Big\}\cap  \Omega_K^0 \Big) \le \vep + n^{-c_\star},
 \end{equation}
for some $c_\star >0$. To this end, we borrow ideas from the proof of \cite[Theorem 1.1]{BR-invertibility}. Since $np \ge \log(1/p)$ we have that $np \ge \log n/2$. From \cite[Proposition 4.1]{BR-invertibility} it follows that for such choices of $p$, with probability at least $1- e^{-c'n}$, for some $c'>0$, we have that $D(v) \ge \exp(-c'' np)$, where $D(v)$ the least common denominator, as defined \cite[Definition 2.6]{BR-invertibility}, of the normal vector $v$ and $c''>0$ is some constant (We point out to the reader that \cite{BR-invertibility} considers the case $np \ge \ol C \log n$, for some large constant $\ol C < \infty$. However, one can check that \cite[Proposition 4.1]{BR-invertibility} holds for all $p$ such that $np \ge c_0 \log n$ for any $c_0>0$). Conditioning on a $v$ such that $D(v) \ge \exp(-c'' np)$ we now apply \cite[Proposition 4.2]{BR-invertibility} to deduce that \eqref{eq:Omega-new1} holds for such a vector $v$. 
 This concludes the proof of \eqref{eq:Omega-new1} and therefore the proof of the theorem is now complete.
\end{proof}

\section{Bound on the spectral norm}\label{sec:spectral-norm}
In this short section we prove Theorem \ref{thm:s-max-general} which yields the desired bound on $\|A_n - \E A_n\|$. 

\begin{proof}[Proof of Theorem \ref{thm:s-max-general}]
The proof consists of two parts. We will show that $\|A_n -\E A_n\|$ concentrates near its mean and then  find bounds on $\E\|A_n -\E A_n\|$. First let us derive the concentration of $\|A_n -\E A_n\|$. Since $a_{i,j}$ may depend on $a_{j,i}$ we split $A_n$ into its upper and lower triangular part (excluding the diagonal), denoted hereafter by $A_n^U$ and $A_n^L$, respectively, and work with them separately.

The function $\|A_n^U - \E A_n^U\|$ when viewed as a function from $\R^{{n(n-1)}/{2}}$ to $\R$ is a $1$-Lipschitz, quasi-convex function. So using Talagrand's inequality (see \cite[Theorem 7.12]{BLM}) we obtain that for any $t >0$,
\beq\label{eq:conc-T}
\P\left(\left|\|A_n^U -\E A_n^U\| -\M_n\right|\ge t\right) \le 4 \exp(-t^2/4),
\eeq
where $\M_n$ is the median of $\|A_n^U - \E A_n^U\|$. Using integration by parts from \eqref{eq:conc-T} it also follows that $| \E\|A_n^U - \E A_n^U\| -\M_n| \le C^*$ for some absolute constant $C^*$. Since $A_n$ satisfies Assumption \ref{ass:matrix-entry}, so does $A_n^{\sf T}$. Hence, proceeding similarly as above, we find that same holds for $A_n^L$. As the entries of $A_n$ are $\{0,1\}$-valued it follows that $\|A_n^D - \E A_n^D\| \le 1$, where $A_n^D$ is the diagonal part of $A_n$. Hence, using the triangle inequality and the condition $np \ge c_0 \log n$, we deduce that
\beq\label{eq:conc-T-2}
\P\left( \|A_n - \E A_n\| \le \E \|A_n^U - \E A_n^U\| + \E \|A_n^L - \E A_n^L\| +\wt{C}\sqrt{np}\right) \le \exp(-C_0\log n),
\eeq
for some large constant $\wt{C}$.

Now it remains to show that $\E\|A_n^\dagger -\E A_n^\dagger\| \le \bar{C}\sqrt{np}$ for $\dagger \in \{U,L\}$. To this end, let $A_n'$ be an independent copy of $A_n$ and $R_n$ be a $n \times n$ symmetric matrix consisting of independent Rademacher random variables. Since, the entries of $A_n-A_n'$ have a symmetric distribution,  applying Jensen's inequality we obtain that,
\beq\label{eq:op-norm-ineq}
\E\|A_n^U -\E A_n^U\| \le \E\|A_n^U -A_n'^U\| = \E\|D_n^U\odot R_n\|,
\eeq
where we denote $D_n:=A_n-A_n'$, and $D_n \odot R_n$ denotes the Hadamard product of $D_n$ and $R_n$. Next, let us denote $G_n$ to be a $n \times n$ symmetric matrix with independent standard Gaussian random variables and $|G_n|$ to be the matrix constructed from $G_n$ by taking absolute value of each of its entries. We write $\E_r$ and $\E_g$ to denote the expectations with respect to $R_n$ and $G_n$ respectively. Therefore, applying Jensen's inequality again
\beq\label{eq:norm-bound-r}
\E_r\|D_n^U \odot R_n\| = \sqrt{\f{\pi}{2}} \E_r \|D_n^U \odot R_n \odot \E |G_n|\| \le \sqrt{\f{\pi}{2}} \E_r \E_g \|D_n^U \odot R_n \odot  |G_n|\| =\sqrt{\f{\pi}{2}} \E_g\|D_n^U \odot G_n\|.
\eeq
This implies that it is enough to bound $\E\|D_n^U \odot G_n\|$. Using \cite[Theorem 1.1]{BH} we obtain that
\beq\label{eq:norm-bound-g}
\E_g\|D_n^U \odot G_n\| \le C\left[ \sigma + \sqrt{\log n}\right],
\eeq
where
\[
\sigma:= \max\left\{\max_i \sqrt{\sum_j\gd_{i,j}^2}, \max_j \sqrt{\sum_i\gd_{i,j}^2} \right\},
\]
$\gd_{i,j}$ is the $(i,j)$-th entry of $D_n$, and $C$ is an absolute constant. Using Chernoff bound and the union bound we note that there exists a constant $C'$ depending only $c_0$ (recall $p\ge c_0\f{\log n}{n}$), such that
\[
\P(\Omega) \ge 1 - n^{-2}, \quad \text{ where } \quad \Omega:= \{\sigma \le C' \sqrt{np}\}.
\]
Therefore fixing a realization of $D_n$ such that $\sigma \in \Omega$ from \eqref{eq:norm-bound-r}-\eqref{eq:norm-bound-g} we find
\[
\E_r\|D_n^U \odot R_n \| \le \f{\bar{C}}{2} \sqrt{np},
\]
for some constant $\bar{C}$, depending only on $c_0$. On the other hand noting that the entries of $D_n \odot R_n$  are $\{-1,0,1\}$ valued it is easily follows that $\|D_n \odot R_n\| \le n$. So
\[
\E \|D_n^U \odot R_n\| \le \E\left[\bI(\Omega) \E_r\|D_n^U \odot R_n\|\right] + n \P(\Omega^c) \le \bar{C}\sqrt{np}.
\]
Hence, from \eqref{eq:op-norm-ineq} we now have
\[
\E\|A_n^U -\E A_n^U\| \le \bar{C}\sqrt{np}.
\]
Same bound holds for $\E\|A_n^L -\E A_n^L\|$.  Therefore the proof now finishes from \eqref{eq:conc-T-2}.
\end{proof}

\appendix

\section{Structural properties of the adjacency matrices of sparse graphs}\label{sec:appendix}
In this section we prove that certain structural properties of $A_n$, as listed in Lemma \ref{lem: typical structure}, hold with high probability when $A_n$ satisfies Assumption \ref{ass:matrix-entry} with $p$ such that $np \ge \log (1/\bar C p)$, for some $\bar C \ge 1$. We also show that under the same assumption we have bounds on the number of light columns of $A_n$, namely we prove Lemma \ref{lem:light-col-card}.

First let us provide the proof of Lemma \ref{lem:light-col-card}.

\begin{proof}[Proof of Lemma \ref{lem:light-col-card}]
The proof is a simple application of Chernoff bound and Markov's inequality.

Since the entries of $A_n$ satisfies Assumption \ref{ass:matrix-entry}, using Stirling's approximation we note that
\begin{align}\label{eq:light-col-bd}
\P(\col_j(A_n) \text{ is light})   \le \sum_{\ell=0}^{\delta_0 np }\binom{n-1}{\ell} p^{\ell} (1-p)^{n-1-\ell} & \le 2\delta_0 np \left(\frac{e}{\d_0}\right)^{\d_0 np} \cdot \exp(-p(n-\d_0 np)) \notag\\
& \le \exp\left(-np \left[1-\delta_0 p - \d_0 \log \left(\frac{2e}{\d_0} \right) \right]\right),
\end{align}
where in the second inequality we have used the fact that $p \le 1/4$. Therefore, for $np \ge C \log n$, with $C$ large, using the union bound we find $\E[ |\cL(A_n)| ]<1/n$. Hence by Markov's inequality we deduce that
\[
\P(\cL(A_n) \ne \emptyset) = \P(|\cL(A_n)| \ge 1) \le \E[|\cL(A_n)|] \le 1/n.
\]
To prove the upper bound on the cardinality of $\cL(A_n)$ we note that the assumption $np \ge \log(1/\bar C p)$ implies that $np \ge (1-\delta) \log n$, for any $\delta >0$, for all large $n$. Therefore, using \eqref{eq:light-col-bd} and Markov's inequality, setting $\delta=\f19$, we find that for $np \le 2 \log n$,
\[
\P(|\cL(A_n)| \ge n^{\f13}) \le n^{-{\f13}} \E |\cL(A_n)| \le n^{\f23} \cdot n^{-\f89} \cdot n^{2 \delta_0 p +2 \d_0 \log \left(\frac{2e}{\d_0} \right)} \le n^{-\f19},
\]
for all large $n$, whenever $\delta_0$ is chosen sufficiently small. For $p$ such that $2 \log n  \le np \le  C_{\ref{lem:light-col-card}} \log n$ we note from \eqref{eq:light-col-bd} that
\[
\P(\col_j(A_n) \text{ is light}) \le \f{1}{n}, \qquad j \in [n].
\]
Therefore, an union bound followed by Markov's inequality yield the desired result.
\end{proof}

\begin{proof}[Proof of Lemma \ref{lem: typical structure}]
We will show that each of the six properties of the event $\Omega_{\ref{lem: typical structure}}$ hold with probability at least $1 - Cn^{-2\bar c_{\ref{lem: typical structure}}}$, for some constant $C >0$. Then, taking a union bound the desired conclusion would follow.

First let us start with the proof of \eqref{item: no heavy}. Since the inequality $np \ge \log(1/\bar C p)$ implies that $np \ge \log n/2$, for all large $n$, it follows from Chernoff bound that property \eqref{item: no heavy} of the event $\Omega_{\ref{lem: typical structure}}$ holds with probability at least $1-1/n$, for all large $n$. We omit the details.

Next let us prove that property \eqref{item: disjoint support} of $\Omega_{\ref{lem: typical structure}}$ holds with high probability. For $(i,j) \in \binom{[n]}{2}$ and $k \in [n]$ denote by $\Omega_{(i,j),k}$ the event that the columns $\col_i(A_n), \col_j(A_n)$ are light and $a_{k,i}, a_{k,j} \ne 0$. Note that the event that two light columns intersect is contained in the event $\cup_{i,j,k} \Omega_{(i,j),k}$. Therefore, we need to find bounds $\P(\Omega_{(i,j),k})$.
Since the entry $a_{i,j}$ may depend on $a_{j,i}$ we need to consider the cases $k \in [n]\setminus \{i,j\}$ and $k \in \{i,j\}$ separately.

 First let us fix $k \in [n]\setminus \{i,j\}$. We note that
 \[
 \Omega_{(i,j),k} \subset \left\{a_{k,i}=a_{k,j} =1, \, |\supp(\col_i(A_n))\setminus\{i,j\}|, |\supp(\col_j(A_n))\setminus\{i,j\}| \le \delta_0 np\right\}.
 \]
 Therefore, recalling that under Assumption \ref{ass:matrix-entry} the entries of the sub-matrix of $A_n$ indexed by $([n]\setminus\{i,j\}) \times \{i,j\}$ are i.i.d.~$\dBer(p)$ we obtain that
 \[
   \P (\Omega_{(i,j), k}) \le p^2 \exp \left( -2np \left[1-\d_0 p - \d_0 \log \left(\frac{2e}{\d_0} \right) \right] \right)=:q,
 \]
for all large $n$, where we have proceeded similarly as in  \eqref{eq:light-col-bd} to bound the probability of the event
 \[
 \left\{ |\supp(\col_i(A_n))\setminus\{i,j\}|, |\supp(\col_j(A_n))\setminus\{i,j\}| \le \delta_0 np \right\}.
 \]
 Since $np \ge \log (1/\bar C p)$ an application of the union bound shows that
 \begin{align}\label{eq:prop2-1}
  \P\left(\bigcup_{i \ne j \in [n], k \notin \{i,j\}}\Omega_{(i,j),k} \right) \le n \cdot \binom{n}{2} q
 & \le \frac{p^{-1}}{2} e^{-np} \cdot (np)^3 \cdot
  \exp \left( -np \left[1-2\d_0 p - 2\d_0 \log \left(\frac{2e}{\d_0} \right) \right] \right)\notag\\
 & \le \f{\bar C}{2} \cdot (np)^3 \cdot  \exp \left( -np \left[1-2\d_0 p - 2\d_0 \log \left(\frac{2e}{\d_0} \right) \right] \right) \le n^{-c},
 \end{align}
 for some absolute constant $c$ and all large $n$, where we use that $np \ge \log n/2$, which as already seen is a consequence of the assumption $np \ge \log (1/\bar C p)$.

 Next let us consider the case $k \in \{i,j\}$. Without loss of generality, let us assume that $k=i$. We see that
 \[
 \Omega_{(i,j),i} \subset \left\{a_{i,j} =1, \, |\supp(\col_i(A_n))\setminus\{i,j\}|, |\supp(\col_j(A_n))\setminus\{i,j\}| \le \delta_0 np\right\}.
 \]
 Hence proceeding same as above we deduce
 \begin{align}\label{eq:prop2-2}
  \P\left(\bigcup_{i \ne j \in [n], k \in \{i,j\}}\Omega_{(i,j),k} \right) & \le 2 \cdot \binom{n}{2}  \cdot p \cdot
  \exp \left( -2np \left[1-\d_0 p - \d_0 \log \left(\frac{2e}{\d_0} \right) \right] \right)\notag\\
 & \le p^{-1} e^{-np} \cdot (np)^2 \cdot \exp \left( -np \left[1-2\d_0 p - 2\d_0 \log \left(\frac{2e}{\d_0} \right) \right] \right) \le n^{-c}.
 \end{align}
 So combining the bounds of \eqref{eq:prop2-1}-\eqref{eq:prop2-2} we conclude that property \eqref{item: disjoint support} of $\Omega_{\ref{lem: typical structure}}$ holds with probability at least $1- n^{-2 \bar c_{\ref{lem: typical structure}}}$.

 Now let us prove that \eqref{item: bounded multiplicity} holds with high probability. We let $j \in [n]$, $I =(i_1 \etc i_{r_0}) \in\binom{[n]\setminus \{j \}}{r_0}$, and $k_1 \etc k_{r_0} \in [n]$, for some absolute constant $r_0$ to be determined during the course of the proof. Denote by $\Omega_{j,I,(k_1 \etc k_{r_0})}$ the event that all the columns indexed by $I$ are light, and for any $i_\ell \in I$, $k_\ell \in \supp(\col_{i_\ell}(A_n)) \cap  \supp(\col_j(A_n))$. Equipped with this notation we see that the event that there exists a column such that its support intersects with the supports of at least $r_0$ light columns is contained in the event $\cup_{j; I; k_\ell, \ell \in [r_0]} \Omega_{j,I, (k_1,k_2,\ldots,k_{r_0})}$.

Since all the columns indexed by $I$ are light, applying property \eqref{item: disjoint support} it follows that $\{k_\ell\}_{\ell=1}^{r_0}$ are distinct. Therefore, for matrices with independent entries \eqref{item: bounded multiplicity} follows upon bounding the probability of the events
 \[
 \left|\supp(\col_{i_\ell}(A_n))\setminus \{k_\ell'\}_{\ell'=1}^{r_0}\right| \le \delta_0 np, \quad \ell \in [r_0]
 \]
 and
 \[
 a_{k_\ell, j} = a_{k_\ell, i_\ell} =1, \qquad \ell \in [r_0],
 \]
 followed a union bound. Recall that under Assumption \ref{ass:matrix-entry} the entry $a_{i,j}$ may only depend on $a_{j,i}$ for $i,j \in [n]$. Therefore, to carry out this scheme for matrices satisfying Assumption \ref{ass:matrix-entry} we additionally need to show that the support of $\col_j(A_n)$ is almost disjoint from the set of light columns with high probability, so that we can omit the relevant diagonal block to extract a sub-matrix with jointly independent entries.

 To this end, we claim that
 \beq\label{eq:prop-2-pre}
 \P\left(\exists j \in [n]: |\supp(\col_j(A_n)) \cap \cL(A_n)| \ge 3\right) \le n^{-c'},
 \eeq
for some $c' >0$. To establish \eqref{eq:prop-2-pre} we fix $j \in [n]$ and note that
 \begin{align*}
& \left\{|\supp(\col_j(A_n)) \cap \cL(A_n)| \ge 3, |\supp(\col_j(A_n))| \le C_{\ref{lem: typical structure}}np  \right\} \\
\subset &  \big\{\exists k \text{ with } 2 \le k \le C_{\ref{lem: typical structure}}np, \text{ and } i_1, i_2, \ldots, i_k \in [n] \setminus\{j\} \text{ distinct } \text{such that } \\
& \qquad \qquad \qquad \qquad |\supp(\col_{i_\ell}(A_n))\setminus \{i_1,i_2\ldots, i_k,j\}| \le \delta_0 np, \ell=1,2,\ldots,k\big\}.
 \end{align*}
For ease of writing, let us denote
\[
q':= \exp\left(-np \left[1-\delta_0 p - C_{\ref{lem: typical structure}}p- \d_0 \log \left(\frac{2e}{\d_0} \right) \right]\right).
\]
By Assumption \ref{ass:matrix-entry} the entries $\{a_{i',j'}\}$ for $(i',j') \in \{i_\ell\}_{\ell=1}^k \times ([n]\setminus (\{i_\ell\}_{\ell=1}^k \cup\{j\})$ are jointly independent $\dBer(p)$ random variables.
Therefore
applying Stirling's approximation once more, and proceeding similarly as in \eqref{eq:light-col-bd} we find that
\begin{align*}
  & \P\left(|\supp(\col_j(A_n)) \cap \cL(A_n)| \ge 3, |\supp(\col_j(A_n))| \le C_{\ref{lem: typical structure}}np \Big| \col_j(A_n)\right) \\
 \le & \sum_{k=2}^{C_{\ref{lem: typical structure}}np} \binom{C_{\ref{lem: typical structure}}np}{k}q'^k  \le  \sum_{k \ge 2} \left(\f{e C_{\ref{lem: typical structure}}np}{k}\right)^k \cdot q'^k  \\
   &  \qquad \qquad \qquad \qquad  \le e^{-np} p^{-1}\cdot p \cdot \left({e C_{\ref{lem: typical structure}}np}\right)^2 \cdot \exp\left(- np \left[1-2\delta_0 p -  2 C_{\ref{lem: typical structure}}p- 2 \d_0 \log \left(\frac{2e}{\d_0} \right) \right]\right).
 \end{align*}
 Since by Lemma \ref{lem:light-col-card} we see that $\cL(A_n) =\emptyset$ with high probability when $p \ge \f{C_{\ref{lem:light-col-card}}\log n}{n}$. Without loss of generality, we therefore assume that $p \le \f{C_{\ref{lem:light-col-card}}\log n}{n}$. So, by the union bound over $j$, using the fact that $np \ge \log (1/\bar C p)$ and property \eqref{item: no heavy} of the event $\Omega_{\ref{lem: typical structure}}$ we have that, for all large $n$,
 \begin{align*}
  \P\left(\exists j \in [n]: |\supp(\col_j(A_n)) \cap \cL(A_n)| \ge 3\right)  &  \le  \bar C \cdot  \left({e C_{\ref{lem: typical structure}}np}\right)^3 \cdot \exp\left(- np (1-\delta)\right) +1/n\\
  & \le2 \exp(-np(1-2\delta)),
 \end{align*}
 for some $\delta >0$. This establishes the claim \eqref{eq:prop-2-pre}.

Equipped with \eqref{eq:prop-2-pre} we turn to proving \eqref{item: bounded multiplicity}. Using \eqref{eq:prop-2-pre} we see that excluding a set of probability at most $n^{-c'}$, for any $j,I,(k_1 \etc k_{r_0})$ such that
 $\Omega_{j,I,(k_1 \etc k_{r_0})}$ occurs, we can find $\ell_{1},\ldots, \ell_{{r_0-3}}$ with $k_{\ell_s} \in [n]\setminus (\cL(A_n) \cup \{j\}) \subset [n] \setminus (I \cup\{j\})$ for all $s=1,2,\ldots,r_0-3$.
 For such $k_{\ell_s}$, all events $|\supp(\col_{i_{\ell_s}}(A_n))\setminus (I \cup \{j\})| \le \delta_0 np$ and $a_{k_{\ell_s}, j} = a_{k_{\ell_s}, i_{\ell_s}}=1$ with $s=1,2,\ldots,r_0-3$ are independent.
  Denote for brevity
\[
\bar{q}:= \exp\left(-np \left[1-\delta_0 p - \d_0 \log \left(\frac{2e}{\d_0} \right) \right]\right).
\]
Note that under the assumption $np \ge \log(1/\bar C p)$ we have $\bar q \le \exp( - \log n/2)$ for all large $n$.
Hence, recalling Assumption \ref{ass:matrix-entry}, using \eqref{eq:prop-2-pre} and property \eqref{item: disjoint support} of $\Omega_{\ref{lem: typical structure}}$, and proceeding similarly as in \eqref{eq:light-col-bd} once again we see that
\begin{align}\label{eq:prop-3-1}
& \P\left(\bigcup_{\substack{j, k_1,\etc, k_{r_0} \in [n]\\ I \in \binom{[n]\setminus \{j\}}{r_0}}} \Omega_{j,I, (k_1,\etc, k_{r_0})} \right) \\
\le  &  \sum_{\substack{j, k_1,\etc, k_{r_0} \in [n]\\ I \in \binom{[n]\setminus \{j\}}{r_0}}} \prod_{s=1}^{r_0-3}\P\left(|\supp(\col_{i_{\ell_s}}(A_n))\setminus (I \cup \{j\})| \le \delta_0 np \right)  \cdot \P\left(a_{k_{\ell_s}, j} = a_{k_{\ell_s}, i_{\ell_s}}=1 \right)+ n^{-c_0} \notag\\
\le &   n^{r_0+1} \binom{n-1}{r_0} p^{2(r_0-3)} \cdot \bar q^{r_0-3} +n ^{-c'}  \le (np)^{2(r_0-3)} \cdot n^7 \cdot \bar q^{r_0-3} + n^{-c'}\le n^{-\bar c} + n^{-c_0}, \notag
\end{align}
for some $\bar c, c_0 >0$, where the last step follows upon choosing $r_0$ such that $r_0 - 3 > 15$. This completes the proof of property \eqref{item: bounded multiplicity}.

Next let us show that \eqref{item: normal column} holds with high probability. First we will prove that for any $j \in [n]$ such that $\col_j(A_n)$ is normal we have
\beq\label{eq:prop-4-not-fold}
 \left|\supp(\col_j(A_n)) \cap  \left( \bigcup_{i \in \cL(A_n)} \supp(\col_i(A_n)) \right) \right|
        \le \frac{\delta_0}{64} np,
\eeq
with high probability. Note that the difference between \eqref{eq:prop-4-not-fold} and property \eqref{item: normal column} of $\Omega_{\ref{lem: typical structure}}$ is that in \eqref{eq:prop-4-not-fold} it is claimed that for any $j \in [n]$ such that $\col_j(A_n)$ is normal its support does not have a large intersection with that of light columns. To establish property \eqref{item: normal column} we need to strengthen the above to deduce that one can replace the matrix $A_n$ by its folded version on the \abbr{LHS} of \eqref{eq:prop-4-not-fold} with the loss of factor of four in its \abbr{RHS}.

Turning to prove \eqref{eq:prop-4-not-fold}, we see that if \eqref{item: bounded multiplicity} holds then given any $j \in [n]$ there exists only $r_0$ light columns $\col_{i_1}(A_n), \etc, \col_{i_{r_0}}(A_n)$ such that their supports intersect that of $\col_j(A_n)$. Hence,
\begin{align}\label{eq:conn-3-to-4}
&\left\{ \exists j \in [n]\setminus \cL_n(A): \left|\supp(\col_j(A_n)) \cap \left(\bigcup_{i \in \cL(A_n)} \supp(\col_i(A_n))\right)  \right| \ge \f{\d_0}{64}np\right\}\cap \left\{ \eqref{item: bounded multiplicity} \text{ holds}\right\} \notag\\
& \qquad \qquad \qquad \subset \left\{\exists i \ne j \in [n]: \left|\supp(\col_j(A_n)) \cap  \supp(\col_i(A_n))  \right| \ge \f{\d_0}{64 r_0}np\right\}.
\end{align}
Since by \eqref{item: no heavy} we have that $|\supp(\col_j(A_n))| \le C_{\ref{lem: typical structure}} np$, using Stirling's approximation and a union bound we show that the event on the \abbr{RHS} of \eqref{eq:conn-3-to-4} holds with small probability.

Indeed, for $i \ne j \in [n]$, denoting
\[
\bar \Omega_{i,j}: = \left\{\left|\supp(\col_j(A_n)) \cap  \supp(\col_i(A_n))  \right| \ge \f{\d_0}{64 r_0}np\right\},
\]
and using the fact that property \eqref{item: no heavy} holds with high probability we deduce that
\begin{align}\label{eq:conn-3-to-4-1}
\P\left(\bigcup_{i \ne j \in [n]} \bar \Omega_{i,j} \right)  & \le \sum_{i \ne j \in [n]} \E\left[\P\left(\bar \Omega_{i,j} \cap \left\{|\supp(\col_j(A_n))| \le C_{\ref{lem: typical structure}}np \right\} \Big| \col_j(A_n)\right)\right] + n^{-1}\notag\\
& \le \binom{n}{2} \cdot \binom{C_{\ref{lem: typical structure}}np}{\f{\delta_0}{64r_0} np} p^{\f{\delta_0}{64r_0} np} + n^{-1} \le n^2 \cdot \left(\f{e C_{\ref{lem: typical structure}} 64 r_0 p}{\delta_0} \right)^{\f{\delta_0}{64r_0} np} + n^{-1} \le 2 n^{-1},
\end{align}
for all large $n$. Thus combining \eqref{eq:conn-3-to-4}-\eqref{eq:conn-3-to-4-1} and applying property \eqref{item: no heavy} of the event $\Omega_{\ref{lem: typical structure}}$ we  establish that \eqref{eq:prop-4-not-fold}  holds with probability at least $1- n^{-\wt c}$ for some $\wt c >0$.

As mentioned above, to show that property \eqref{item: normal column} holds with high probability we need to strengthen \eqref{eq:prop-4-not-fold}. To this end, recalling the definition of the folded matrix (see Definition \ref{dfn:folded}) we note that $k \in \supp(\col_i(\fold(A_n)) \cap \supp(\col_j(\fold(A_n))$ implies that
\[
k  \in \supp_{\gu}(\col_i(A_n)) \cap \supp_{\gv}(\col_j(A_n))
\]
for some $\gu, \gv \in \{1,2\}$, where for any $\ell \in [n]$.
\[
\supp_1(\col_\ell(A_n)) := \supp(\col_\ell(A_n)) \cap [\gn],
\]
\[
\supp_2(\col_\ell(A_n)) := (\supp(\col_\ell(A_n)) \cap [\gn+1, 2\gn]) - \gn,
\]
and for any set $S \subset [n]$ and $k \in \Z$ we denote $S+k :=\{x+k: x \in S\}$. Using the observation we see that it suffices to show that
\beq\label{eq:prop-4-not-fold-1}
 \left|\supp_\gu(\col_j(A_n)) \cap  \left( \bigcup_{i \in \cL(A_n)} \supp_\gv(\col_i(A_n)) \right) \right|
        \le \frac{\delta_0}{64} np,
\eeq
with high probability, for all $\gu, \gv \in \{1,2\}$. If $\gu=\gv$ then \eqref{eq:prop-4-not-fold-1} is an immediate consequence of \eqref{eq:prop-4-not-fold}. It remains to prove \eqref{eq:prop-4-not-fold-1} for $\gu \ne \gv$. Let us consider the case $\gu=1$ and $\gv=2$. From \eqref{eq:prop-2-pre} we have
 \beq
 \P\left(\exists j \in [n]: |\supp_1(\col_j(A_n)) \cap \cL(A_n) | \ge 3\right) \le n^{-c'}. \notag
 \eeq
 Therefore, proceeding similarly as in the steps leading to \eqref{eq:prop-3-1} we deduce that, with the desired high probability, for any $j \in [n]$, such that $\col_j(A_n)$ is a normal column, there are at most $r_0$ light columns $\{\col_{i_\ell}(A_n)\}_{\ell=1}^{r_0}$ so that $\supp_1(\col_j(A_n)) \cap \supp_2(\col_{i_\ell}(A_n)) \ne \varnothing$. Now arguing similarly as in the proof of \eqref{eq:prop-4-not-fold} we derive \eqref{eq:prop-4-not-fold-1} for $\gu=1$ and $\gv=2$. The proof of the other case is similar and hence is omitted.

Next we show that \eqref{item: extension} holds with high probability. We first fix an $I \subset [n]$ with $2 \le |I| \le c_{\ref{lem: typical structure}} p^{-1}$ and derive that \eqref{item: extension} holds with certain probability for each such choice of $I$ and then take an union over $I$.

Since the entry $a_{i,j}$ may depend on $a_{j,i}$, for $i \ne j$, to derive that \eqref{item: extension} holds with the desired probability we need to split it into two cases. Namely, the off-diagonal and the diagonal blocks require separate arguments. First we consider the off-diagonal block.

To this end, define the random variables
\[
\eta_i:=\max(|\{j\in I:\,\ga_{i,j}\ne 0\}|-1,0), \quad\quad i\in [\gn]\setminus \bar I,
\]
where we recall $\bar I:=\bar I(I):=\{j \in [\gn]: j \in I \text{ or } j +\gn \in I\} \subset [\gn]$, $\gn:=\lfloor n/2 \rfloor$, and $\ga_{i,j}$ denotes the $(i,j)$-th entry of $\fold(A_n)$. Observe that
\[\left|\bigcup_{j\in I}\left(\supp(\col_j(\fold(A_n)))\setminus \bar I\right)\right|=
\sum_{j\in I}\left|\supp(\col_j(\fold(A_n)))\setminus \bar I\right|-\sum_{i \in [\gn]\setminus \bar I} \eta_i.
\]
To prove \eqref{item: extension} we need to show that $\sum \eta_i$ cannot be too large with large probability. To show the latter we use the standard Laplace transform method. 

Note that
\[
\ga_{i,j} = \xi_{i,j} \cdot \delta_{i,j}, \quad i \in [\gn]\setminus \bar I, j \in I,
\]
where $\{\xi_{i,j}\}$ are i.i.d.~Rademacher random variables, $\delta_{i,j}$ are  i.i.d.~$\dBer(\gp)$ random variables, and $\gp:=2p(1-p)$. Therefore,
\[
\P\{\eta_i=\ell\}\leq {|I|\choose \ell+1}\gp^{\ell+1}, \quad \ell\in\N.
\]
Thus, for any $\lambda>0$ such that $e^\lambda \gp|I|\leq 1$, we have
\[
\E \left(e^{\lambda \eta_i}\right)
\leq 1+\sum_{\ell=1}^\infty \big(e^\lambda\big)^{\ell}\,\gp^{\ell+1}|I|^{\ell+1}((\ell+1)!)^{-1}\leq 1+e \gp|I|,
\]
and hence
\[
\P\left\{\sum_{i \in [\gn]\setminus \bar I} \eta_i \geq t\right\}\leq \frac{\big(1+e\gp|I|\big)^{|[\gn]\setminus \bar I|}}{\exp(\lambda t)},\quad\quad t>0.
\]
In particular, taking $t:=\f{\d_0}{32} n p|I|$ and $\lambda:=\log\frac{1}{\gp|I|}$, we get
\begin{align}\label{eq:cB-1-bd}
\P\left\{\sum_{i\in [\gn]\setminus \bar I} \eta_i\geq \f{\d_0}{32} np |I|\right\}
\leq \exp\left(e\gp\gn|I|-\lambda \f{\d_0}{32} np |I|\right)& \leq \exp\left(-\lambda \f{\d_0}{64} np |I|\right) \notag \\
& \le \exp\left( - \log\left(\f{1}{2 c_{\ref{lem: typical structure}}}\right) \cdot \f{\d_0}{64}np |I|\right) \le n^{-2|I|},
\end{align}
where the second and the third inequalities follow from recalling that $ p |I| \le c_{\ref{lem: typical structure}}$ for some sufficiently small constant $c_{\ref{lem: typical structure}}$, depending only on $\delta_0$, and the last inequality follows from our assumption that $np \ge \log n/2$ and shrinking $c_{\ref{lem: typical structure}}$ even further, if necessary.

To complete the proof of the fact that \eqref{item: extension} holds with high probability, we show that
\beq\label{eq:cB-2-bd}
\P\left( \sum_{j\in I}\left|\supp(\col_j(\fold(A_n)))\cap \bar I\right| \ge \f{\d_0}{32} np |I| \right) \le 2 n^{-2|I|}.
\eeq
Now the proof finishes from \eqref{eq:cB-1-bd}-\eqref{eq:cB-2-bd} by first taking a union over $I \in \binom{[n]}{k}$ followed by a union over $k=2,3,\ldots, c_{\ref{lem: typical structure}} p^{-1}$. We omit the details.

Turning to prove \eqref{eq:cB-2-bd}, we denote \(
\hat I(I):=\hat I:= \cup_{i \in \bar I} \{ i , \gn+i\}
\).
As the entries of $A_n$ are $\{0,1\}$-valued,
we see that
\[
\supp (\col_j(\fold(A_n))) \cap \bar I \subset \supp(\col_j(A_n)) \cap \hat I.
\]
Moreover, $I \subset \hat I$. Therefore, it is enough to show that
\beq\label{eq:cB-2-bd-1}
\P \left( \sum_{j\in \hat I}\left|\supp(\col_j(A_n))\cap \hat I\right| \ge \f{\d_0}{32} np |I|\right) \le 2 n^{-2|I|}.
\eeq
Since $A_n$ satisfies Assumption \ref{ass:matrix-entry} we have that the upper triangular part of the sub-matrix of $A_n$ induced by the rows and columns indexed by $\hat I$ consists of independent~$\{0,1\}$-valued random variables stochastically dominated by i.i.d.~$\dBer(p)$ variables. So does the lower triangular part of that sub-matrix.

For ease of writing let us write
\[
\sX_L:= \sum_{i \ge j \in \hat I} a_{i,j}  \quad \text{ and } \quad \sX_U:=\sum_{i \le j \in \hat I} a_{i,j}
\]
and note $\sX_U$ and $\sX_L$ has the same law. Thus 
to establish \eqref{eq:cB-2-bd-1} it suffices to show that
\beq\label{eq:cB-2-bd-2}
\P(\sX_U \ge \f{\delta_0}{64} np |I|) \le n^{-2|I|}.
\eeq
The above is obtained by using the Laplace transform method as above. Indeed, we note that
\[
\P(\sX_U = \ell) \le \binom{|\hat I|^2}{\ell} p^\ell, \quad \ell \in \N \cup \{0\}
\]
and therefore
\[
\E \left[ \exp\left(\lambda X_U\right)\right] \le \exp \left( e^\lambda p |\hat I|^2 \right) \le \exp (4 |I|),
\]
where $\lambda= \log \f{1}{p |I|}$ and we have used the fact that $|\hat I| \le 2 |I|$. Hence, upon using Markov's inequality and proceeding similarly as in \eqref{eq:cB-1-bd} we deduce \eqref{eq:cB-2-bd-2}. It completes the proof of \eqref{eq:cB-2-bd-1}.

Now it remains to prove that property \eqref{item:diff} holds with high probability.  Recalling the definition of the folded matrix again we note that $|\supp(\col_j(\fold(A_n)))| \le |\supp(\col_j(A_n))|$. To show that the cardinality of the support of $\col_j(\fold(A_n))$ is not too small compared to its unfolded version we observe that if $k \in \supp(\col_j(A_n))$ but $k \notin \supp(\col_j(\fold(A_n)))$ then we must have that $a_{k,j}=a_{k,\gn+j}=1$. Using estimates on the binomial probability and Chernoff bound we show that number of such $k$ is small.

To carry out the above heuristic, we fix $j \in [n]$ and since the entries of $A_n$ are $\{0,1\}$ valued we note that
\[
|\supp(\col_j(\fold(A_n))) | = \sum_{i \in [\gn]} \left[a_{i,j} \cdot (1-a_{i+\gn, j}) + a_{i+\gn,j} \cdot (1-a_{i, j})\right]
\]
Further observe that
\begin{align*}
\left|\supp(\col_j(A_n))\cap [\gn]\right| = \sum_{i=1}^{\gn} a_{i,j} = \sum_{i=1}^{\gn} a_{i,j} \cdot a_{i+\gn, j} + \sum_{i=1}^\gn a_{i,j} \cdot (1-a_{i+\gn, j})
\end{align*}
and
\[
\left|\supp(\col_j(A_n))\cap ([2\gn]\setminus [\gn])\right| = \sum_{i=\gn+1}^{2\gn} a_{i,j} = \sum_{i=1}^{\gn} a_{i,j} \cdot a_{i+\gn, j} + \sum_{i=1}^\gn a_{i+\gn,j} \cdot (1-a_{i, j}).
\]
Therefore,
\[
\left| \left|\supp(\col_j(A_n))\right| - \left|\supp(\col_j(\fold(A_n))) \right| \right| \le 2 \sum_{i=1}^{\gn} a_{i,j} \cdot a_{i+\gn, j} +1.
\]
Denoting
\[
\Delta_j:= \sum_{i=1}^{\gn} a_{i,j} \cdot a_{i+\gn, j},
\]
we see that $\Delta_j$ is stochastically dominated by $\dBin(\gn, p^2)$. To finish the proof we need to find bounds on $\Delta_j$.

First let us consider the case $p \le n^{-5/12}$. For any $k_0 \in \N$, sufficiently large, we see that
\beq\label{eq:delta_j-1}
\P(\Delta_j \ge k_0) \le \binom{\gn}{k_0} p^{2k_0} \le (np^2)^{k_0} \le n^{-k_0/6} \le n^{-2}.
\eeq
For $ n^{-5/12} \le p \le c$, for some small $c >0$ depending on $\delta_0$, we use Chernoff bound to deduce that
\beq\label{eq:delta_j-2}
\P\left(\Delta_j \ge \f{\d_0}{16}np\right) \le \P\left(  \Delta_j \ge 2 p^{-1/2} \cdot \gn p^2  \right) \le \exp\left( -\f13 p^{-1/2} \cdot \gn p^2  \right) \le \exp\left(-\f19 n^{3/8} \right).
\eeq
Combining \eqref{eq:delta_j-1}-\eqref{eq:delta_j-2} and taking an union over $j \in [n]$ we show that property \eqref{item:diff} holds with high probability. This completes the proof of the lemma.
 \end{proof}

 \section{Proof of invertibility over sparse vectors with a large spread component}\label{sec:invert-large-spread}

 In this section we prove Proposition \ref{p: spread vectors-1}. As already mentioned in Section \ref{sec:sparse-3} the proof is similar to that of Proposition \ref{p: spread vectors}. There are two key differences. Since our goal is to find a uniform bound on $\|A_n x \|_2$ for $x$'s with a large spread component, unlike in the proof of Proposition \ref{p: spread vectors-1}, we use  Lemma \ref{lem:bound-levy} to estimate the small ball probability. Moreover, as noted earlier, Assumption \ref{ass:matrix-entry} allows some dependencies among its entries. Therefore, to tensorize the small ball probability we need to extract a sub-matrix of $A_n$ with jointly independent entries such that the coordinates of $x$ corresponding to the columns of this chosen sub-matrix form a vector with a large spread component and a sufficiently large norm. Below we make this idea precise.

 \begin{proof}[Proof of Proposition \ref{p: spread vectors-1}]
First, let us show that \eqref{eq:spread-prop-pre-bd} implies \eqref{eq:spread-prop-bd}. To this end, we begin by noting that if $c_{\ref{p: spread vectors-1}} < \f12$ then for any $x \in {\rm Dom}(c_0^*n, c_{\ref{p: spread vectors-1}}K^{-1})$ we have that $\|x_{[M_0+1: c_0^*n]}\|_2 \ge \|x_{[c_0^*n+1:n]}\|_2$ (see also \eqref{eq:r-v-relation}). Hence, for $x \notin V_{M_0}$ we obtain that $\|x_{[M_0+1: c_0^*n]}\|_2 \ge \rho/\sqrt{2}$. Therefore, \eqref{eq:spread-prop-pre-bd} implies that
\begin{equation}
\P\left(\left\{\exists x \in V_{c_0^*,{c}_{\ref{p: spread vectors-1}}}\setminus V_{M_0}: \norm{(A_n-p {\bm J}_n ) x -y}_2 \le 2\wt{c}_{\ref{p: spread vectors-1}}\rho \sqrt{np}\right\} \cap \Omega_K^0\right) \le \exp(-2\bar{c}_{\ref{p: spread vectors-1}}n), \notag
\end{equation}
where we recall the definition of $\Omega_K^0$ from \eqref{eq:Omega-K-0}.
Hence, proceeding as in the steps leading to \eqref{eq:moderate-pre-prop-2} we deduce \eqref{eq:spread-prop-bd} upon assuming \eqref{eq:spread-prop-pre-bd}.

So, to complete the proof of the proposition it remains to establish \eqref{eq:spread-prop-pre-bd}. To prove it, we fix $x \notin V_{M_0}$. Then
\[
\|x_{[M_0+1:n]}\|_2 \ge \rho \quad \text{ and } \quad \f{\|x_{[M_0+1:n]}\|_\infty}{\|x_{[M_0+1:n]}\|_2} \le \f{K}{c_{\ref{p: spread vectors}}}\cdot \sqrt{\f{\log n}{n \sqrt{\log \log n}}}.
\]
Fixing $\bar c_0 \in (c_0^*,1)$, as $M_0 \le \f{1-\bar c_0}{2}n$ for all large $n$, recalling the fact that the non-zero entries of $x_{[m_1: m_2]}$, for $m_1 <m_2$, are the coordinates of $x$ that  take places from $m_1$ to $m_2$ in the non-increasing arrangement according their absolute values, we note that
\beq\label{eq:small-comp-part}
\|x_{[M_0+1:n]}\|_2^2 = \|x_{[M_0+1:(1-\bar c_0)n]}\|_2^2 + \|x_{[(1-\bar c_0)n+1:n]}\|_2^2 \le \f{1+\bar c_0}{1- \bar c_0} \cdot \|x_{[M_0+1:(1-\bar c_0)n]}\|_2^2.
\eeq
Therefore
\[
\|x_{[M_0+1:(1-\bar c_0)n]}\|_2 \ge \rho \cdot \sqrt{\f{1-\bar c_0}{1+\bar c_0}} \quad \text{ and } \quad \f{\|x_{[M_0+1:(1-\bar c_0)n]}\|_\infty}{\|x_{[M_0+1:(1-\bar c_0)n]}\|_2} \le \f{K}{c_{\ref{p: spread vectors}}} \cdot \sqrt{\f{1+\bar c_0}{1-\bar c_0}}\cdot \sqrt{\f{\log n}{n\sqrt{\log \log n}}}.
\]
Note that this shows $x_{[M_0+1:(1-\bar c_0)n]}$ has a large spread part and a large norm. Denoting $\cI :=\cI(x):= \supp(x_{[M_0+1:(1-\bar c_0)n]})$ we note that Assumption \ref{ass:matrix-entry} implies that the entries $\{a_{i,j}\}_{j \in \cI, i \notin \cI}$ are i.i.d.~$\dBer(p)$. So, now we can carry out the scheme that was outlined above by using the joint independence of $\{a_{i,j}\}_{j \in \cI, i \notin \cI}$.

Indeed, using Lemma \ref{lem:bound-levy} we find that for any $i \notin \cI$, $y \in \R^n$, and $\vep_0>0$ we have
\begin{multline}\label{eq:levy-bd-1}
 \P\left( \left|((A_n - p {\bm J}_n) x)_i -y_i\right| \le p^{1/2} (1-p)^{1/2} \|x_{[M_0+1:(1-\bar c_0)n]}\|_2 \cdot \bar c_0^{-1/2}\vep_0 \right)  \\
\le   \cL \left( \langle {\bm a}_i, x\rangle, p^{1/2} (1-p)^{1/2} \|x_{[M_0+1:(1-\bar c_0)n]}\|_2 \cdot \bar c_0^{-1/2}\vep_0 \right) \\
\le C_{\ref{lem:bound-levy}} \left( \f{\vep_0}{\sqrt{\bar c_0}} + \f{2K}{c_{\ref{p: spread vectors}}} \cdot \sqrt{\f{1+\bar c_0}{1-\bar c_0}}\cdot \sqrt{\f{\log n}{np\sqrt{\log \log n}}} \right) \le 2 C_{\ref{lem:bound-levy}} \f{\vep_0}{\sqrt{\bar c_0}},
\end{multline}
for all sufficiently large $n$ (depending only on $\vep_0$), where ${\bm a}_i$ is the $i$-th row of $A_n$ and we have used the fact that $np \ge c_1 \log n$ for some $c_1 >0$.
We will choose $\vep_0$ as a small constant  during the course of the proof.

Since the entries $\{a_{i,j}\}_{j \in \cI, i \notin \cI}$ are i.i.d.~$\dBer(p)$, we apply a standard tensorization argument, for example \cite[Lemma 5.4]{R}, to deduce from \eqref{eq:levy-bd-1} that for any $x \notin V_{M_0}$
\begin{multline}\label{eq:levy-tensorize-1}
\P\left(\|(A_n - p {\bm J}_n) x - y\|_2 \le \sqrt{np(1-p)}  \|x_{[M_0+1:(1-\bar c_0)n]}\|_2  \vep_0\right)  \\
 \le \P\left( \sum_{i \notin \cI}\left|((A_n - p {\bm J}_n) x)_i -y_i\right|^2 \le p(1-p)  \|x_{[M_0+1:(1-\bar c_0)n]}\|_2^2 \vep_0^2 \bar c_0^{-1}\cdot |\cI^c|\right) \le \left(C_0 \cdot \vep_0\right)^{\bar c_0 n},
\end{multline}
for some constant $C_0$, depending only on $\bar c_0$, where the last two steps follow from the fact that $|\cI^c| \ge \bar c_0 n$ and upon choosing $\vep_0$ such that $ C_0 \cdot \vep _0 \le \f12$.

To complete the proof we use an $\vep$-net similar to the proof of Proposition \ref{p: spread vectors}. First, setting
\beq\label{eq:vep-tau-spread}
\vep=\f{\rho}{2}\tau= \f{\vep_0 \rho}{448 K} \cdot \sqrt{\f{1-\bar c_0}{1+\bar c_0}},
\eeq
and using  Fact \ref{fact:net-bd} we obtain a net $\wt \cM $ in $V_{c_0^*,{c}_{\ref{p: spread vectors-1}}} \setminus V_{M_0}$ with
\[
|\wt \cM| \le \bar C^n {n \choose M_0}{n \choose c_0^* n}  \left(\f{1}{\vep_0}\right)^{c_0^*n+1} \left(\f{1}{\rho}\right)^{M_0+1} \le \bar C^n \left(\f{en}{M_0}\right)^{M_0} \left(\f{e}{c_0^*}\right)^{c_0^* n} \left(\f{1}{\vep_0}\right)^{c_0^*n+1}  \left(\f{1}{\rho}\right)^{M_0+1},
\]
for some $\bar C$, depending only on $\bar c_0$ and $c_0^*$. Recalling that $M_0 = \f{n \sqrt{\log \log n}}{\log n}$ and the definition of $\rho$ we observe that
\[
\left(\f{n}{M_0\rho^2}\right)^{M_0} = \exp(o(n)),
\]
for $p \in(0,1/2]$ satisfying $np \ge c_1 \log n$. Therefore, we further have that
\beq\label{eq:wt-cM-bd}
|\wt \cM| \le C_\star^n \cdot \left(\f{1}{\vep_0}\right)^{c_0^*n+1},
\eeq
for some other constant $C_\star$, depending only on $c_0^*$ and $\bar c_0$. Next proceeding as in the steps leading to \eqref{eq:net-triangle} we obtain that for any $x \in V_{c_0^*,{c}_{\ref{p: spread vectors-1}}} \setminus V_{M_0}$ there exists $\bar x \in \wt \cM$ such that for any $y \in \R^n$
\begin{multline*}
\|(A_n - p {\bm J}_n) \bar{x} - y\|_2   \le  \|(A_n- p {\bm J}_n) x - y\|_2 + 4K \sqrt{np}\cdot  \vep + 2K \sqrt{np} \cdot \tau \cdot \|v_{\bar x}\|_2 + 12 c_{\ref{p: spread vectors-1}} \sqrt{np} \cdot  \norm{v_{\bar{x}}}_2.
\end{multline*}
Since $\|\bar{x}_{[M_0+1:c_0^*n]}\|_2 = \|v_{\bar x}\|_2 \ge \rho/\sqrt{2}$, using \eqref{eq:vep-tau-spread} and setting
\beq\label{eq:vep-tau-spread-1}
c_{\ref{p: spread vectors-1}} \le  \f{\vep_0}{56} \cdot \sqrt{\f{1-\bar c_0}{1+\bar c_0}},
\eeq
we deduce from above that any $x \in V_{c_0^*,{c}_{\ref{p: spread vectors-1}}} \setminus V_{M_0}$ there exists $\bar x \in \wt \cM$ such that for any $y \in \R^n$
\[
\|(A_n - p {\bm J}_n) \bar{x} - y\|_2   \le  \|(A_n- p {\bm J}_n) x - y\|_2 +\f{\vep_0}{7} \cdot \sqrt{\f{1-\bar c_0}{1+\bar c_0}} \|\bar{x}_{[M_0+1:c_0^*n]}\|_2 \cdot \sqrt{np}.
\]
Furthermore, by our construction of the net $\wt \cM$,
\[
\|{x}_{[M_0+1:c_0^*n]}\|_2 \le \|\bar{x}_{[M_0+1:c_0^*n]}\|_2+ \vep \le \left(1 + \f{\vep_0}{224 K}\right) \cdot \|\bar{x}_{[M_0+1:c_0^*n]}\|_2.
\]
Therefore, upon assuming $p \le \f14$ and recalling \eqref{eq:small-comp-part}, this further yields that
\begin{multline}\label{eq:inf-to-union-spread}
\P\left(\exists {x} \in V_{c_0^*,{c}_{\ref{p: spread vectors-1}}} \setminus V_{M_0}: \|(A_n- p {\bm J}_n) {x} -y \|_2 \le \f{\vep_0}{4}\|{x}_{[M_0+1:c_0^*n]}\|_2 \cdot \sqrt{\f{1-\bar c_0}{1+\bar c_0}}\cdot \sqrt{np} \right) \\
\le   \P\left(\exists \bar{x} \in \wt \cM: \|(A_n- p{\bm J}_n) \bar{x} -y \|_2 \le  \sqrt{np(1-p)}  \|\bar x_{[M_0+1:(1-\bar c_0)n]}\|_2 \vep_0 \right)\\
 \le |\wt \cM| \cdot \left(C_0 \cdot \vep_0\right)^{\bar c_0 n} \le C_0^{\bar c_0 n} C_\star^n \vep_0^{-1}\vep_0^{(\bar c_0 - c_0^*)n} \le \vep_0^{\f{\bar c_0- c_0^*}{2} n},
\end{multline}
where the second last step follows from \eqref{eq:wt-cM-bd} and the last step follows upon using the fact that $\bar c_0 > c_0^*$ and choosing $\vep_0$ sufficiently small. This yields \eqref{eq:spread-prop-pre-bd} and hence the proof of the proposition is complete.
\end{proof}


\end{document}